%% file: main.tex
\documentclass[oneside, hidelinks, 12pt, a4paper,reqno]{amsart}
\input{preamble.tex}
\begin{document}	
	\title{Higher local systems and the categorified monodromy equivalence}
	\author{James Pascaleff}
	\address[James Pascaleff]{University of Illinois\\West Green Street, 1409\\61801, Urbana, IL, United States}\email{\href{mailto:jpascale@illinois.edu}{jpascale@illinois.edu}}
	\author{Emanuele Pavia}
	\address[Emanuele Pavia]{SISSA\\Via Bonomea 265\\34136 Trieste, TS, Italy}	\email{\href{mailto:epavia@sissa.it}{epavia@sissa.it}}
	\author{Nicolò Sibilla}
	\address[Nicolò Sibilla]{SISSA\\Via Bonomea 265\\34136 Trieste, TS, Italy}
	\email{\href{mailto:nsibilla@sissa.it}{nsibilla@sissa.it}}
\maketitle

\begin{abstract}
We study local systems of $(\infinity,n)$-categories on spaces.  
We prove that categorical local systems are captured by (higher) monodromy data: in particular,  if $X$ is $(n+1)$-connected, then 
local systems of $(\infinity,n)$-categories over $X$  
 can be described as 
$\mathbb{E}_{n+1}$-modules over the iterated loop space $\Omega_{n+1}X$. This generalizes the classical  \emph{monodromy equivalence} presenting ordinary local systems as modules over the based loop spaces. Along the way we revisit from the perspective of $\infty$-categories Teleman's influential theory of topological group actions on categories, and we extend it to topological actions on $(\infty,n)$-categories.  Finally, we show that the group of  invertible objects in the category of local systems of $(\infinity,n)$-categories over an $n$-connected space $X$ is isomorphic to the group of characters of $\pi_n(X)$. This should be thought of as a topological analogue of the higher  Brauer group of the space $X$. We conclude the paper with applications of the theory of categorical local systems to the fiberwise Fukaya category of symplectic fibrations.
 \end{abstract}
\tableofcontents
{\section*{Introduction}
\addtocontents{toc}{\protect\setcounter{tocdepth}{0}}
\renewcommand\thesubsection{I.\arabic{subsection}}
\renewcommand\theequation{\thesubsection.\arabic{equation}}
\setcounter{equation}{0}
\input{introduction.tex}}
 
\subsection*{Notations and conventions}
\begin{itemize}
\item We will use throughout the language of $(\infinity,1)$-categories and higher homotopical algebra, as developed in \cite{htt,ha}, from which we borrow most of the notations and conventions.
\item Since our work heavily relies on intrinsically derived and homotopical concepts, we shall simply write ``limits", ``colimits", ``tensor product", suppressing adjectives such as ``homotopy" or ``derived" in our notations. Similarly, we shall simply write ``categories" instead of ``$\infinitone$-categories", and ``$n$-categories" instead of ``$(\infinity,n)$-categories". 
\item We will work with \textit{local systems} and \textit{sheaves} of categories, and it will be important pay attention to size issues. We fix a sequence of nested universes $\scrU \in\scrV\in\mathscr{W}\in\ldots$. We shall say that a category $\scrC$ is \textit{small} if it is $\scrU$-small, that $\scrC$ is \textit{large} if it is $\scrV$-small without being $\scrU$-small, that $\scrC$ is \textit{very large} if it is $\mathscr{W}$-small without being $\scrV$-small, and that $\scrC$ is \textit{huge} if it is not even $\mathscr{W}$-small. When dealing with categories of (possibly decorated) categories, we shall adopt the following notations in order to distinguish the size: large categories of categories will be denoted with a normal font; very large categories of categories will be denoted with $\smash{\widehat{(-)}}$; huge categories of categories will be denoted with $\smash{\widehat{(-)}}$ and capital letters.\\
For example, $\CatUinfty$ is the large category of small categories, while $\CatVinfty$ is the very large category of large categories, and $\CatWinfty$ is the huge category of very large categories.
\item We shall denote the large category of small spaces by $\scrS$. In particular, by \textit{space} we always mean \textit{small space}.
\item The large category $\PrLU$ of large presentable categories and the very large category of $\CatVinftycocom$ of large cocomplete categories are both symmetric monoidal categories: $\Ebb_k$-algebras inside $\PrLU$ and $\CatVinftycocom$ are (respectively) presentable and cocomplete categories endowed with an $\Ebb_k$-monoidal structure that commutes with colimits separately in each variable. In order to compactify our notations, in the rest of our paper we shall refer to an $\Ebb_k$-algebra in $\PrLU$ as a \textit{presentably $\Ebb_k$-monoidal category}, and to an $\Ebb_k$-algebra in $\CatVinftycocom$ as a \textit{cocompletely $\Ebb_k$-monoidal category}; in the case $k=\infinity$ we shall simply write \textit{symmetric monoidal} in place of $\Einf$-monoidal. The notation for $\Ebb_k$-algebras in $\CatVinftycocom$ can sound ambiguous, since an $\Ebb_k$-monoidal structure on a cocomplete category can fail to be compatible with colimits: for an easy counterexample, just consider the category of pointed spaces endowed with the Cartesian symmetric monoidal structure. However, we shall never be interested in such kind of monoidal structures in this work.
\item In a similar fashion, for any $k\in\NN_{\geqslant1}\cup\left\{\infinity\right\}$ and a cocompletely (resp. presentably) $\Ebb_k$-monoidal \infinity-category $\scrA$, we shall say that a category $\scrC$ is \textit{cocompletely} (resp. \textit{presentably}) left tensored over $\scrA$ if it is a left $\scrA$-module in $\CatVinftycocom$ (resp. in $\PrLU$). This formula amounts to the datum of a cocomplete (or presentable) category $\scrC$ which is left tensored over $\scrA$ in such a way that the tensor action functor commutes with colimits separately in each argument.
\item In \cref{sec:localsystemncat,sec:topactionncat} we shall deal with higher (i.e., $n$-)categories, and in particular with $(n+1)$-categories of (possibly decorated) $n$-categories. We shall denote such $n$-categories with a bold font. In order to avoid confusion concerning the ``categorical height" we are working at, we shall also adopt the following highly non-standard notation as well: if we want to refer to the (very large) higher category of large $m$-categories seen as a $n$-category, we shall write $n\CatVtwoinfty[m]$. In the particular case $n=1$, we shall drop both the bold font and the $1$ before our notations, and simply write $\CatVinfty[m]$. For example, $3\CatVtwoinfty[2]$ is the very large $3$-category of all large $2$-categories, while $2\CatVtwoinfty[2]$ is its underlying $2$-category, and $\CatVinfty[2]$ is its underlying $1$-category. (See also \cref{notation:highercat}.)
\item Most of the times we will consider categories which are enriched over some preferred category (e.g., modules in spectra which are enriched over themselves, or presentably enriched categories which are enriched over themselves, and so forth). At the same time, we will need to consider the underlying spaces of maps between objects in such categories. For this reason, when $\scrC$ is enriched over a category $\scrA$, we will denote as $\Map_{\scrC}(-,-)$ the space of maps in $\scrC$, and as $\Mapin_{\scrC}(-,-)$ 
the morphism object of $\scrA$ providing the enrichment, so as to to highlight whether we are seeing a morphism object as a space or as something more structured. If $\scrC$ is a higher category of categories (e.g., $\scrC=\CatVinftycocom$ or $\scrC=\PrLU$) we will also use $\Funin(-,-)$, possibly with decorations, to mean the category of structure-preserving functors which serves as the category of morphisms in $\scrC$.
\end{itemize}
\subsection*{Acknowledgements}
This project took quite a long time to complete, and along the way  we benefited enormously from conversations and email exchanges with friends and colleagues that helped us navigate the many subtle issues in the theory of higher local systems and $n$-category theory. We owe special thanks to Ivan Di Liberti, Andrea Gagna, Guglielmo Nocera, Mauro Porta and German Stefanich.
\addtocontents{toc}{\protect\setcounter{tocdepth}{2}}
\input{categorical_local_systems1.tex}
\input{categorical_local_systems2.tex}
\input{symplectic-geometry.tex}
\printbibliography
\end{document}

%% file: preamble.tex
\pdfoutput=1
\newif\ifpersonal

\personaltrue 

\ifpersonal
\newcommand*{\personal}[1]{\textcolor[rgb]{0,0,1}{(Personal: #1)}}
\newcommand*{\todo}[1]{\textcolor{red}{(Todo: #1)}}
\else
\newcommand*{\personal}[1]{\ignorespaces}
\newcommand*{\todo}[1]{\ignorespaces}
\fi

\usepackage[toc,page]{appendix}
\usepackage[foot]{amsaddr}
\usepackage{etoolbox}

\makeatletter
\patchcmd{\@setaddresses}{\indent}{\noindent}{}{}
\patchcmd{\@setaddresses}{\indent}{\noindent}{}{}
\patchcmd{\@setaddresses}{\indent}{\noindent}{}{}
\patchcmd{\@setaddresses}{\indent}{\noindent}{}{}
\makeatother

\usepackage{amsthm,mathtools,bm,tensor}
\newcommand*{\boldone}{\text{\usefont{U}{bbold}{m}{n}1}}
\usepackage[T1]{fontenc}\newcommand{\yo}{\text{\usefont{U}{min}{m}{n}\symbol{'110}}}
\DeclareFontFamily{U}{min}{}
\DeclareFontShape{U}{min}{m}{n}{<-> dmjhira}{}
\usepackage[bitstream-charter]{mathdesign}
\usepackage{XCharter}
\usepackage{enumerate,comment,braket,xspace,tikz-cd,csquotes}
\usepackage[driver=dvips,centering,vscale=0.7,hscale=0.8]{geometry}
\usepackage{url}
\usepackage{microtype}
\usepackage[greek.polutoniko,english]{babel}
\usepackage[style=authoryear,backend=biber,sorting=nyt,citestyle=ieee-alphabetic, bibstyle=alphabetic,texencoding=ascii,bibencoding=utf8,maxnames=5,minnames=3,maxbibnames=99]{biblatex}
\usepackage{setspace}
\usepackage{enumitem}
\usepackage{eucal}
\usepackage[scr=boondoxo]{mathalfa}
\usepackage{listings}
\usepackage{colordvi}
\usepackage{tikz,pgf}
\usepackage{pict2e}
\usetikzlibrary{arrows,decorations.pathmorphing,backgrounds,positioning,fit,matrix,calc}
\tikzset{curve/.style={settings={#1},to path={(\tikztostart)
			.. controls ($(\tikztostart)!\pv{pos}!(\tikztotarget)!\pv{height}!270:(\tikztotarget)$)
			and ($(\tikztostart)!1-\pv{pos}!(\tikztotarget)!\pv{height}!270:(\tikztotarget)$)
			.. (\tikztotarget)\tikztonodes}},
	settings/.code={\tikzset{quiver/.cd,#1}
		\def\pv##1{\pgfkeysvalueof{/tikz/quiver/##1}}},
	quiver/.cd,pos/.initial=0.35,height/.initial=0}
\tikzset{tail reversed/.code={\pgfsetarrowsstart{tikzcd to}}}
\tikzset{2tail/.code={\pgfsetarrowsstart{Implies[reversed]}}}
\tikzset{2tail reversed/.code={\pgfsetarrowsstart{Implies}}}
\tikzset{Rightarrow/.style={double equal sign distance,>={Implies},->},
triple/.style={-,preaction={draw,Rightarrow}},
quadruple/.style={preaction={draw,Rightarrow,shorten >=0pt},shorten >=1pt,-,double,double
distance=0.2pt}}
\usepackage{epsfig}
\usepackage{epstopdf}
\usepackage{tikz-cd}
\usepackage{fancyhdr}
\usepackage{graphicx}
\usepackage{fullpage}
\makeatletter
\newcommand{\mylabel}[2]{#2\def\@currentlabel{#2}\label{#1}}
\makeatother
\usepackage{hyperref}
\hypersetup{
	colorlinks=true,
	linkcolor=red,
        linktoc=page,
	anchorcolor=black,
	filecolor=blue,
	citecolor = blue, 
	menucolor=., 
	urlcolor=cyan,
	breaklinks=true,
    unicode=true,
}
\usepackage[capitalize]{cleveref}
\usepackage{nameref}
\usepackage{enumitem}
\crefalias{enumi}{lemman}
\usepackage{comment}
\usepackage{color}
\usepackage[totoc]{idxlayout}
\usepackage[lastexercise]{exercise}
\usepackage{xcolor}
\usepackage{breakcites}
\usetikzlibrary{positioning}

\makeatletter
\newcommand{\twoarrows}[3][0.2ex]{%
  \mathrel{\mathpalette\twoarrows@{{#1}{#2}{#3}}}%
}
\newcommand{\twoarrows@}[2]{\twoarrows@@#1#2}
\newcommand{\twoarrows@@}[4]{%
  \vcenter{\offinterlineskip\m@th
    \ialign{\hfil##\hfil\cr
      $#1#3$\cr
      \noalign{\vskip#2}
      $#1#4$\cr
    }%
  }%
}
\makeatother

\makeatletter
\newcommand*{\doublerightarrow}[2]{\mathrel{
		\settowidth{\@tempdima}{$\scriptstyle#1$}
		\settowidth{\@tempdimb}{$\scriptstyle#2$}
		\ifdim\@tempdimb>\@tempdima \@tempdima=\@tempdimb\fi
		\mathop{\vcenter{
				\offinterlineskip\ialign{\hbox to\dimexpr\@tempdima+1em{##}\cr
					\rightarrowfill\cr\noalign{\kern.5ex}
					\rightarrowfill\cr}}}\limits^{\!#1}_{\!#2}}}
\newcommand*{\triplerightarrow}[1]{\mathrel{
		\settowidth{\@tempdima}{$\scriptstyle#1$}
		\mathop{\vcenter{
				\offinterlineskip\ialign{\hbox to\dimexpr\@tempdima+1em{##}\cr
					\rightarrowfill\cr\noalign{\kern.5ex}
					\rightarrowfill\cr\noalign{\kern.5ex}
					\rightarrowfill\cr}}}\limits^{\!#1}}}
\makeatother
\usepackage{hologo}
\makeatletter
\def\@tocline#1#2#3#4#5#6#7{\relax
	\ifnum #1>\c@tocdepth 
	\else
	\par \addpenalty\@secpenalty\addvspace{#2}%
	\begingroup \hyphenpenalty\@M
	\@ifempty{#4}{%
		\@tempdima\csname r@tocindent\number#1\endcsname\relax
	}{%
		\@tempdima#4\relax
	}%
	\parindent\z@ \leftskip#3\relax \advance\leftskip\@tempdima\relax
	\rightskip\@pnumwidth plus4em \parfillskip-\@pnumwidth
	#5\leavevmode\hskip-\@tempdima
	\ifcase #1
	\or\or \hskip 1em \or \hskip 2em \else \hskip 3em \fi%
	#6\nobreak\relax
	\dotfill\hbox to\@pnumwidth{\@tocpagenum{#7}}\par
	\nobreak
	\endgroup
	\fi}
\makeatother

\newlist{defenum}{enumerate}{1} 
\setlist[defenum]{label=\arabic*), ref=\thedefn.(\arabic*)}
\crefalias{defenumi}{definition} 
\newlist{lemenum}{enumerate}{1} 
\setlist[lemenum]{label=\arabic*), ref=\thelemma.(\arabic*)}
\crefalias{lemenumi}{lemma} 
\newlist{assumpenum}{enumerate}{1} 
\setlist[assumpenum]{label=\arabic*), ref=\theassumption.(\arabic*)}
\crefalias{assumpenumi}{assumption} 
\newlist{propenum}{enumerate}{1} 
\setlist[propenum]{label=\arabic*), ref=\theproposition.(\arabic*)}
\crefalias{propenumi}{proposition}
\newlist{exmpenum}{enumerate}{1} 
\setlist[exmpenum]{label=\arabic*), ref=\theexmp.(\arabic*)}
\crefalias{exmpenumi}{example}
\newlist{thmenum}{enumerate}{1} 
\setlist[thmenum]{label=\arabic*), ref=\thethm.(\arabic*)}
\crefalias{thmenumi}{theorem}
\newlist{factenum}{enumerate}{1} 
\setlist[factenum]{label=\arabic*), ref=\thefact.(\arabic*)}
\crefalias{factenumi}{fact}
\newlist{exercenum}{enumerate}{1} 
\setlist[exercenum]{label=\arabic*), ref=\theexercise.(\arabic*)}
\crefalias{exercenumi}{exercise}
\newlist{remarkenum}{enumerate}{1} 
\setlist[remarkenum]{label=\arabic*), ref=\theremark.(\arabic*)}
\crefalias{remarkenumi}{remark} 
\newlist{axiomenum}{enumerate}{1} 
\setlist[axiomenum]{label=\arabic*), ref=\theaxiom.(\arabic*)}
\crefalias{axiomenumi}{axiom} 
\newlist{questenum}{enumerate}{1} 
\setlist[questenum]{label=\arabic*), ref=\thequestion.(\arabic*)}
\crefalias{questenumi}{question} 
\newtheorem{innercustomthm}{Theorem}
\newenvironment{customthm}[1]
{\renewcommand\theinnercustomthm{#1}\innercustomthm}
  {\endinnercustomthm}
  
   \makeatletter
\newcommand{\adjunction}{\@ifstar\named@adjunction\normal@adjunction}
\newcommand{\normal@adjunction}[4]{%
  #1\colon #2%
  \mathrel{\vcenter{%
    \offinterlineskip\m@th
    \ialign{%
      \hfil$##$\hfil\cr
      \longrightharpoonup\cr
      \noalign{\kern-.3ex}
      \smallbot\cr
      \longleftharpoondown\cr
    }%
  }}%
  #3 \noloc #4%
}
\newcommand{\named@adjunction}[4]{%
  #2%
  \mathrel{\vcenter{%
    \offinterlineskip\m@th
    \ialign{%
      \hfil$##$\hfil\cr
      \scriptstyle#1\cr
      \noalign{\kern.1ex}
      \longrightharpoonup\cr
      \noalign{\kern-.3ex}
      \smallbot\cr
      \longleftharpoondown\cr
      \scriptstyle#4\cr
    }%
  }}%
  #3%
}
\newcommand{\longrightharpoonup}{\relbar\joinrel\rightharpoonup}
\newcommand{\longleftharpoondown}{\leftharpoondown\joinrel\relbar}
\newcommand\noloc{%
  \nobreak
  \mspace{6mu plus 1mu}
  {:}
  \nonscript\mkern-\thinmuskip
  \mathpunct{}
  \mspace{2mu}
}
\newcommand{\smallbot}{%
  \begingroup\setlength\unitlength{.15em}%
  \begin{picture}(1,1)
  \roundcap
  \polyline(0,0)(1,0)
  \polyline(0.5,0)(0.5,1)
  \end{picture}%
  \endgroup
}
\makeatother

\newcommand{\infinity}{\mbox{\footnotesize $\infty$}}

\newcommand{\triv}{\operatorname{triv}}
\newcommand{\oblv}{\operatorname{oblv}}

\newcommand{\Einf}{\mathbb{E}_{{\scriptstyle\infty}}}

\newcommand{\LMod}{{\operatorname{LMod}}}
\newcommand{\LModtwo}{\mathbf{LMod}}
\newcommand{\Modtwo}{\mathbf{Mod}}

\newcommand{\RMod}{{\operatorname{RMod}}}

\newcommand{\op}{\operatorname{op}}

\newcommand{\End}{{\operatorname{End}}}
\newcommand{\Endin}{{\underline{\smash{\End}}}}

\newcommand{{\Cn}}{\operatorname{Cn}}

\newcommand{\Sphere}{\mathbb{S}}
\DeclareMathOperator*{\colim}{colim}

\newcommand{\Mapin}{{\underline{\smash{\Map}}}}

\newcommand{\scrA}{\mathscr{A}}
\newcommand{\scrC}{\mathscr{C}}
\newcommand{\scrD}{\mathscr{D}}

\newcommand{\scrF}{\mathscr{F}}

\newcommand{\scrO}{\mathscr{O}}

\newcommand{\Hom}{{\operatorname{Hom}}}
\newcommand{\hsp}{\hspace{0.05cm}}

\newcommand{\Mod}{{\operatorname{Mod}}}

\newcommand{\lp}{\left(}
\newcommand{\rp}{\right)}
\newcommand{\scrU}{\mathscr{U}}
\newcommand{\scrV}{\mathscr{V}}
\newcommand{\CatUinfty}[1][1]{\operatorname{Cat}_{({\scriptstyle\infty},#1)}}

\newcommand{\CatVinfty}[1][1]{\smash{\widehat{\operatorname{Cat}}_{({\scriptstyle\infty},#1)}}}
\newcommand{\CatVinftytiny}[1][1]{\smash{\widehat{\operatorname{Cat}}}_{({\scriptscriptstyle\infty},#1)}}
\newcommand{\CatVinftycocom}[1][1]{\smash{\widehat{\operatorname{Cat}}}^{\operatorname{rex}}_{({\scriptstyle\infty},#1)}}
\newcommand{\CatVinftycocomtiny}[1][1]{\smash{\widehat{\operatorname{Cat}}}^{\operatorname{rex}}_{({\scriptscriptstyle\infty},#1)}}

\newcommand{\CatVtwoinfty}[1][1]{\smash{\widehat{\mathbf{Cat}}_{({\scriptstyle\infty},#1)}}}
\newcommand{\CatVinftytwo}[1][1]{\widehat{\mathbf{Cat}}_{({\scriptstyle\infty},#1)}}
\newcommand{\CatWinfty}[1][1]{\smash{\widehat{\operatorname{CAT}}_{({\scriptstyle\infty},#1)}}}
\newcommand{\CatWinftycocom}[1][1]{\smash{\widehat{\operatorname{CAT}}}^{\operatorname{rex}}_{({\scriptstyle\infty},#1)}}

\newcommand{\infinitone}{(\infinity,1)}

\newcommand{\PrLU}[1][1]{\operatorname{Pr}^{\operatorname{L}}_{({\scriptstyle\infty},#1)}}
\newcommand{\PrLV}[1][1]{\smash{\widehat{\operatorname{Pr}}}^{\operatorname{L}}_{({\scriptstyle\infty},#1)}}

\newcommand{\PrLUtwo}[1][1]{\mathbf{Pr}^{\operatorname{L}}_{({\scriptstyle\infty},#1)}}
\newcommand{\PrLUtwotiny}[1][1]{\mathbf{Pr}^{\operatorname{L}}_{({\scriptscriptstyle\infty},#1)}}

\newcommand{\PrLk}{\operatorname{Pr}^{\operatorname{L}}_{\kappa}}
\newcommand{\PrRU}[1][1]{\operatorname{Pr}^{\operatorname{R}}_{({\scriptstyle\infty},#1)}}

\newcommand{\LC}{\operatorname{LocSys}}
\newcommand{\LocSysCat}{\operatorname{LocSysCat}}
\newcommand{\LocSysCattwo}{\mathbf{LocSysCat}}

\newcommand{\LCXPrL}{\operatorname{LocSysCat}{\lp X\rp}}

\newcommand{\LinAPrLU}[1][1]{\operatorname{Lin}_{\scrA}\operatorname{Pr}^{\operatorname{L}}_{({\scriptstyle\infty},#1)}}
\newcommand{\LinAPrLUtiny}[1][1]{\mathrm{Lin}_{\scrA}\mathrm{Pr}^{\operatorname{L}}_{({\scriptscriptstyle\infty},#1)}}

\newcommand{\LinAPrLUtwo}[1][1]{\mathbf{Lin}_{\scrA}\mathbf{Pr}^{\operatorname{L}}_{({\scriptstyle\infty},#1)}}
\newcommand{\LinAPrLUtwotiny}[1][1]{\mathbf{Lin}_{\scrA}\mathbf{Pr}^{\operatorname{L}}_{({\scriptscriptstyle\infty},#1)}}

\newcommand{\LinkPrLU}[1][1]{{{\operatorname{Lin}_{\Bbbk}}{\operatorname{Pr}^{\operatorname{L}}_{({\scriptstyle\infty},#1)}}}}\newcommand{\LinkPrLUinv}[1][1]{{{\operatorname{Lin}_{\Bbbk}}{\operatorname{Pr}^{\operatorname{L,inv}}_{({\scriptstyle\infty},#1)}}}}
\newcommand{\LinkPrLUtwo}[1][1]{{\mathbf{Lin}_{\Bbbk}\mathbf{Pr}^{\operatorname{L}}_{({\scriptstyle\infty},#1)}}}

\newcommand{\LCXPrLA}{\operatorname{LocSysCat}{\lp X;\scrA\rp}}

\newcommand{\LCXPrLAtwo}{\mathbf{LocSysCat}{\lp X;\scrA\rp}}

\newcommand*{\longhookrightarrow}{\ensuremath{\lhook\joinrel\relbar\joinrel\rightarrow}}

\makeatletter

\def\enddoc@text{}
\makeatother
\newcounter{savedchapter}
\preto\appendix{\setcounter{savedchapter}{\arabic{chapter}}}
\newcommand\resumechapters{
	\setcounter{chapter}{\arabic{savedchapter}}
	\setcounter{section}{0}
	\gdef\@chapapp{\chaptername}
	\gdef\thechapter{\@arabic\c@chapter}
}
\makeatother

\newcommand{\scrS}{\mathscr{S}}

\newcommand{\HH}{{\operatorname{HH}}}

\newcommand{\Ebb}{\mathbb{E}}

\newcommand{\ZZ}{\mathbb{Z}}

\newcommand{\NN}{\mathbb{N}}

\newcommand{\Map}{{\operatorname{Map}}}

\newcommand{\Fun}{{\operatorname{Fun}}}
\newcommand{\Funin}{{\underline{\smash{\Fun}}}}
\newcommand{\LocSys}{{\operatorname{LocSys}}}
\newcommand{\Funtwo}{\mathbf{Fun}}

\newcommand{\FuninL}{\Funin^{\operatorname{L}}}

\newcommand{\scrM}{{\mathscr{M}}}

\newcommand{\CAlg}{{\operatorname{CAlg}}}

\newcommand{\Alg}{{\operatorname{Alg}}}

\numberwithin{equation}{section}
\theoremstyle{plain}
\newtheorem{theorem}[equation]{Theorem}

\newtheorem{lemma}[equation]{Lemma}
\newtheorem{proposition}[equation]{Proposition}

\newtheorem{corollary}[equation]{Corollary}
\newtheorem{conj}[equation]{Conjecture}

\theoremstyle{definition}
\newtheorem{defn}[equation]{Definition}
\newtheorem{parag}[equation]{}

\newtheorem{remark}[equation]{Remark}

\newtheorem{construction}[equation]{Construction}

\newtheorem{warning}[equation]{Warning}
\newtheorem{notation}[equation]{Notation}

\newtheorem{exmp}[equation]{Example}

\newtheorem*{warning*}{Warning}
\newtheorem*{assumption*}{Standing assumption}
\newcommand{\Sp}{\mathscr{S}\mathrm{p}}

\newcommand\restr[2]{{
		\left.\kern-\nulldelimiterspace
		#1
		\vphantom{\|}
		\right|_{#2} 
}}
\usepackage{lipsum}
\usepackage{fancyhdr}
\providecommand{\abstract}{}
\usepackage{abstract}
\DeclareGraphicsExtensions{.png,.jpg,.pdf,.mps}
\usepackage{fancyhdr}
\usepackage{lastpage}
\usepackage{multirow}
\usepackage{faktor}

\usepackage{array}
\usepackage{makecell}

\newcommand{\stackspace}{1.7}
\newcommand{\stack}[2][1cm]{\;\tikz[baseline, yshift=.65ex]%
	{\foreach \k [evaluate=\k as \r using (.5*#2+.5-\k)*\stackspace] in {1,...,#2}{%
			\ifodd\k{\draw[->](0,\r pt)--(#1,\r pt);}%
			\else{\draw[<-](0,\r pt)--(#1,\r pt);}\fi
	}}\;}
    \newcommand{\stackrev}[2][1cm]{\;\tikz[baseline, yshift=.65ex]%
	{\foreach \k [evaluate=\k as \r using (.5*#2+.5-\k)*\stackspace] in {1,...,#2}{%
			\ifodd\k{\draw[<-](#1,\r pt) -- (0,\r pt);}%
			\else{\draw[->](#1,\r pt) -- (0,\r pt);}\fi
	}}\;}

\usepackage{doc}

%% file: introduction.tex
In this paper we study local systems of higher categories over spaces. Our aim is to generalize and extend to higher categories the following well-known classical picture. Let $X$ be a simply connected space and let $\Bbbk$ be a field. Local systems of $\Bbbk$-vector spaces over $X$ are determined by monodromy data, in the sense that the abelian category of such local systems is equivalent to the category of representations of $\pi_{1}(X)$. Understanding the higher cohomology of local systems requires more information that is not captured by $\pi_1(X)$, and in fact depends on the full homotopy type of $X$. Indeed, the stable category of complexes of sheaves of vector spaces on $X$ whose cohomology sheaves are local systems is equivalent to the stable category of modules over $\mathrm{C}_{\bullet}(\Omega_*X;\Bbbk)$, the algebra of chains on the based loop space of $X$. In formulas, there is an equivalence
\begin{equation}
\label{fundintro}
\LocSys(X;\Bbbk) \simeq \LMod_{\mathrm{C}_{\bullet}(\Omega_*X;\Bbbk)} 
\end{equation} which we shall refer to as the \emph{monodromy equivalence}.

Passing to the $n$-categorical level, local systems of vector spaces are replaced by local systems of $\Bbbk$-linear \emph{presentable} $(\infinity,n)$-categories,  the loop space is replaced by the $(n+1)$-fold iterated loop space $\Omega_{*}^{n+1}X$. This paper seeks to  establish analogues of the monodromy equivalence in the setting of local systems of  \emph{presentable} $(\infinity,n)$-categories. 
Presentable categories have long been familiar to practitioners of $\infinity$-categories, but their $(\infinity,n)$-categorical  analogues have only been recently introduced by Stefanich in \cite{stefanich2020presentable}. If $X$ is an $(n+2)$-connected space, our main result states that there is an equivalence between two $(\infinity,n+1)$-categories: on the one hand, the $(\infinity,n+1)$-category of local systems of presentable $(\infinity,n)$-categories over $X$; and on the other hand, a category of iterated modules over the $\Ebb_{n+1}$-algebra of chains over the $(n+1)$-fold loop space $
\mathrm{C}_{\bullet}(\Omega_*^{n+1}X;\Bbbk)
$. 
This fits well with the familiar picture according to which higher local systems should have monodromy along higher dimensional spheres.

Much of our interest in these questions stems from that fact that categorical local systems  and more generally schobers, i.e., categorical local systems  with singularities, play an increasing role in symplectic geometry and mirror symmetry. They also feature prominently in recent approaches to 3d mirror symmetry  
\cite{gammage2023perverse}. This is a mysterious duality which is only beginning to be explored, and that has deep connections with many areas of mathematics and particularly geometric representation theory, where is sometimes referred to as symplectic duality \cite{braden2016conical}. 
We are particularly indebted to ideas of Teleman on 3d homological mirror symmetry, and some of our results rigorously formalise  insights that first appeared in his \cite{teleman}. The thesis of Toly Preygel \cite{preygel2011thomsebastianidualitymatrix}  also sketches some ideas that we have formalised.

We shall give next a more analytical description of the contents of the paper, and state our main results. In the final section of the introduction we shall explain more broadly the motivations of our work coming from symplectic geometry.


\subsection{Main results}
\label{mainresults}
In \cref{sec:localsystemsinfinityn} of the main text we survey briefly all preliminary material which will be required in the remainder of the paper. In particular we give  a thorough overview of the basic definitions and results in the theory of local systems. 
Let $\scrC$ be an $(\infinity,1)$-category, and let $X$ be a space. We define the category of $\scrC$-valued local systems on $X$ as the category of functors$$
\LocSys(X;\scrC) \coloneqq\Fun(X, \hsp \scrC)
$$between $X$ (viewed as an $\infinity$-groupoid) and $\scrC$. A general form of the monodromy equivalence for sheaves with values in $\scrC$ was  proved by Lurie, and then by Beardsley and P\'eroux \cite{koszulduality}  in a formulation which is more directly relevant for us. Namely, let $X$ be a connected space, and let  $\scrC$ be a \emph{presentable} $(\infinity,1)$-category, then there exists an equivalence of $(\infinity,1)$-categories
\begin{equation}
\label{master}
\LC(X;\scrC)\simeq \LMod_{\Omega_*X}{\lp\scrC\rp}. 
\end{equation}

Equivalence \eqref{master} is fundamental, but its generality is too narrow to be applicable to local systems of categories. For instance, we would like to take 
$\scrC$ to be $\PrLU$ itself, as this would allow us to describe local systems whose  sections are presentable categories; however $\PrLU$ is  not presentable, but only cocomplete, and therefore does not fall under the scope of the previous statement. In \cref{sec:categorifyingmodules} we address  this issue, by proving that cocompleteness is sufficient to obtain a monodromy description of local systems.

 \begin{customthm}{A}[\cref{monodromyeq}, \cref{cor:koszulpresentable}]
\label{intromain0}
Let $X$ be a connected space, and let $\scrC$ be a cocomplete (not necessarily presentable) category. Then there exists an equivalence of categories
$$\LocSys(X;\scrC)\simeq \LMod_{\Omega_*X}{\lp\scrC\rp}.$$
In particular, when $\scrC = \PrLU$ we obtain an equivalence of categories
\begin{equation}
\label{master2}
\LocSysCat(X)\simeq \LMod_{\Omega_*X}{\lp\PrLU\rp}.
\end{equation}
\end{customthm}

In \cref{sec:highermonodromydata} we  generalise equivalence \eqref{master2} 
  to local systems of \emph{presentable} $(\infinity,n)$-categories. 
  Presentable  $(\infinity,n)$-categories have been recently introduced by Stefanich \cite{stefanich2020presentable}. For the benefit of the reader we include a summary of the theory in the main text. For the sake of this introduction however we will limit ourselves to say that the category $\PrLU[n]$ of presentable $(\infinity,n)$-categories is symmetric monoidal and that it is a $n$-fold delooping of $\PrLU$ (i.e. we can recover $\PrLU$ by taking iterated endomorphisms of the unit). It enjoys many of the formal properties of $\PrLU$, and as such it provides a favourable environment for $(\infinity,n)$-category theory. We denote the incarnation as a $(\infinity,n+1)$-category of $\PrLU[n]$ as $(n+1)\PrLUtwo[n].$
The next is one of our main results. We comment on the statement below.  
\begin{customthm}{B}[\cref{conj:infinityn}]
\label{intromain1}
Let $n\geqslant 1$ be an integer, let $X$ be a pointed $n$-connected space (i.e., $\pi_k(X)\cong 0$ for every $k\leqslant n$). Then there exist equivalences of $(\infinity,n+1)$-categories
\begin{align*}
(n+1)\LocSysCattwo^n(X)&\simeq  (n+1)\mathbf{LMod}_{n\LModtwo_{\Omega_*^{n+1}X}(\scrS)}\PrLUtwo[n].
\end{align*}
\end{customthm}
Let us make some comments on the statement, as some of the notations will only be introduced in the main text. The category on the left hand side is the $(\infinity,n+1)$-category of local systems of presentable $(\infinity,n)$-categories over $X$; the category on the right hand side is the $(\infinity,n+1)$-category of presentable $(\infinity,n)$-categories with an action of the presentable $(\infinity,n)$-category of iterated left modules over the grouplike topological $\mathbb{E}_{n+1}$-monoid $\Omega_*^{n+1}X$. In fact, the connectedness assumptions on $X$ can be dropped, as we clarify in Paragraph \ref{remark:nconnectedness} of the main text. We also stress that in the main text we always work relative to a monoidal presentable category $\scrA$: for simplicity, above,  we stated   our result only in the absolute case, when $\scrA$ is the $(\infinity,1)$-category of topological spaces $\scrS$. Of particular interest for applications is also the stable setting when $\scrA$ is e.g. the $(\infinity,1)$-category of spectra or $\Bbbk$-modules for a field $\Bbbk$. In this latter case, the corresponding $(\infinity,n+1)$-category on the right hand side in \cref{intromain1} can be interpreted as an $(\infinity,n+1)$-category of ``presentable $\mathrm{C}_{\bullet}(\Omega^{n+1}_*X;\Bbbk)$-linear $n$-categories".

It is implicit in the statement of \cref{intromain1}  
that we can make sense of the action of a topological monoid on a category.   
This is built in in $\infinity$-category theory: by definition all cocomplete $(\infinity,1)$-categories (e.g. $\PrLU$) are  tensored over the $(\infinity,1)$-category of spaces $\scrS$. This provides a natural notion of the action of a monoid in $\scrS$ on  an object in a cocomplete category, and so in particular on a presentable category (i.e. an object in $\PrLU$, which is itself cocomplete). Working in the dg setting  makes the nature of topological actions on categories less evident, and in fact Teleman proposes various alternative definitions in \cite{teleman}, before settling on one. We show that Teleman's preferred model coincides with the natural concept of topological action in $\infinity$-category theory, and   turn Teleman's beautiful ansatz into a theorem.

\begin{customthm}{C}[\cref{thm:teleman}]
\label{thm:main2}
Topological actions of a connected group $G$ on a differential graded category $\scrC$ which is linear over some base commutative ring $\Bbbk$ are completely captured, up to contractible choices, by the induced $\Ebb_2$-algebra morphisms$$\mathrm{C}_{\bullet}{\lp\Omega_*G;\Bbbk\rp}\longrightarrow\operatorname{HH}^{\bullet}(\scrC)$$where the source is simply the algebra of chains of $\Omega_*G$ with coefficients in $\Bbbk$, endowed with the Pontrjagin product, and the target is the Hochschild cohomology of the differential graded category $\scrC$.
\end{customthm}

We remark that Theorem \ref{thm:main2} appears as \cite[Theorem $2.5$]{teleman}. We stress however that in \cite{teleman} this claim appears without proof: in fact one could argue that rather than a theorem,   it is a reformulation of Teleman's definition of a topological action on dg-categories. Within the framework of $\infinity$-categories, however, it becomes a  a non-tautological statement about actions of topological monoids, and we will prove it rigorously. In this respect, we believe that our contribution   consists in providing a   formalisation  of Teleman's insight within the theory of $\infinity$-categories.

In fact, in the main text we prove  a more precise and more general statement, which applies to  a presentable category $\scrC$ which is linear over a presentably symmetric monoidal category $\scrA$. Namely, let  $\LMod_{\Omega_*G\otimes\boldone_{\scrA}}(\scrA)\text{-}\operatorname{ModStr}(\scrC)$ be the space of all possible left $\LMod_{\Omega_*G\otimes \boldone_{\scrA}}{\lp\scrA\rp}$-module structures on $\scrC$, and let $G\text{-}\operatorname{ModStr}(\scrC)$ be the space of all possible left $G$-module structures on $\scrC$. Then \cref{lemma:hochschildand2modules} in the main text states  that there are equivalences of spaces
	\begin{equation}
	\label{eqeqintro}
	\Map_{\Alg_{\Ebb_2}(\scrA)}{\lp\Omega_*G\otimes \boldone_{\scrA},\hsp \HH^{\bullet}(\scrC)\rp}
	\simeq 
	\LMod_{\Omega_*G\otimes\boldone_{\scrA}}(\scrA)\text{-}\operatorname{ModStr}(\scrC)
	 \simeq  
	  G\text{-}\operatorname{ModStr}(\scrC).
	\end{equation}
	Note that when $\scrA= \Mod_{\Bbbk}$ the chain of equivalences \ref{eqeqintro} specializes to a refinement of \cref{thm:main2}, where the topology of the moduli spaces is also taken into account. 
	
We also obtain a generalization of \ref{eqeqintro}  to the setting of topological actions on presentable $(\infinity,n)$-categories. The study of the category $(n+1)\PrLUtwo[n]$ is still in its initial stages, and some of its fundamental properties have yet to be  established. In particular it is expected, but not known, that the symmetric monoidal structure on $n\PrLUtwo[n-1]$ is closed when $n \geqslant 3$. We state this explicitly as \cref{conj:notenriched} in the main text. Under the assumption that \cref{conj:notenriched} holds, we prove the following result.

	\begin{customthm}{D}[\cref{cor:en}]
\label{thm:main3}
	Let $\scrA$ be a presentably symmetric monoidal $(\infinity,1)$-category with monoidal unit $\boldone_{\scrA}$, and let $n\geqslant 1$ be an integer. Let $n\bm{\scrC}$ be an $\scrA$-linear presentable $n$-category, and let $G$ be an $(n-1)$-connected topological group. Let $G\text{-}\operatorname{ModStr}(n\bm{\scrC})$ denote the space of all possible left $G$-module structures on $n\bm{\scrC}$. If \cref{conj:notenriched} holds, there is an equivalence of spaces
	$$
	\Map_{\Alg_{\Ebb_n}(\scrA)}{\lp \Omega^n_*G\otimes\boldone_{\scrA} ,\hsp \mathrm{HH}^{\bullet}_{\Ebb_n}(\bm{\scrC}) \rp}
	 \simeq  
	  G\text{-}\operatorname{ModStr}(n\bm{\scrC}).
	$$
\end{customthm}
	
As a consequence of \cref{thm:main2} and \cref{thm:main3}, we obtain a classification of invertible objects in $(n+1)\LocSysCattwo^n(X)$, which is unconditional in the case of $n=1$ (i.e., local systems of presentable categories); but for $n\geqslant2$ it depends on \cref{conj:notenriched} and on a still conjectural characterization of higher Brauer groups of algebraically closed fields.
	\begin{customthm}{E}[\cref{prop:invertibleobjectsoflocsyscat}, \cref{highertele}]
\label{thm:main4}
	\label{prop:invertibleobjectsoflocsyscat2}
Let $\Bbbk$ be an algebraically closed field. Let $n \geqslant1$ and let $X$ be an $n$-connected space. Then we have an isomorphism of abstract groups
\[
\pi_0{\lp\lp (n+1)\LocSysCattwo^n(X)^{\mathrm{inv}}\rp^{\simeq}\rp}\cong\Hom_{\mathrm{Grp}}{\lp \pi_n(X),\hsp \Bbbk^{\times}\rp}
\]
between the group of equivalence classes of invertible local systems of $\Bbbk$-linear presentable $(\infinity,n)$-categories on $X$, and the group of multiplicative characters of $\pi_{n+1}(X)$. 
\end{customthm}
The group $\pi_0\lp (n+1)\LocSysCattwo^n(X)^{\mathrm{inv}}\rp^{\simeq}$ should be understood as a topological analogue of the (higher) derived Brauer group of $X$; or equivalently of its Betti stack. 
 In fact, the connectedness assumptions could be dropped, but at the cost of a  more cumbersome statement which we leave to the interested reader.

  We conclude the paper with a short section explaining the relevance of our results to symplectic geometry. The message is that local systems of $(\infinity,1)$-categories arise naturally from Hamiltonian fibrations, by taking the Fukaya $(\infinity,1)$-categories of the fibers. More precisely, as an easy consequence of our results and important previous work of Oh--Tanaka \cite{oh-tanaka} and Savelyev \cite{savelyev}, we obtain the following theorem.
  
\begin{customthm}{F}[\cref{hamfib}]
\label{thm:main5}
	\label{symp}
  Let $\pi\colon X \to S$ be a Hamiltonian fibration, such that the fibers $(X_{s},\omega_{s})$ are either compact monotone or are Liouville sectors. Then there is an associated local system of $(\infinity,1)$-categories over $S$ whose fiber over $s \in S$ is the Fukaya category of $(X_{s},\omega_{s})$.
\end{customthm}

%% file: categorical_local_systems1.tex
\section{Categorical local systems and categorical loop space representations}
\label{sec:localsystemsinfinityn}
\numberwithin{equation}{subsection}
\subsection{Preliminaries on local systems in the homotopy setting}
\label{sec:locsys}
In this section we collect the first definitions and notations concerning categorical local systems over spaces, i.e., local systems over spaces with coefficients in some category of categories. Our preferred coefficients shall be the $2$-category of presentable categories, possibly enriched over a presentably symmetric monoidal category $\scrA$.

Given a (strict) topological space, we have a natural way to define what a local system with coefficients in some category is.
\begin{parag}
Let $X$ be a topological space, let $\mathrm{Op}(X)$ be the poset of its open subsets equipped with the Grothendieck topology $\tau$ generated by jointly surjective maps, and let $\scrC$ be any category. We can either consider the topos$$\mathrm{Shv}(X;\scrC)\coloneqq\mathrm{Shv}_{\tau}\lp\mathrm{Op}(X);\scrC\rp$$of $\scrC$-valued sheaves over $X$, or its hypercompletion$$\mathrm{Shv}^{\mathrm{hyp}}(X;\scrC)\coloneqq \widehat{\mathrm{Shv}}(X;\scrC).$$The latter is a localization of the former, i.e., it is a full subcategory closed under limits which admits a hypersheafification left adjoint$$(-)^{\mathrm{hyp}}\colon\mathrm{Shv}(X;\scrC)\longrightarrow\mathrm{Shv}^{\mathrm{hyp}}(X;\scrC).$$
\end{parag}
\begin{defn}
\label{def:localsystems}Let $X$ be a topological space and let $\scrC$ be any category.
\begin{enumerate}
    \item We say that a $\scrC$-valued  sheaf $\scrF$ over $X$ is \textit{constant} if it lies in the essential image of the pullback functor$$\Gamma^*\colon\scrC\simeq \mathrm{Shv}(\left\{*\right\};\scrC)\longrightarrow\mathrm{Shv}(X;\scrC).$$
    \item We say that a sheaf $\scrF$ is \textit{locally constant} if there exists a small collection of objects $\left\{ U_{\alpha}\hookrightarrow X\right\}_{\alpha}$ which is jointly surjective over $X$ such that $\restr{\scrF}{U_{\alpha}}$ is constant in $\mathrm{Shv}(U_{\alpha};\scrC)$.
    \item We say that a hypersheaf $\scrF$ over $X$ is \textit{hyperconstant} if it belongs to the essential image of the functor$$\Gamma^{\mathrm{hyp},*}\colon\scrC\overset{\Gamma^*}{\longrightarrow}\mathrm{Shv}(X;\scrC)\overset{(-)^{\mathrm{hyp}}}{\longrightarrow}\mathrm{Shv}^{\mathrm{hyp}}(X;\scrC).$$
    \item We say that a hypersheaf $\scrF$ over $X$ is \textit{hyperlocally hyperconstant} if there exists a small collection of objects $\left\{ U_{\alpha}\hookrightarrow X\right\}_{\alpha}$ which is jointly surjective over $X$ such that $\restr{\scrF}{U_{\alpha}}$ is hyperconstant in $\mathrm{Shv}^{\mathrm{hyp}}(U_{\alpha};\scrC)$.
\end{enumerate}
\end{defn}
Locally constant sheaves and locally hyperconstant hypersheaves  form  two full subcategories of $\mathrm{Shv}(X;\scrC)$ and $\mathrm{Shv}^{\mathrm{hyp}}(X;\scrC)$. Call them $\mathrm{LC}(X;\scrC)$ and $\mathrm{LC}^{\mathrm{hyp}}(X;\scrC)$, respectively.
\begin{warning}[{\cite[Warning $1.19$]{haineportateyssier}}]
\label{warning:locsysandlocsyshyp}It is not in general true that the natural inclusion$$\mathrm{LC}(X;\scrC)\cap\mathrm{Shv}^{\mathrm{hyp}}(X;\scrC)\subseteq\mathrm{LC}^{\mathrm{hyp}}(X;\scrC)$$is an equivalence of categories, unless $X$ is \textit{locally of constant shape} in the sense of \cite[Definition A.$4.15$]{ha}. In this case, all locally constant sheaves are locally hyperconstant (\cite[Corollary A$.1.7$]{ha}), hence the two expressions obviously match.
\end{warning}
\cite[Corollary $3.7$ and Observation $3.8$]{haineportateyssier} show that the correct notion of locally constant $\scrC$-valued sheaf over a topological space $X$  is  given by locally hyperconstant hypersheaves. If $\scrC$ is a presentable category, then 
\begin{equation}
\label{hyperdisplay}
\mathrm{LC}^{\mathrm{hyp}}(X;\scrC)\simeq\scrS_{/ X}\otimes\scrC\simeq\Fun(X,\hsp\scrC)
\end{equation}
where $\otimes$ denotes the Lurie tensor product of presentable categories and where we have implicitly identified $X$ and its fundamental groupoid $\Pi_{\scriptstyle\infty}(X)$.

We are mostly interested in considering topological spaces $X$ as objects in the category $\scrS$ of homotopy types, rather than in the point-set theoretic sense.   Equivalence   
\eqref{hyperdisplay} gives us a way to think about local systems on a strict topological spaces in terms of data that only depend on its underlying homotopy type, i.e. $\scrC$-valued functors out of $X$ (at least if $\scrC$ is presentable).  
 This motivates the following definition. 
\begin{notation}
\label{notation:locsys}
For $X$ a space and for $\scrC$ \textit{any} category, we set $$\LocSys(X;\scrC)\coloneqq\Fun(X,\hsp\scrC).$$If $\scrC\coloneqq\scrS$ is the category of spaces itself, we shall simplify notations and set $$\LocSys(X)\coloneqq\LocSys(X;\scrS).$$In the same way, if $\Bbbk$ is any $\Ebb_1$-ring spectrum and $\scrC\coloneqq\Mod_{\Bbbk}$ is the category of $\Bbbk$-modules in spectra, then we set $$\LocSys(X;\Bbbk)\coloneqq\LocSys(X;\Mod_{\Bbbk}).$$
In the rest of the paper, we shall often abuse notations and identify a topological space $X$ with its underlying homotopy type $\Pi_{\scriptstyle\infty}(X)$, and simply refer to it as a \textit{space}.
\end{notation}
For later use, we recall the following fundamental monodromy equivalence statement.
\begin{lemma}[{Monodromy equivalence, \cite[Lemma $3.9$]{koszulduality}}]
\label{lemma:koszulduality}
Let $X$ be a connected space, and let $\scrC$ be a presentable category. Then there exists an equivalence of categories$$\LocSys(X;\scrC)\simeq \LMod_{\Omega_*X}{\lp\scrC\rp}.$$
\end{lemma}
\subsection{The monodromy equivalence for cocomplete categories of coefficients}
\label{sec:presentableinftycat}
\label{parag:catinftypresentable}
Our goal in this Section is generalizing   \cref{lemma:koszulduality}  to cocomplete categories which are not necessarily presentable.  In order to explain why this is important for our project, we  have to introduce  first some  objects which will play a key role in the sequel.

Let  $\scrA$ be a presentably symmetric monoidal category.  Let 
$\smash{\Mod_{\scrA}{\lp\PrLU\rp}}$ 
 the category of presentable categories which are presentably tensored over $\scrA$. By \cite[Theorem $1.2$]{heineenrichedaction}, this is the same as the category of presentable categories enriched over $\scrA$ in the sense of \cite{gepnerhaugseng}. In symbols, there is an equivalence
\begin{equation}
\label{onamelin}
\Mod_{\scrA}{\lp\PrLU\rp} \simeq \operatorname{Lin}_{\scrA}\PrLU. 
\end{equation}
In the sequel we shall always use the  notation $\operatorname{Lin}_{\scrA}\PrLU$ to refer to this category. 
In the particular case in which $\scrA$ is the category of $\Bbbk$-modules in spectra for some $\Einf$-ring spectrum $\Bbbk$ we shall simply write  
$\LinkPrLU$. This includes the case in which $\Bbbk=\Sphere$ is the sphere spectrum, hence $\operatorname{Lin}_{\Sphere}\PrLU\simeq\operatorname{Pr}^{\mathrm{L,st}}_{(\scriptstyle\infty,1)}$ denotes the category of stable presentable categories.
\begin{remark}
Note that if we are considering left $\scrA$-modules in the larger category of categories $\CatVinfty$, left $\scrA$-modules and $\scrA$-enriched categories are \textit{not} equivalent anymore. Indeed \cite[Theorem $1.1$]{heineenrichedaction} guarantees only an equivalence between \textit{closed} left $\scrA$-modules in $\CatVinfty$ and a \textit{non-full} subcategory of all $\scrA$-enriched categories. 
\end{remark}
\begin{notation} 
We set  
$$
\LocSysCat(X) \coloneqq {\LocSys}{\lp X;\PrLU\rp} \quad \text{and}\quad \LocSysCat(X;\scrA)\coloneqq {\LocSys}{\lp X;\LinAPrLU\rp}.
$$
\end{notation}
Note that, although both $\PrLU$ and $\LinAPrLU$ are complete and  cocomplete categories, they are \textit{not} presentable. Thus we cannot    apply   \cref{lemma:koszulduality} directly to $\LocSysCat(X)$ and $\LocSysCat(X;\scrA)$. The main result of this Section, \cref{monodromyeq},  generalizes  \cref{lemma:koszulduality}  to cocomplete categories and thus will allow us to circumvent this difficulty.

The following recent result of Stefanich is the key ingredient in the generalization of \cref{lemma:koszulduality} to cocomplete categories. 
\begin{proposition}[{\cite[Propositions $5.1.4$ and $5.1.14$]{stefanich2020presentable}}]
 \label{prop:stefanichcompgen}
    Let $\scrA$ be any presentably symmetric monoidal category, let $\operatorname{Lin}_{\scrA}\CatVinftycocom$ be the category of $\scrA$-linear cocomplete categories, and let $\kappa_0$ be the smallest large cardinal of the theory. Then $\operatorname{Lin}_{\scrA}\CatVinftycocom$ is $\kappa_0$-compactly generated by the category $\LinAPrLU$ of presentably $\scrA$-linear categories.
 \end{proposition}
 Using \cref{prop:stefanichcompgen}, we shall write any cocomplete category $\scrC$ as a (large) filtered colimit of presentable categories, and then use \cref{lemma:koszulduality} to deduce a more general version of the monodromy equivalence for categorical local systems. First, we need to establish that the operation of taking modules for a topological $\Ebb_k$-monoid commutes with colimits of large cocomplete categories.
\begin{lemma}
\label{lemma:modprL}
Let $\CatVinftycocom$ be  the very large category of large cocomplete categories with cocontinuous functors between them. For any $\Ebb_k$-monoid in spaces $A$, the functor$$\LMod_A\colon\CatVinftycocom\longrightarrow\CatVinftycocom$$is part of an ambidextrous adjunction$$\adjunction{\LMod_A}{\CatVinftycocom}{\CatVinftycocom}{\RMod_A}.$$
\end{lemma}
\begin{proof}
First, let us remark that both functors $\LMod_A$ and $\RMod_A$ are actually well defined, since $\CatVinftycocom$ is a symmetric monoidal category under Lurie's tensor product (\cite[Corollary $4.8.1.4$]{ha}) with unit provided by the category of spaces $\scrS$. In particular,$$\Mod_{\scrS}{\lp\CatVinftycocom\rp}\simeq\CatVinftycocom,$$hence any cocontinuous functor between cocomplete categories is $\scrS$-linear, in the sense of \cite[Definition $4.6.2.7$]{ha}. Moreover, for any cocomplete category $\scrC$ the categorical Eilenberg-Watts Theorems (see \cite[Theorems $4.8.4.1$ and $4.8.4.6$]{ha}) yield equivalences:\begin{enumerate}
\item ${\FuninL}{\lp\LMod_A(\scrS),\hsp\scrC\rp}\simeq\RMod_A(\scrC)$ and ${\FuninL}{\lp\RMod_A(\scrS),\hsp\scrC\rp}\simeq\LMod_A(\scrC)$.
\item $\scrC\otimes\RMod_A(\scrS)\simeq \RMod_A(\scrC)$ and $\LMod_A(\scrS)\otimes\scrC\simeq \LMod_A(\scrC)$.
\end{enumerate}
In the first statement, $\FuninL$ denotes the category of cocontinuous functors, which plays the role of a internal mapping object for the \textit{closed} symmetric monoidal structure of $\CatVinftycocom$ (\cite[Remark $4.8.1.6$]{ha}), while in the second statement $\otimes$ denotes Lurie's tensor product of cocomplete of categories. We make the following remarks: 
\begin{enumerate}
    \item Let  $\scrA$ be a cocomplete monoidal category, let $A$ be a $\Ebb_k$-monoid  inside $\scrA$, and let $\scrC$ be a cocomplete category which is left tensored over $\scrA$.  Then the cocomplete category $\RMod_A(\scrC)$ is only a left $\scrA$-module and $\LMod_A(\scrC)$ is only a right $\scrA$-module. However, in our case left and right modules are equivalent because $\scrS$ is symmetric monoidal; in particular, we can harmlessly swap the factors in the formulas $\scrC\otimes\RMod_A(\scrS)$ and $\LMod_A(\scrS)\otimes\scrC$.
    \item In principle, the equivalences in the statement of the categorical Eilenberg-Watts Theorems might not be natural in $\scrC$. It turns out, however, that these equivalences  are in fact  natural. We will prove this in \cref{lemma:eilenbergwattsnatural} below.
    \end{enumerate}

So, given two cocomplete categories $\scrC$ and $\scrD$ we have a chain of equivalences\begin{align*}
{\Map_{\CatVinftycocomtiny}}{\lp\LMod_A(\scrC),\hsp\scrD\rp}&\simeq{\Map_{\CatVinftycocomtiny}}{\lp\scrC\otimes\LMod_A(\scrS),\hsp\scrD\rp}\\&\simeq{\Map_{\CatVinftycocomtiny}}{\lp\scrC,\hsp{\FuninL}{\lp\LMod_A(\scrS),\hsp \scrD\rp}\rp}\\&\simeq{\Map_{\CatVinftycocomtiny}}{\lp\scrC,\hsp\RMod_A(\scrD)\rp},
\end{align*}which are moreover natural in both $\scrC$ and $\scrD$, thanks to \cref{lemma:eilenbergwattsnatural}. It follows that $\LMod_A(-)$ is a left adjoint to $\RMod_A(-)$, hence preserves small colimits.
\end{proof}

\begin{lemma}
\label{lemma:eilenbergwattsnatural}
The equivalences of cocomplete categories$$\FuninL(\LMod_A(\scrS),\hsp\scrC)\simeq\RMod_A(\scrC)$$and$$\scrC\otimes\RMod_A(\scrS)\simeq\RMod_A(\scrC)$$are natural in $\scrC$.
\end{lemma}
\begin{proof}
  First, suppose that $\scrC=\scrS$ is the category of spaces. In this case, \cite[Remark $4.8.4.8$]{ha} guarantees that the first equivalence exhibits $\LMod_A(\scrS)$ as a \textit{left dual} to $\RMod_A(\scrS)$, hence as a \textit{weak left dual} in the sense of \cite[Remark $5.2.5.6$]{ha}. The choice of a weak left dual in a closed symmetric monoidal category can be made functorially (\cite[Remark $5.2.5.10$]{ha}). Moreover, in any closed symmetric monoidal category, fixing any left dualizable object $X$ and an arbitrary object $Y$, the tensor product $Y\otimes X^{\vee}$ serves as an exponential of $Y$ by $X$ (\cite[Lemma $4.6.1.5$]{ha}); again, exponentials can be chosen functorially in $Y$ (\cite[Remark $4.6.1.3$]{ha}). Combining this argument with \cite[Lemma $4.6.1.6$]{ha}, we deduce the existence of a string of equivalences\begin{align*}
     \FuninL(\LMod_A(\scrS),\hsp\scrC)&\simeq \scrC\otimes\FuninL(\LMod_A(\scrS),\hsp\scrS)\\&\simeq\scrC\otimes\RMod_A(\scrS),
    \end{align*}which are all natural in $\scrC$. So, proving the naturality in $\scrC$ of the equivalence
    \[\scrC\otimes\RMod_A(\scrS)\simeq\RMod_A(\scrC)\]will yield that also the equivalence\[\FuninL(\LMod_A(\scrS),\hsp\scrC)\simeq\RMod_A(\scrC)\]is natural in $\scrC$. But under the (natural) equivalence $\scrC\otimes\RMod_A(\scrS)\simeq\FuninL(\LMod_A(\scrS),\hsp\scrC)$, the functor
    \[
\RMod_A(\scrC)\longrightarrow\scrC\otimes\RMod_A(\scrS)\simeq\FuninL(\LMod_A(\scrS),\hsp\scrC)
    \]
    is readily seen to be the functor obtained via adjunction from the functor
    \[
\LMod_A(\scrS)\otimes\RMod_A(\scrC)\longrightarrow\scrS\otimes\scrC\simeq\scrC
    \]
    given by the tensor product of the obvious forgetful functors. Since forgetting the action of a monoid $A$ in $\scrS$ is functorial with respect to colimit-preserving functors (which are obviously $\scrS$-linear), we can conclude in virtue of the naturality of the correspondence between adjoint morphisms.


\end{proof}
 \begin{proposition}
 [Monodromy equivalence, second take]
 \label{monodromyeq}
Let $X$ be a connected space, and let $\scrC$ be a cocomplete (not necessarily presentable) category. Then there exists an equivalence of categories$$\LocSys(X;\scrC)\simeq \LMod_{\Omega_*X}{\lp\scrC\rp}.$$
\end{proposition}
\begin{proof}
When $\scrA\coloneqq\scrS$, \cref{prop:stefanichcompgen} guarantees that $\CatVinftycocom$ is $\kappa_0$-compactly generated under large colimits by presentable categories, where $\kappa_0$ is the smallest large cardinal for our theory. So, for any cocomplete category $\scrC$ we can choose a presentation as a large colimit of presentable categories $\scrC_{i}$. Using \cref{lemma:modprL} and the fact that small spaces are compact with respect to large colimits in virtue of \cite[Proposition 5.4.1.2]{htt}, we obtain that$$\LocSys(X;\scrC)\simeq\colim_i \LocSys\lp X;\scrC_i\rp\simeq\colim_i\LMod_{\Omega_*X}(\scrC_i)\simeq\LMod_{\Omega_*X}(\scrC)$$
\end{proof}
\begin{corollary}
	\label{cor:koszulpresentable}
Let  $X$ be a connected space. Then there is an equivalence of categories
$$\LocSysCat(X)\simeq \LMod_{\Omega_*X}{\lp\PrLU\rp}.$$
More generally, if $\scrA$ is a presentably symmetric monoidal category, there is an equivalence of categories
$$\LocSysCat(X;\scrA)\simeq \LMod_{\Omega_*X}{\lp\LinAPrLU\rp}.$$
\end{corollary}
\begin{proof}
Note that the second part of the statement specializes to the first when we take $\scrA\coloneqq\scrS$ to be the category of spaces. As $\LinAPrLU$ is cocomplete, the statement follows immediately  from \cref{monodromyeq}.
 \end{proof}

 \begin{parag}
	\label{parag:PrLfiltered}
	In fact, it is possible to prove \cref{cor:koszulpresentable} directly  without appealing to Stefanich's \cref{prop:stefanichcompgen}. For semplicity we shall  focus on the case where the category of coefficients in $\PrLU$. The point is that we can explicitly write $\PrLU$ as a colimit of presentable categories. For completeness, let us sketch the argument.

The first ingredient is given by \cite[Lemmas $5.3.2.9$ and $5.3.2.11$]{ha}, which we summarize here for the convenience of the reader. First, if we fix a regular cardinal $\kappa$, the large category $\PrLk$ of  $\kappa$-compactly generated presentable categories, together with left exact functors which preserve $\kappa$-compact objects, \textit{is} presentable and admits small colimits (which agree with small colimits in $\PrLU$). Moreover, the tensor product of two small presentable and $\kappa$-compactly generated categories $\scrC$ and $\scrD$ is again $\kappa$-compactly generated: indeed, it is generated by $\kappa$-filtered colimits of objects of the form $C\otimes D$, where $C$ and $D$ are $\kappa$-compact generators of $\scrC$ and $\scrD$ respectively. The restriction of the tensor product $\otimes\colon \PrLk\times\PrLk\to \PrLk$ commutes again with small colimits, hence \cite[Remark $4.2.1.34$]{ha} guarantees that $\PrLk$ is enriched over itself.

Note that $\PrLU$ may be regarded as the filtered colimit inside $\CatVinftycocom$ of the categories $\PrLk$'s, where the colimit ranges all over regular cardinals $\kappa$ which are $\scrV$-small. A priori, this is only a colimit in $\CatVinfty$; however, the discussion above guarantees that all the functors making up this diagram are cocontinuous functors between cocomplete (actually presentable) categories, hence the diagram lies in $\CatVinftycocom$. Since the inclusion $$\CatVinftycocom\subseteq\CatVinfty$$ preserves filtered colimits (\cite[Proposition $5.5.7.11$]{htt}), we can regard $\PrLU$ as the colimit of $\PrLk$'s in $\CatVinftycocom$, as we claimed. Then arguing exactly as in the proof of \cref{monodromyeq} we conclude that there is an equivalence
$$\LocSysCat(X)\simeq \LMod_{\Omega_*X}{\lp\PrLU\rp}.$$
 \end{parag}
\subsection{Naturality of the correspondence}
The statements of \cref{lemma:koszulduality},  \cref{monodromyeq} and \cref{cor:koszulpresentable} are not obviously  \textit{natural} in $X$, for $X$ a connected and pointed space.  In fact even making sense of naturality in this context requires some care. Indeed, the functor $$\LocSys(-;\scrC)\coloneqq{\Fun}{\lp-,\hsp\scrC\rp}\colon\scrS\longrightarrow\PrLU$$does not depend  on a choice of a pointing. On the other hand, the functor $\LMod_{\Omega_*(-)}(\scrC)$ makes sense only for pointed spaces. In particular, the domain of these two functors is, \textit{a priori}, quite different. 
We will explain how to get around these issues, and upgrade the monodromy equivalence of \cref{lemma:koszulduality} to a natural equivalence of functors. The main goal of this section is to prove the following.

\begin{proposition}
	\label{prop:koszulisnatural}
	Let $\scrS^{\geqslant1}_*$ be the category of connected pointed spaces. Let $\scrC$ be a cocomplete category. Then there is a natural equivalence of functors
	$$\LMod_{\Omega_*(-)}{\lp\scrC\rp}\simeq 	\restr{\LocSys(-;\scrC)}{\scrS^{\geqslant1}_*}\colon\scrS^{\geqslant1}_*\longrightarrow\CatVinftycocom$$  More generally, if $\scrA$ is a presentably symmetric monoidal category and  $\scrC$ is a cocomplete category which is cocompletely tensored over $\scrA$ (i.e., it is an $\scrA$-module in $\CatVinftycocom$), 
then there is a natural equivalence of functors $$\LMod_{\Omega_*(-)}{\lp\scrC\rp}\simeq 	\restr{\LocSys(-;\scrC)}{\scrS^{\geqslant1}_*}\colon\scrS^{\geqslant1}_*\longrightarrow \mathrm{Lin}_{\scrA}\CatVinftycocom.$$  
\end{proposition}
\cref{prop:koszulisnatural} will play an important role in the later sections of the article, since it will be a stepping stone in proving that the Day convolution monoidal structure on the category of functors over a connected topological monoid $G$ naturally corresponds to the relative tensor product monoidal structure over the $\Ebb_2$-monoid $\Omega_*G$ (\cref{prop:koszulismonoidal}).

\begin{parag}
The second part of \cref{prop:koszulisnatural} specializes to the first when we set  $\scrA\coloneqq\scrS$. For ease of exposition, we will limit ourselves to prove  \cref{prop:koszulisnatural} in this latter case; the general case is proved in the same way.

The proof of \cref{lemma:koszulduality} given in   \cite[Lemma $3.9$]{koszulduality} depends on a sequence of equivalences of categories. We will show that each of them is natural in $X$ in a series of Lemmas (\cref{lemma:equivalencepart1}, \cref{lemma:equivalencepart2},  \cref{lemma:equivalencepart3} and \cref{lemma:equivalencepart4}). This will show that \cref{prop:koszulisnatural} holds when $\scrC$ is presentable. We will then conclude that the statement holds  for an  arbitrary cocomplete category  $\scrC$ using \cref{prop:stefanichcompgen}.

Let $\scrC$ be a presentable category. 
Consider  the category $\LocSys(X;\scrC)$ of $\scrC$-valued local systems over $X$. The proof of \cref{lemma:koszulduality} depends on the following chain of equivalences: 
\begin{align}
\label{equivalences:chain}
\begin{split}
\LocSys(X;\scrC) &\coloneqq \Fun(X,\hsp\scrC)
\\& \stackrel{(a)}\simeq\FuninL{\lp \Fun(X^{\op},\hsp\scrS),\hsp\scrC\rp}
\\& \stackrel{(b)}={\FuninL}{\lp\LocSys{\lp X^{\op}\rp},\hsp\scrC\rp}\\& \stackrel{(c)}\simeq{\FuninL}{\lp\scrS_{/X},\hsp\scrC\rp}
\\&\stackrel{(d)}\simeq{\FuninL}{\lp\LMod_{\Omega_*X}{\lp\scrS\rp},\hsp\scrC\rp}
\\&\stackrel{(e)}\simeq \RMod_{\Omega_*X}(\scrC)
\\&\stackrel{(f)}\simeq \LMod_{\Omega_*X}(\scrC).
\end{split}
\end{align}
Next, let us explain why each of these equivalences holds. 
\begin{itemize}
\item  Equivalence $(a)$ follows from the Yoneda Lemma.
\item  Equality $(b)$ is definitional. 
\item Equivalence $(c)$ follows from the straightening/unstraightening construction of \cite{htt}, which gives precisely 
\begin{align}
\label{equivalence:straightening}
\LocSys(X^{\op})\coloneqq\Fun(X^{\op},\hsp\scrS)\simeq\scrS_{/X}.
\end{align}
\item Equivalence $(d)$ follows from the fact that when $X$ is pointed and connected, \cite[Remark $5.2.6.28$]{ha} yields an equivalence
\begin{align}
\scrS_{/X}\simeq\LMod_{\Omega_*X}(\scrS).
\end{align} 
\item Equivalence $(e)$ follows from the categorical Eilenberg-Watts theorem (\cite[Theorem $4.8.4.1$]{ha}).
\item Equivalence $(f)$ depends on the fact that $\Omega_*X$ is   a \textit{grouplike} topological monoid (\cite[Definition $5.2.6.2$]{ha}). In particular, it can be regarded as a group object in spaces (\cite[Remark $5.2.6.5$]{ha}), and the antipode map $\iota\colon\Omega_*X\xrightarrow{\simeq} \Omega_*X$ yields an equivalence between $\Omega_*X$ and $\Omega_*X^{\mathrm{rev}},$ where $\Omega_*X^{\mathrm{rev}}$ is the same underlying space as $\Omega_*X$ endowed with the reverse monoid structure (\cite[Remark $4.1.1.7$]{ha}). Hence, left and right modules over $\Omega_*X$ are the same because of \cite[Remark $4.6.3.2$]{ha}. \end{itemize}
\end{parag}
Let us add some comments on this last point. In \cite{koszulduality}, $\LMod_{\Omega_*X}{\lp\scrS\rp}$ and $\RMod_{\Omega_*X}{\lp\scrS\rp}$ are in fact used interchangeably. The assignment $$X\mapsto X^{\mathrm{rev}}$$ is a functor induced by the canonical involution of the operad $\mathrm{Assoc}\simeq\Ebb_1$ i.e., the equivalence between left and right modules over $\Omega_*X$ is \textit{natural} in $X$. Thus,  as in \cite{koszulduality}, we shall often abuse notations and blur the difference between $\LMod_{\Omega_*X}{\lp \scrS\rp}$ and $\RMod_{\Omega_*X}{\lp\scrS\rp}$. 

In order to establish \cref{prop:koszulisnatural}, we will analyze in turn each of the equivalences in the chain \eqref{equivalences:chain}, and show that that they are natural in 
$X$. By our previous discussion, the naturality of $(b)$ and $(f)$ is clear, so we will focus on the remaining four equivalences.

We start with equivalence $\eqref{equivalences:chain}.(a)$. The naturality of the Yoneda Lemma is well-known in  ordinary category theory. In the $\infinity$-categorical setting   it  was recently established in \cite{yonedaisnatural3}.
\begin{theorem}{\cite[Theorem $3.6$]{yonedaisnatural3}}
	\label{thm:yonedanatural}
Let $\scrA$ be a presentably monoidal category, and let $\operatorname{Lin}_{\scrA}\CatUinfty$ (resp. $\operatorname{Lin}_{\scrA}\CatVinfty$) be the category of small (resp. large) $\scrA$-enriched categories. The inclusion$$\LinAPrLU\longhookrightarrow\operatorname{Lin}_{\scrA}\CatVinfty$$admits a left adjoint relative to the inclusion $\operatorname{Lin}_{\scrA}\CatUinfty\subseteq\operatorname{Lin}_{\scrA}\CatVinfty$, given by taking the category of presheaves with values in $\scrA$. The partial unit is provided by the $\scrA$-enriched Yoneda embedding $\yo_{\scrA}\colon \scrC\to{\Fun}{\lp\scrC^{\op},\hsp\scrA\rp}$.
\end{theorem}
\cref{thm:yonedanatural} immediately implies the following.
\begin{lemma}[Equivalence $\eqref{equivalences:chain}.(a)$ is natural]
\label{lemma:equivalencepart1}
There exists a natural equivalence of functors$${\FuninL}{\lp \LocSys{\lp(-)^{\op}\rp};\hsp\scrC\rp}\simeq\LocSys{\lp(-)^{\op};\scrC\rp}\colon\scrS\subseteq\CatUinfty\longrightarrow\PrLU.$$
\end{lemma}
 \begin{lemma}[Equivalence $\eqref{equivalences:chain}.(c)$ is natural]
\label{lemma:equivalencepart2}
There is a natural equivalence of functors
$$
 {\FuninL}{\lp\LocSys{\lp (-)^{\op}\rp},\hsp\scrC\rp} \simeq {\FuninL}{\lp\scrS_{/(-)},\hsp\scrC\rp} \colon\scrS \longrightarrow\PrLU.
 $$
\end{lemma}
\begin{proof}
This is an immediate consequence of the straightening process underlying the Grothendieck construction of \cite{htt}; in particular, the naturality is a consequence of \cite[Proposition $2.2.1.1$]{htt}. 
\end{proof}
 We now consider equivalence $\eqref{equivalences:chain}.(d)$. The functoriality of the equivalence
 \[
\scrS_{/X}\simeq\LMod_{\Omega_*X}(\scrS)
 \]is actually already proved in \cite[Remark $5.2.6.28$]{ha}. Hence, after precomposition with the forgetful functor
$$
\scrS^{\geqslant1}_* \longrightarrow\scrS^{\geqslant 1}\subseteq\scrS
$$
we obtain the naturality statement we need.
\begin{lemma}[Equivalence $\eqref{equivalences:chain}.(d)$ is natural]
\label{lemma:equivalencepart3}
There is a natural equivalence of functors
$$ {\FuninL}{\lp\scrS_{/(-)},\hsp\scrC\rp} \simeq{\FuninL}{\lp\LMod_{\Omega_*(-)}{\lp\scrS\rp},\hsp\scrC\rp}
\colon\scrS^{\geqslant1}_*\longrightarrow \PrLU.$$ 
\end{lemma}

Finally, we are left to prove the following.
\begin{lemma}[Equivalence $\eqref{equivalences:chain}.(e)$ is natural]
\label{lemma:equivalencepart4}
There is a natural equivalence of functors
$$
{\FuninL}{\lp\LMod_{\Omega_*(-)}{\lp\scrS\rp},\hsp\scrC\rp}
\simeq \RMod_{\Omega_*(-)}(\scrC)\colon\scrS^{\geqslant1}_*\longrightarrow \PrLU.$$
\end{lemma}
\begin{proof}
We first assume that $\scrC$ is the category of spaces $\scrS$. In this case, there is a well defined functor$$\FuninL(-,\hsp\scrS)\circ\LMod_{\Omega_*(-)}{\lp\scrS\rp}\colon\scrS\longrightarrow\smash{\lp\PrLU\rp^{\op}}\simeq\PrRU\longrightarrow\PrLU$$which point-wise agrees with$$\RMod_{\Omega_*(-)}{\lp\scrS\rp}\colon\scrS\longrightarrow\PrLU.$$So, let $\sigma$ be an $n$-simplex in the category of spaces $\scrS$: the image of such $n$-simplex under ${\FuninL}{\lp\LMod_{\Omega_*(-)}{\lp\scrS\rp},\hsp\scrS\rp}$ produces an $n$-dimensional commutative diagram $\widehat{\sigma}$ of categories of left modules inside $\PrLU$. Each $1$-simplex$$\left\{f\colon X\to Y\right\}\subseteq \sigma$$becomes, as a $1$-simplex of $\widehat{\sigma}$, a cocontinuous functor between categories of right modules $$f_*\colon\RMod_{\Omega_*X}{\lp\scrS\rp}\longrightarrow\RMod_{\Omega_*Y}{\lp\scrS\rp}.$$Using again the categorical Eilenberg-Watts Theorem, any functor as above is canonically equivalent to taking the tensor product with some ${\lp\Omega_*X,\Omega_*Y\rp}$-bimodule over $\Omega_*X$: an immediate inspection shows that such bimodule is the pullback along the forgetful functor $$f^*\colon \LMod_{\Omega_*Y}{\lp\scrS\rp}\longrightarrow\LMod_{\Omega_*X}{\lp\scrS\rp}$$of $\Omega_*Y$, seen as a left $\Omega_*Y$-module. In particular, the functor $f_*$ is canonically equivalent to$$-\otimes_{\Omega_*X}\Omega_*Y\colon\RMod_{\Omega_*X}{\lp\scrS\rp}\longrightarrow\RMod_{\Omega_*Y}{\lp\scrS\rp}.$$Using the fact that the relative tensor product is associative up to \textit{canonical} homotopy (\cite[Section $4.4.3$]{ha}), it follows that the $n$-simplex $\widehat{\sigma}$ agrees naturally with the homotopy coherence witnessing the associativity of the relative tensor product; hence, we have the desired equivalence of functors$${\FuninL}{\lp\LMod_{\Omega_*(-)}{\lp\scrS\rp},\hsp\scrS\rp}\simeq\RMod_{\Omega_*(-)}{\lp\scrS\rp}.$$The case of a general presentable category $\scrC$ is implied by the case $\scrC=\scrS$ since both functors ${\FuninL}{\lp\LMod_{\Omega_*(-)}{\lp\scrS\rp},\hsp\scrC\rp}$ and $\RMod_{\Omega_*(-)}{\lp\scrC\rp}$ naturally agree with the composition of the functor$${\FuninL}{\lp\LMod_{\Omega_*(-)}{\lp\scrS\rp},\hsp\scrS\rp}\simeq\RMod_{\Omega_*(-)}{\lp\scrS\rp}\colon\scrS\longrightarrow\PrLU$$with the functor$$-\otimes\scrC\colon\PrLU\longrightarrow\PrLU$$where $\otimes$ denotes Lurie's tensor product of cocomplete categories. 
\end{proof}
\begin{proof}[Proof of \cref{prop:koszulisnatural}]
\cref{lemma:equivalencepart1,lemma:equivalencepart2,lemma:equivalencepart3,lemma:equivalencepart4} together  imply \cref{prop:koszulisnatural} when $\scrC$ is presentable. Arguing as in \cref{monodromyeq}, we deduce that the statement holds also when $\scrC$ is a general (not necessarily presentable) cocomplete category. The $\scrA$-linear case is proved in the same way. 
\end{proof}
\section{Categorical local systems and Teleman's topological actions}
\label{sec:categorifyingmodules}
\numberwithin{equation}{section}
\setcounter{subsection}{1}
\cref{monodromyeq} gives a description of  of local systems on a space $X$ with coefficients in a cocomplete category in terms of monodromy data. In this Section we refine this statement. We will show that when $X$ is simply connected, categorical local systems can be described in terms of \emph{higher monodromy}: i.e. in terms of appropriate actions of the iterated loop space $\Omega_*^2X$. This will allow us to revisit, from the perspective of $\infinity$-categories, an interesting proposal of Teleman on topological group actions on categories.

 In   \cite{teleman}, Teleman 
  argues that  the datum of a $G$-action on $\Bbbk$-linear differential graded category   $\scrC$ should be equivalent to a morphism of $\Ebb_2$-algebras
$$\mathrm{C}_{\bullet}(\Omega_*G;\Bbbk)\longrightarrow\mathrm
{HH}^{\bullet}(\scrC).$$Here the source is the $\Ebb_2$-algebra of chains on the based loop space $\Omega_*G$, endowed with its Pontrjagin product (we are assuming that $G$ is connected); while the target is the Hochschild cohomology of $\scrC$.  While motivating the plausibility of this statement via a couple of examples, Teleman does not actually propose a proof of it. However, in the context of $\infinity$-category theory there is a natural way to interpret the action of a group object $G$ on any cocomplete category $\scrC$. Indeed by \cite[Section $4.4.4$]{htt} every cocomplete category $\scrC$ is tensored over the category of spaces $\scrS$. For any topological group $G$  one can consider the category $\LMod_G(\scrC)$ of left $G$-modules in $\scrC$. We will  prove that when $G$ is connected the datum of such a $G$-action on $\scrC$, where $\scrC$ is a presentable category which is $\Bbbk$-linear over a ring spectrum $\Bbbk$, is indeed encoded equivalently as a map of $\mathbb{E}_2$-ring spectra
\begin{equation}
\label{topacthh}
\Sigma^{\scriptstyle\infty}_+\Omega_*G \wedge \Bbbk\longrightarrow \operatorname{HH}^{\bullet}(\mathscr{C}).
\end{equation}
When $\Bbbk$ is a ordinary commutative ring, this recovers precisely Teleman's statement.

The main result of this Section is \cref{cor:maincor}. 
It is the key ingredient in the proof of \eqref{topacthh}.    \cref{cor:maincor} refines  \cref{cor:koszulpresentable} by describing local systems of categories in terms of  higher  monodromy data. Namely, we show that if $X$ is   \emph{simply connected}, local systems of presentable $\Bbbk$-linear categories on $X$ can be  described as iterated modules over the $\Ebb_2$-algebra $\Sigma_+^{\scriptstyle\infty}(\Omega_*\Omega_*X)\wedge\Bbbk$.

The first part of this Section will be dedicated to the proof of \cref{cor:maincor}. This will require several preliminary steps, starting with a \emph{linearization} statement which we  prove next  (see \cref{lemma:rectificationofmodules} below). Then in the second part of the Section we will turn our attention to the comparison with Teleman's notion of topological action.

\begin{lemma}
	\label{lemma:rectificationofmodules}
	Let $\scrA$ be a cocompletely symmetric monoidal category. Let $\boldone_{\scrA}$ denote the unit for the monoidal structure on $\scrA$, and let $G$ be an $\Ebb_{k+1}$-monoid in spaces. Then there is an equivalence of $\Ebb_{k}$-monoidal categories$$\LMod_G{\lp\scrA\rp}\simeq\LMod_{G\otimes \boldone_{\scrA}}{\lp\scrA\rp}.$$
\end{lemma}
\begin{proof}
Since the monoidal unit for the symmetric monoidal structure on $\CatVinftycocom$ is the category of spaces $\scrS$, we have a chain of equivalences$$\CAlg{\lp\CatVinftycocom\rp}\simeq\CAlg{\lp\Mod_{\scrS}{\lp\CatVinftycocom\rp}\rp}\simeq\CAlg{\lp\CatVinftycocom\rp}_{\scrS/}$$where the second equivalence is provided by \cite[Corollary $3.4.1.7$]{ha}. In particular, there exists an essentially unique symmetric monoidal and colimit-preserving functor$$-\otimes\boldone_{\scrA}\colon\scrS\longrightarrow\scrA$$
 which is uniquely determined by the assignment $\left\{*\right\}\mapsto \boldone_{\scrA}$. It follows that, if $G$ is a topological $\Ebb_{k+1}$-monoid, then  $G\otimes\boldone_{\scrA}$ is an $\Ebb_{k+1}$-algebra object in $\scrA$.  This means that the (essentially unique) action of $\scrS$ on $\scrA$ is encoded in the functor $-\otimes\boldone_{\scrA}$. Hence for any topological $\Ebb_{k+1}$-monoid $G$ we have the desired equivalence of categories of left modules.

Let us explain next why this equivalence is   $\Ebb_k$-monoidal. This follows from the fact that the $\Ebb_k$-monoidal structures on $\LMod_G(\scrA)$ and $\LMod_{G\otimes \boldone_{\scrA}}{\lp\scrA\rp}$ are both induced by the $\Ebb_{k+1}$-monoid structures of $G$ and $G\otimes\boldone_{\scrA}$, respectively, via the symmetric monoidal functor$$\LMod\colon\Alg(\scrA)\longrightarrow\lp\LinAPrLU\rp_{\scrA/}.$$Since the $\Ebb_{k+1}$-monoid structure of $G\otimes\boldone_{\scrA}$ is in turn induced by the one on $G$ via the symmetric monoidal functor $-\otimes\boldone_{\scrA}$, it follows that the two $\Ebb_k$-monoidal structures on $\LMod_G(\scrA)$ and $\LMod_{G\otimes\boldone_{\scrA}}(\scrA)$  agree as we claimed, and this concludes the proof.
\end{proof}
\begin{remark}
	\label{remark:chainscochains}
	An important setting for \cref{lemma:rectificationofmodules} is when  
	 $$\scrA=\Mod_{\Bbbk}\coloneqq\Mod_{\Bbbk}(\Sp)$$ for some $\Einf$-ring spectrum $\Bbbk$. Then $\boldone_{\scrA}=\Bbbk$ and $\Omega_*G\otimes \Bbbk$ computes the $\Bbbk$-valued chains of  $\Omega_*G$ with coefficients in $\Bbbk$. Indeed, $\Mod_{\Bbbk}$ is  presentable and stable (\cite[Corollaries $4.2.3.7$ and $7.1.1.5$]{ha}), hence the essentially unique cocontinuous functor $\scrS\to\Mod_{\Bbbk}$ factors through the essentially unique cocontinuous functor $$-\wedge\Bbbk\colon\Sp\longrightarrow\Mod_{\Bbbk}.$$
	 In fact, the category of spectra $\Sp$ is initial among stable presentably symmetric monoidal categories (\cite[Proposition $4.8.2.18$]{ha}). In particular, for any space $X$ we have that $$X\otimes \boldone_{\scrA}= X\otimes \Bbbk\coloneqq \Sigma^{\scriptstyle\infty}_+X\wedge \Bbbk$$where $\Sigma^{\scriptstyle\infty}_+\colon\scrS\to\Sp$ is the  suspension spectrum functor. But $\Sigma^{\scriptstyle\infty}_+X\wedge E$ is precisely the spectrum computing the homology of $X$ with coefficients in the generalized homology theory $E$. Because of this, in the following, for any space $X$ and for any $\Einf$-ring $\Bbbk$ we shall write $\mathrm{C}_{\bullet}(X;\Bbbk)$ for the $\Bbbk$-module $X\otimes\Bbbk$. Its $\Bbbk$-linear dual, which computes the $\Bbbk$-linear \textit{cochains} of $X$, shall similarly  be denoted as $\operatorname{C}^{\bullet}(X;\Bbbk)$. When $\Bbbk$ is discrete, these objects agree with the usual $\Bbbk$-valued chain and cochain complexes of classical algebraic topology.
\end{remark}

\begin{corollary}
	\label{cor:rectificationofmodules}
	Let $\scrA$ be a presentably symmetric monoidal category, and let $X$ be a connected space. Then there is an equivalence of categories$$\LocSysCat(X;\scrA)\simeq\operatorname{Lin}_{\Omega_*X\otimes\scrA}\PrLU.$$
\end{corollary}
\begin{proof}
This follows from \cref{lemma:rectificationofmodules} since $\LinAPrLU$ is cocompletely symmetric monoidal.
\end{proof}
Our  next goal is to better understand $\Omega_*X\otimes\scrA$.  We find it convenient to study the general problem of describing   $G \otimes \scrA$ where $G$ is an $\Ebb_k$-monoid in spaces. Our results in this direction are \cref{prop:FunandTensor}  and \cref{prop:koszulismonoidal} below. This will be key to establish the main result of this Section, \cref{cor:maincor}. We start by stating a couple of Lemmas. 
 
 \begin{lemma}
	\label{lemma:externalduality}
Let $\scrC$ be a symmetric monoidal category which is additionally both tensored and cotensored over a symmetric monoidal category $\scrA$ (in the sense of \cite[$4.2.1.28$]{ha}). Then the cotensor functor $\boldone_{\scrC}^{(-)}\colon \scrA^{\op}\to\scrC$ is lax monoidal.
\end{lemma}
\begin{proof}
	For $A$ and $C$ objects in $\scrA$ and $\scrC$ respectively, let us denote by $\left\{A\right\}\otimes C$ the object obtained from $A$ and $C$ via the tensor action of $\scrA$ over $\scrC$. Then, the contravariant bifunctor$$\Map_{\scrC}{\lp\left\{-\right\} \otimes -,\hsp\boldone_{\scrC}\rp}\colon\scrA^{\op}\otimes\scrC^{\op}\longrightarrow\scrS$$is classified by a pairing $\scrM\to\scrA\times \scrC$, which is left representable (in the sense of \cite[Definition $5.2.1.8$]{ha}). Indeed, by the very definition of \textit{cotensored category}, for every object $A$ in $\scrA$ the functor$$\Map_{\scrC}{\lp \left\{A\right\} \otimes -,\hsp\boldone_{\scrC}\rp}\colon\scrC^{\op}\longrightarrow\scrS$$is represented \textit{precisely} by $\boldone_{\scrC}^{A}$. So, the duality map $\boldone_{\scrC}^{(-)}\colon\scrA^{\op}\to\scrC$ is lax monoidal by \cite[Remark $5.2.2.25$]{ha}.
\end{proof}

The following useful observations have already been established in existing literature (see for example \cite{carmeli2023characterstransfermapscategorified,gammage2023perverse,hedenlund2023twisted}). We still provide proofs for the convenience of the reader.

	\begin{lemma}[{\cite[Theorem $3.2$]{gammage2023perverse}}]
\label{lemma:limitsandcolimitsovergroupoids}
		Let $\scrA$ be a presentably symmetric monoidal category, let $X$ be a space, and let $F\colon X\to\LinAPrLU$ be a diagram of shape $X$. Then, there is a natural equivalence $\lim F\simeq\colim F$ in $\LinAPrLU$.
	\end{lemma}
	\begin{proof}[Sketch of proof]
		Limits in $\LinAPrLU$ are computed as in $\PrLU$ (\cite[Corollary $4.2.3.3$]{ha}); and these in turn  are computed  as in $\CatVinfty$. On the other hand, colimits in $\LinAPrLU$ are computed as in $\PrLU$, because it is a cocompletely symmetric monoidal category, hence we can use \cite[Corollary $4.2.3.5$]{ha}. However, colimits in $\PrLU$ do not agree with colimits $\CatVinfty$: rather, they agree with \textit{limits} in $\CatVinfty$ \textit{after passing to the diagram of right adjoints}. Since any equivalence can be promoted to an \textit{adjoint equivalence} (\cite[Proposition $2.1.12$]{riehlverity}), and since $X$ is a groupoid, it follows that the "adjoint diagram" $F^{\op}$ is equivalent to $F$ itself, so the limit and the colimit coincide.
	\end{proof}
\begin{remark}
During the final stages of preparation of this paper,  \cref{lemma:limitsandcolimitsovergroupoids} was further generalized in \cite{benmoshe2024categoricalambidexterity}. The author shows that in fact the statement already holds in $\CatVinftycocom$. Even if this probably allows to harmlessly generalize our arguments and results to the cocomplete setting, we do not investigate this direction in the present work.
\end{remark}

 \begin{proposition}[{\cite[Corollary $4.12$]{carmeli2023characterstransfermapscategorified}}]
\label{prop:FunandTensor}
Let $\scrA$ be a presentably symmetric monoidal category and let $G$ be an $\Ebb_k$-monoid in spaces. Then we have an equivalence of $\scrA$-enriched presentably $\Ebb_k$-monoidal categories
$$ G\otimes\scrA \simeq \LocSys(G;\scrA).$$
\end{proposition}
\begin{proof} 
Recall that $G\otimes\scrA$ is  the image of $G$ under the unique symmetric monoidal and colimit-preserving functor $\scrS\to\scrA$ determined by the  assignment 
$$\left\{*\right\}\mapsto\boldone_{\scrA}.$$ 
We will show that $\LocSys(-;\scrA)$ is also a symmetric monoidal and colimit-preserving functor mapping 
$\left\{*\right\}$ to $\boldone_{\scrA}$. 
 This implies that there is a  canonical  equivalence of functors 
$$-\otimes\scrA \simeq \LocSys(-;\scrA)$$
and so in particular proves the claim.

Note that the functor $\LocSys(-;\scrA)$ can be interpreted as the natural copowering functor$$\scrA^{(-)}\colon\scrS^{\op}\longrightarrow\LinAPrLU,$$ 
This immediately shows that it satisfies the condition that $\left\{*\right\}\mapsto\boldone_{\scrA}$.
Next, since the category $\LinAPrLU$ is a cocomplete closed symmetric monoidal category which is both (cocompletely) tensored and cotensored over spaces, we deduce that $\LocSys(-;\scrA)$ is a lax monoidal functor thanks to \cref{lemma:externalduality}. It remains to show that the functor $\LocSys(-;\scrA)$ preserves colimits, and that for any spaces $X$ and $Y$ the natural map $$\alpha_{XY}\colon\LocSys(X;\scrA)\otimes_{\scrA}\LocSys(Y;\scrA)\longrightarrow\LocSys(X\times Y;\scrA)$$is an equivalence. 	Using \cref{lemma:limitsandcolimitsovergroupoids}, we can now prove the following statements.
	\begin{enumerate}
\item The functor $\LocSys(-;\scrA)$ is cocontinuous. Indeed, let $Y$ be the colimit of a diagram $I\to\scrS$ with values in the category of spaces. Recall that every space $X$ is a colimit of a diagram of shape $X$ itself with constant valute at the point $\left\{*\right\}$. Obviously, we have a natural equivalence of $\scrA$-linear presentable categories$$Y\otimes\scrA\coloneqq\colim_Y\scrA\simeq\colim_{i\in I}\colim_{X_i} \scrA.$$We have also natural equivalences of $\scrA$-linear presentable categories\begin{align*}
\LocSys(Y;\scrA)&\simeq\LocSys{\lp\colim_Y\hsp \left\{*\right\};\scrA\rp}\\&\simeq \lim_Y\LocSys(\left\{*\right\};\scrA)\\&\overset{\ref{lemma:limitsandcolimitsovergroupoids}}{\simeq}\colim_Y \LocSys(\left\{*\right\};\scrA)\simeq\colim_Y\scrA.
	\end{align*}
 Analogously,$$\LocSys(X_i;\scrA)\simeq \smash{\colim_{X_i}\scrA},$$so we can conclude that $\LocSys(Y;\scrA)\simeq\smash{\colim_i}\LocSys(X_i;\scrA)$.
 \item The functor $\LocSys(-;\scrA)$ is strongly monoidal. Consider two spaces $X$ and $Y$, and let us again present each as a colimit of a constant diagram whose value is the point. Combining the compatibility of the monoidal structure of $\LinAPrLU$ with colimits and \cref{lemma:limitsandcolimitsovergroupoids}, we immediately see that the map $\alpha_{XY}$ above boils down to the natural equivalence$$\colim_X\colim_Y\scrA\overset{\simeq}{\longrightarrow}\colim_{X\times Y}\scrA$$provided by the \textit{Fubini theorem for homotopy colimits} (see for example \cite[Theorem $24.9$]{fubini}).
\end{enumerate}
\end{proof}
 \begin{remark}
	\label{porism:keyporism}
	It is implicit in the statement of the \cref{prop:FunandTensor} that $\LocSys(G;\scrA)$ 
	carries a natural $\Ebb_k$-monoidal structure.  This is clarified by the proof of \cref{prop:FunandTensor}. Indeed, we show that the functor $\LocSys(-;\scrA)$ is strongly monoidal. Hence, in particular, it preserves $\Ebb_k$-monoids. We remark that the resulting $\Ebb_k$-monoidal structure  on $\LocSys(G;\scrA)$ agrees with (the reverse of) the $\Ebb_k$-monoidal Day convolution product of \cite[Remark $2.2.6.8$]{ha}.
\end{remark}
\begin{proposition}
	\label{prop:koszulismonoidal}
	Let $\scrA$ be a presentably symmetric monoidal category. For any $k\geqslant 1$, we have a commutative diagram of functors$$\begin{tikzpicture}[scale=0.75]
	\node(a1)at(-4,2){$\Alg_{\Ebb_k}{\lp\scrS^{\geqslant1}\rp}$};\node(b1)at(3,2){$\Alg_{\Ebb_k}{\lp\LinAPrLU\rp}$};
	\node(a2)at(-4,0){$\scrS^{\geqslant1}_*$};\node(b2)at(3,0){${\lp\LinAPrLU\rp}_{\scrA/}$};
	\node(a3)at (-4,-2){$\scrS^{\geqslant1}$};\node(c) at(-1.5,-2){$\scrS$};\node(b3)at (3,-2){$\LinAPrLU.$};
	\draw[->,font=\scriptsize](a1)to node[above]{}(b1);
	\draw[->,font=\scriptsize]	(a2)to node[above]{$\LMod_{\Omega_*(-)}{\lp\scrA\rp}$}(b2);
	\draw[->,font=\scriptsize]	(c)to node[above]{$\LocSys(-;\scrA)$}(b3);
	\draw[->,font=\scriptsize](a1)to node[left]{$\oblv_{\Ebb_k}$}(a2);
	\draw[->,font=\scriptsize](a2)to node[left]{$\oblv_*$}(a3);
	\draw[->,font=\scriptsize](b1)to node[right]{$\oblv_{\Ebb_k}$}(b2);
	\draw[->,font=\scriptsize](b2)to node[right]{$\oblv_*$}(b3);
	\draw[right hook->,font=\scriptsize](a3) to node[above]{} (c);
	\end{tikzpicture}$$In particular, for $G$ a connected $\Ebb_k$-monoid in spaces, there is an equivalence of $\scrA$-enriched presentably $\Ebb_k$-monoidal categories
	$$\LocSys(G;\scrA)\simeq\LMod_{\Omega_*G}{\lp\scrA\rp}$$which is natural in $G$ and agrees with the equivalence of \cref{lemma:koszulduality}.
\end{proposition}
\begin{proof}
	\cref{prop:FunandTensor} and \cref{porism:keyporism} together imply that whenever $X$ is a pointed and connected space, we can reinterpret the natural equivalence $\LocSys(X;\scrA)\simeq\LMod_{\Omega_*X}{\lp\scrA\rp}$ of \cref{prop:koszulisnatural} as follows. Notice that the natural functor
	$$\restr{\LocSys(-;\scrA)}{\scrS^{\geqslant 1}}\colon\scrS^{\geqslant1}\subseteq\scrS\longrightarrow\LinAPrLU$$is again strongly monoidal. Here, $\scrS^{\geqslant1}$ is seen as a Cartesian symmetric monoidal category thanks to the fact that finite products of connected spaces are again connected (\cite[Corollary $6.5.1.13$]{htt}); in particular, the inclusion $\scrS^{\geqslant1}\subseteq\scrS$ is a strongly monoidal functor. Therefore, $\restr{\LocSys(-;\scrA)}{\scrS^{\geqslant 1}}$ induces a functor at the level of $\Ebb_k$-algebras for every $k \geqslant 0$, that we again denote by$${\LocSys(-;\scrA)}\colon\Alg_{\Ebb_k}{\lp\scrS^{\geqslant1}\rp}\longrightarrow\Alg_{\Ebb_k}{\lp\LinAPrLU\rp}.$$For $k=0$, this is merely a functor from $\scrS^{\geqslant1}_*$ to $\smash{\lp\LinAPrLU\rp}_{\scrA/}$, since $\Ebb_0$-algebras in a monoidal category are simply objects pointed by the monoidal unit, without further requirements (\cite[Proposition $2.1.3.9$]{ha}). In particular, \cref{prop:koszulisnatural} simply states that such functor agrees with
 \begin{equation}
 \label{functor:modulesoverloopspace}
    \LMod_{\Omega_*(-)}{\lp\scrA\rp}\colon\scrS_*^{\geqslant1}\simeq\Alg^{\mathrm{grp}}_{\mathbb{E}_1}{\lp\scrS\rp}\subseteq\Alg_{\Ebb_1}{\lp\scrS\rp}\longrightarrow\smash{\lp\LinAPrLU\rp}_{\scrA/},
 \end{equation}where the first equivalence is given by May's delooping theorem (\cite[Theorem $5.2.6.10$]{ha}). Endowing the category of $\Ebb_k$-algebras in connected spaces with its natural Cartesian monoidal structure (i.e., with the monoidal structure provided by the underlying tensor product inside $\scrS^{\geqslant1}_*$) and endowing the category $\smash{\lp\LinAPrLU\rp}_{\scrA/}$ with its natural symmetric monoidal structure given by the relative tensor product of presentable categories over $\scrA$, it follows that the functor \eqref{functor:modulesoverloopspace} is again strongly monoidal. Indeed, it is a composition of strongly monoidal functors: this follows from the fact that, given any presentably symmetric monoidal category $\scrC$ which is presentably tensored over a presentably symmetric monoidal category $\scrA$, for any associative algebra $A$ in $\scrA$ the assignation $A \mapsto \LMod_A(\scrC)$ is strongly monoidal (\cite[Theorem $4.8.5.16$]{ha}). Since $$\Alg_{\Ebb_k}{\lp\Alg_{\Ebb_0}(\scrC)\rp}\simeq\Alg_{\Ebb_k}{\lp\scrC\rp}$$in virtue of Dunn's Additivity Theorem (\cite[Theorem $5.1.2.2$]{ha}), we deduce that indeed the diagram of functors pictured above exists and is commutative, deducing our claim.
\end{proof}
\begin{corollary}
\label{cor:maincor}
Let $\scrA$ be a presentably symmetric monoidal category, and let $X$ be a simply connected space. Then there are equivalences of categories$$\LocSysCat(X;\scrA)\simeq\LMod_{\Omega_*X}{\lp\LinAPrLU\rp}\simeq\operatorname{Lin}_{\LMod_{\Omega^2_*X}(\scrA)}\PrLU.$$
\end{corollary}
\begin{proof} By  \cref{lemma:rectificationofmodules} and \cref{prop:FunandTensor}, we have an equivalence of categories $$\LMod_{\Omega_*X\otimes \scrA}{\lp\LinAPrLU\rp}\simeq\operatorname{Lin}_{\LocSys{\lp\Omega_*X;\scrA\rp}}\PrLU.$$
Since $X$ is simply connected, $\Omega_*X$ is connected. It follows from \cref{prop:koszulismonoidal} that we have an equivalence of $\Ebb_1$-monoidal categories $$\LocSys{\lp\Omega_*X;\scrA\rp}\simeq\LMod_{\Omega^2_*X}(\scrA)$$which implies an equivalence between their categories of left modules in $\PrLU$. Our statement follows by combining these results with the equivalence provided by \cref{cor:rectificationofmodules}.
\end{proof}
\begin{parag}
\label{parag:hochschild}
In the last paragraph of this section, we  explain the connection between \cref{cor:maincor} and Teleman's notion of topological action from \cite{teleman}. We start by recalling the definition of the Hochschild cohomology of a presentable category $\scrC$ enriched over some presentably symmetric monoidal category $\scrA$. We follow the construction presented in \cite{iwanari}, which is obtained by  combining  various results from \cite[Sections $4.8.5$ and $5.3.2$]{ha}.

Let $\scrA$ be a presentably symmetric monoidal category. By \cite[Theorem $4.8.5.5$]{ha} we have a fully faithful functor
$$\LMod_{(-)}\colon\Alg{\lp\scrA\rp}\longrightarrow{\lp\LinAPrLU\rp}_{\scrA/}$$sending an associative algebra $A$ in $\scrA$ to the category of its left modules $\LMod_A{\lp\scrA\rp}$, with the pointing $\scrA\to\LMod_A{\lp\scrA\rp}$ given by the essentially unique colimit-preserving functor sending $\boldone_{\scrA}$ to $A$. \begin{remark}
In \cite{iwanari}  the author considers  instead the  functor $$\RMod_{(-)}\colon\Alg{\lp\scrA\rp}\longrightarrow{\lp\LinAPrLU\rp}_{\scrA/}.$$ 
 This discrepancy however does not impact the present discussion, as $\LMod_{(-)}$ can be obtained  from $\RMod_{(-)}$ by precomposing with the involution of $\Alg(\scrA)$ sending an associative algebra to its opposite algebra (\cite[Remark $4.1.1.7$]{ha}).
 \end{remark} 
 As we recalled in the proof of \cref{prop:koszulismonoidal}, the functor $\LMod_{(-)}$ is symmetric monoidal, so we can promote it to a functor between categories of $\Ebb_1$-algebras: $$\LMod_{(-)}\colon\Alg{\lp\Alg{\lp\scrA\rp}\rp}\simeq{{\Alg_{\Ebb_2}}{\lp\scrA\rp}}\longrightarrow\Alg{\lp{\lp\LinAPrLU\rp}_{\scrA/}\rp}\simeq\Alg{\lp \LinAPrLU\rp},$$where we used again Dunn's Additivity together with the fact that objects pointed by the unit in any monoidal category $\scrC$ are the same as $\Ebb_0$-algebras in $\scrC$. By the general machinery of \cite[Proposition $2.2.1.1$]{ha}, it follows that the functor $\LMod_{(-)}$ admits a right adjoint$$\Phi\colon\Alg{\lp\LinAPrLU\rp}\longrightarrow{{\Alg_{\Ebb_2}}{\lp\scrA\rp}}$$ that sends  a presentably monoidal and $\scrA$-enriched category $\scrC$ to the $\Ebb_2$-algebra of endomorphisms $\End_{\scrC}{\lp\boldone_{\scrC}\rp}$ in $\scrA$ (\cite[Remark $4.8.5.12$]{ha}).
\end{parag}
\begin{defn}
\label{def:hochschild}
Let $\scrA$ be a presentably symmetric monoidal category, and let $\scrC$ be a presentable category enriched over $\scrA$. Let $\Endin(\scrC)$ be the endomorphism category of $\scrC$ in $\LinAPrLU$ in the sense of \cite[Section $4.7.1$]{ha} -- i.e., it is the presentably $\scrA$-linear category $\FuninL_{\scrA}{\lp\scrC,\hsp\scrC\rp}$, seen as a monoidal category via the composition of functors. Then the \textit{Hochschild cohomology} of $\scrC$ is the $\Ebb_2$-algebra in $\scrA$ 
$$\HH^{\bullet}(\scrC)\coloneqq\Phi{\lp\Endin(\scrC)\rp}.$$
\end{defn}
\begin{proposition}
	\label{lemma:hochschildand2modules}
	Let $\scrA$ be a presentably symmetric monoidal category, let $\scrC$ be a presentable category enriched over $\scrA$, and let $G$ be a connected topological group. Let $\LMod_{\Omega_*G\otimes\boldone_{\scrA}}(\scrA)\text{-}\operatorname{ModStr}(\scrC)$ denote the space of all possible left $\LMod_{\Omega_*G\otimes \boldone_{\scrA}}{\lp\scrA\rp}$-module structures on $\scrC$, and let $G\text{-}\operatorname{ModStr}(\scrC)$ denote the space of all possible left $G$-module structures on $\scrC$.\\
	Then, there are equivalences of spaces
	$$
	\Map_{\Alg_{\Ebb_2}(\scrA)}{\lp\Omega_*G\otimes \boldone_{\scrA},\hsp \HH^{\bullet}(\scrC)\rp}
	\simeq 
	\LMod_{\Omega_*G\otimes\boldone_{\scrA}}(\scrA)\text{-}\operatorname{ModStr}(\scrC)
	 \simeq  
	  G\text{-}\operatorname{ModStr}(\scrC).
	$$
\end{proposition}
\begin{proof}
Let us start from the equivalence$$\Map_{\Alg_{\Ebb_2}(\scrA)}{\lp\Omega_*G\otimes \boldone_{\scrA},\hsp \HH^{\bullet}(\scrC)\rp}
	\simeq 
	\LMod_{\Omega_*G\otimes\boldone_{\scrA}}(\scrA)\text{-}\operatorname{ModStr}(\scrC).$$The adjunction between 
$ 
\LMod_{(-)}$  and $\Phi
$ yields an equivalence of spaces$$\Map_{\Alg_{\Ebb_2}(\scrA)}{\lp\Omega_*G\otimes \boldone_{\scrA},\hsp \HH^{\bullet}(\scrC)\rp}\simeq \Map_{\Alg{\lp\LinAPrLUtiny\rp}}{\lp\LMod_{\Omega_*G\otimes \boldone_{\scrA}}{\lp\scrA\rp},\hsp\Endin(\scrC)\rp}.$$But the right hand side is equivalent to the space of left $\LMod_{\Omega_*G\otimes \boldone_{\scrA}}{\lp\scrA\rp}$-module structures on $\scrC$, in virtue of \cite[Corollary $4.7.1.41$]{ha}. Then, the equivalence$$\LMod_{\Omega_*G\otimes\boldone_{\scrA}}(\scrA)\text{-}\operatorname{ModStr}(\scrC)
	 \simeq  
	  G\text{-}\operatorname{ModStr}(\scrC)$$follows immediately from the second equivalence in \cref{cor:maincor} applied to the simply connected space $X\coloneqq\mathbf{B}G$.
\end{proof}
\cref{lemma:hochschildand2modules} states that, for any presentably symmetric monoidal category $\scrA$ and any presentable category $\scrC$ enriched over $\scrA$, giving $\scrC$ an action of the topological monoid $G$ on $\scrC$ is equivalent to giving $\scrC$ a $\LMod_{\Omega_*G\otimes\boldone_{\scrA}}{\lp\scrA\rp}$-module structure; in turn, this is equivalent to providing an $\Ebb_2$-algebra map $$\Omega_*G\otimes\boldone_{\scrA}\longrightarrow\operatorname{HH}^{\bullet}(\scrC).$$This can be interpreted as a generalization, rephrased in purely $\infinity$-categorical terms, of the following result due to Teleman.
\begin{theorem}[{\cite[Theorem $2.5$]{teleman}}]
\label{thm:teleman}
Topological actions of a connected group $G$ on a differential graded category $\scrC$ which is linear over some base commutative ring $\Bbbk$ are completely captured, up to contractible choices, by the induced $\Ebb_2$-algebra morphisms$$\mathrm{C}_{\bullet}{\lp\Omega_*G;\Bbbk\rp}\longrightarrow\operatorname{HH}^{\bullet}(\scrC)$$where the source is simply the algebra of chains of $\Omega_*G$ with coefficients in $\Bbbk$, endowed with the Pontrjagin product, and the target is the Hochschild cohomology of the differential graded category $\scrC$.
\end{theorem}
Indeed, when $\scrA\coloneqq\Mod_{\Bbbk}$ is the presentable category of $\Bbbk$-modules over a classical commutative ring $\Bbbk$, we have already seen that the $\Ebb_2$-algebra $\Omega_*G\otimes\boldone_{\scrA}$ boils down to the $\Ebb_2$-algebra of $\Bbbk$-chains on $\Omega_*G$ (\cref{remark:chainscochains}). Since differential graded $\Bbbk$-linear categories are the same as compactly generated $\Bbbk$-linear presentable categories up to Morita equivalence (\cite[Corollary $5.7$]{cohn}), \cref{lemma:hochschildand2modules} implies \cref{thm:teleman}.\\

We include for completion also a neat characterization of invertible objects inside the symmetric monoidal category $\LocSysCat(X;\Bbbk)$, when $X$ is assumed to be simply connected and $\Bbbk$ to be an algebraically closed field. We start with the following easy remark.
\begin{proposition}
\label{prop:invertibleinlocsyscat}
Let $X$ be a connected space, and let $\eta\colon \left\{*\right\}\to X$ be any choice of a base point. Let $\Bbbk$ be any commutative ring spectrum. Then an object $\scrF$ inside $\LocSysCat(X;\Bbbk)$ is invertible if and only if its stalk at the base point $\scrF_{\eta}$ is invertible in $\LinkPrLU$.
\end{proposition}
\begin{proof}
Recall that an object in a monoidal category $\scrC^{\otimes}$ is invertible if it is fully dualizable and both the evaluation and the coevaluation map are equivalences. In virtue of \cite[Lemma $1.4.6$]{1affineness}, the fully dualizable objects inside the symmetric monoidal category
\[
\LocSysCat(X;\Bbbk)\simeq\lim_{x\to X}\LinkPrLU
\]
are precisely the objects which are fully dualizable when projecting to each copy of $\LinkPrLU$. Under the above equivalence, the projection corresponds to taking the stalk at a point $x\to X$. Since $X$ is connected, a local system of categories $\scrF$ over $X$ is fully dualizable if and only if the stalk $\scrF_{\eta}$ is fully dualizable as a presentably $\Bbbk$-linear category. Moreover, since $\eta^\ast$ is symmetric monoidal, we know that for a fully dualizable local system of categories $\scrF$ the dual in $\LinkPrLU$ of the stalk $\scrF_{\eta}$ is the stalk at $\eta$ of the dual $\scrF^{\vee}$ in $\LocSysCat(X;\Bbbk)$.

Thus, we are left to prove that the evaluation and the coevaluation morphisms that testify the dualizability of $\scrF$ are equivalences if and only if the functors induced at the stalk $\eta$ are equivalences. Since the pullback along $\eta$ is functorial, the "only if" direction is obvious. On the other hand, since $X$ is connected, the functor $\eta^\ast$ is conservative (it corresponds to forgetting the $\Omega_*X$-action, under the equivalence of \cref{cor:koszulpresentable}), hence we deduce also the "if" direction.
\end{proof}
\begin{parag}
When $X$ is connected and $\Bbbk$ is a commutative ring spectrum, \cref{prop:invertibleinlocsyscat} implies that an object in the subgroupoid
\[
\lp\LocSysCat(X;\Bbbk)^{\mathrm{inv}}\rp^{\simeq}\subseteq\LocSysCat(X;\Bbbk)^{\simeq}
\]
spanned by all invertible local systems of $\Bbbk$-linear categories on $X$ consists of the datum of an invertible presentably $\Bbbk$-linear category $\scrC$ together with a $\Omega_*X$-action on $\scrC$. In virtue of \cite[Theorem $3.15$ and Proposition $7.3$]{antieau_gepner_brauer_groups}, we deduce that the connected components of $\smash{\lp\LocSysCat(X;\Bbbk)^{\mathrm{inv}}\rp^{\simeq}}$ are equivalently described as classes in the Brauer group of $\Bbbk$
\[
\mathrm{Br}(\Bbbk)\simeq\pi_0{\lp\lp\LinkPrLUinv\rp^{\simeq}\rp}
\]
together with all possible $\Omega_*X$-actions over each of them.
\end{parag}
\begin{parag}
\label{parag:brauersimply}
When $X$ is moreover \textit{simply} connected we can apply the machinery of \cref{cor:maincor} and of \cref{thm:teleman} to deduce that an invertible local system of $\Bbbk$-linear categories on $X$ consists of the datum of an equivalence class in the Brauer group $[\scrC]\simeq[\Mod_A]\in\mathrm{Br}(\Bbbk)$, where $A$ is an Azumaya algebra over $\Bbbk$, with a morphism of $\Ebb_2$-algebras $\mathrm{C}_{\bullet}(\Omega_*^2X;\Bbbk)\to\mathrm{HH}^{\bullet}(\scrC)$. In particular, suppose that the Brauer group $\mathrm{Br}(\Bbbk)$ is trivial. This happens for every algebraically closed field (\cite[Proposition $1.9$]{toen_azumaya}) and for every commutative ring spectrum whose $\pi_0$ is either $\ZZ$ or the ring of Witt vectors $\mathbb{W}_p$ over $\mathbb{F}_p$ (\cite[Theorem $7.16$]{antieau_gepner_brauer_groups}); in particular, this holds also for the sphere spectrum. Then the invertible objects of $\LocSysCat(X;\Bbbk)$ consists of all possible $\Omega_*X$-action on the category of modules over the essentially unique Azumaya algebra over $\Bbbk$ up to Morita equivalence -- that is, $\Bbbk$ itself. Together with \cref{lemma:hochschildand2modules}, we obtain the following.
\end{parag}
\begin{proposition}
\label{prop:invertibleobjectsoflocsyscat}
Let $X$ be a simply connected space and let $\Bbbk$ be an algebraically closed field. Then we have an isomorphism of abstract groups
\[
\pi_0{\lp\lp\LocSysCat(X;\Bbbk)^{\mathrm{inv}}\rp^{\simeq}\rp}\cong\Hom_{\mathrm{Grp}}{\lp \pi_2(X),\hsp \Bbbk^{\times}\rp}
\]
between the group of equivalence classes of invertible local systems of $\Bbbk$-linear categories on $X$, and the group of multiplicative characters of $\pi_2(X)$.
\end{proposition}
\begin{proof}
In virtue of the discussion in Paragraph \ref{parag:brauersimply}, we only need to characterize the set of connected components of the space $\Map_{\Alg_{\Ebb_2}(\Mod_{\Bbbk})}{\lp\mathrm{C}_{\bullet}(\Omega^2_*X;\Bbbk),\hsp\mathrm{HH}^{\bullet}(\Bbbk)\rp}.$ Notice that the Hochschild cohomology of $\Mod_{\Bbbk}$  computes the ordinary Hochschild cohomology of $\Bbbk$, which is $$\operatorname{HH}^{\bullet}(\Bbbk)\coloneqq\Mapin_{\Bbbk\otimes\Bbbk}(\Bbbk,\Bbbk)\simeq\Bbbk.$$So we are left to study the mapping space as $\Ebb_2$-algebras from $\mathrm{C}_{\bullet}(\Omega^2_*X;\Bbbk)$ and $\Bbbk$. We claim that maps of $\Ebb_2$-algebras from a connective $\Ebb_2$-algebra $A$ to a discrete algebra $R$ over a field $\Bbbk$ always factor through maps of $\Ebb_2$-algebras from $\pi_0A$. Indeed, the adjunction
\[
\adjunction{\tau_{\leqslant0}}{\Mod_{\Bbbk}}
  {\lp\Mod_{\Bbbk}\rp_{\leqslant0}}{\iota_{\leqslant0}}
\]
restricts to an adjunction
\[
\adjunction{\tau^{\heartsuit}_{\leqslant0}}{\Mod_{\Bbbk}\bigcap\lp\Mod_{\Bbbk}\rp_{\geqslant0}=\lp\Mod_{\Bbbk}\rp_{\geqslant0}}{\lp\Mod_{\Bbbk}\rp_{\leqslant0}\bigcap\lp\Mod_{\Bbbk}\rp_{\geqslant0}=\Mod_{\Bbbk}^{\heartsuit}}{\iota_{\leqslant0}^{\heartsuit}}.
\]
The right adjoint is strongly monoidal, because over a field every object is flat; the left adjoint is strongly monoidal as well, because of K\"{u}nneth formula. So we can safely apply \cite[Corollary $7.3.2.12$ and Remark $7.3.2.13$]{ha} to deduce the existence of an adjunction
\[
\adjunction{\tau^{\heartsuit}_{\leqslant0}}{\Alg_{\scrO}\lp\lp\Mod_{\Bbbk}\rp_{\leqslant0}\rp}{\Alg_{\scrO}{\lp\Mod_{\Bbbk}^{\heartsuit}\rp}}{\iota_{\leqslant0}^{\heartsuit}},
\]
for any operad $\scrO$. In the case $\scrO=\Ebb_2$ we obtain
\begin{align*}
\Map_{\Alg_{\Ebb_2}(\Mod_{\Bbbk})}{\lp\mathrm{C}_{\bullet}(\Omega^2_*X;\Bbbk),\hsp\mathrm{HH}^{\bullet}(\Bbbk)\rp}&\simeq \Map_{\Alg_{\Ebb_2}\lp(\Mod_{\Bbbk})_{\geqslant0}\rp}{\lp\mathrm{C}_{\bullet}(\Omega^2_*X;\Bbbk),\hsp\Bbbk\rp}\\&\simeq\Map_{\Alg_{\Ebb_2}{\lp\Mod_{\Bbbk}^{\heartsuit}\rp}}{\lp\tau_{\leqslant0}\mathrm{C}_{\bullet}(\Omega^2_*X;\Bbbk),\hsp\Bbbk\rp}.
\end{align*}
Notice that $\Mod_{\Bbbk}^{\heartsuit}$ is the ordinary discrete category of $\Bbbk$-modules, so $\Ebb_2$-algebras are the same as commutative ($\Einf$-)algebras (\cite[Corollary $5.1.1.7$]{ha}). On the other hand, $\tau_{\leqslant0}\mathrm{C}_{\bullet}(\Omega^2_*X;\Bbbk)$ selects the $0$-th homology of $\mathrm{C}_{\bullet}(\Omega^2_*X;\Bbbk)$, which is
\[
\mathrm{H}_0(\Omega^2_*X;\Bbbk)\cong\Bbbk[\pi_0\Omega^2_*X]\cong\Bbbk[\pi_2(X)].
\]
So we are simply looking at the space of maps from $\Bbbk[\pi_2(X)]$ to $\Bbbk$ seen as discrete commutative algebras, which is a discrete set. In particular, using the group ring-group of units adjunction between the category of discrete algebras $\Alg_{\Bbbk}^{\mathrm{disc}}$ and the category of discrete groups $\mathrm{Grp}$, which restricts to the subcategories of commutative algebras and abelian groups, we have
\[
\Hom_{\CAlg_{\Bbbk}^{\mathrm{disc}}}{\lp\Bbbk[\pi_2(X)],\hsp\Bbbk\rp}\cong\Hom_{\mathrm{Grp}}{\lp\pi_2(X),\hsp\Bbbk^{\times}\rp}.
\]
\end{proof}

%% file: categorical_local_systems2.tex
\section{Local systems of higher categories} 
\label{sec:highermonodromydata}
\numberwithin{equation}{subsection}
In the previous sections we  presented several descriptions of  local systems of (enriched) categories in terms of  monodromy data, culminating in \cref{monodromyeq} and  \cref{cor:maincor}. Those results, however, were all formulated as equivalences of 1-categories, whereas local systems of categories naturally form   a 2-category. In \cref{sec:inftytwo}  we explain how to lift our results   
  to equivalences of appropriate 2-categories. This is the content of \cref{thm:infinitytwothm}. Our strategy can be summarized as follows: we shall prove that the equivalences of \cref{monodromyeq} and \cref{cor:maincor} can be regarded as equivalences between \textit{$\CatVinfty$-tensored categories}. Indeed, $n$-categories can be modeled by $\Theta_n$-spaces (\cite{thetaspace,thetaspaceerr}), which in turn are enriched over $\Theta_{n-1}$-spaces by  \cite{rezk1,rezk2}.  
  One of the key technical inputs in the proof will be provided by the theory of enriched $\infty$-categories developed by Gepner--Haugseng  \cite{gepnerhaugseng} and others.

Next, in \cref{sec:localsystemncat}, we will study  local systems of \emph{presentable} $n$-categories. Our main results in this direction is \cref{conj:infinityn} which gives a description of local systems of higher presentable categories in terms of \emph{higher monodromy data}, i.e. actions of iterated based loop spaces.   While all the relevant tools and concepts which are needed to state and prove \cref{thm:infinitytwothm} are well known among category theorists, this is not the case for \cref{conj:infinityn}. In particular, we shall carefully revisit the theory of \textit{presentable $n$-categories} and \textit{higher $n$-categories of modules} introduced in \cite{stefanich2020presentable}. 
Finally,   we will also explain how to lift Teleman's picture of topological actions in terms of Hochschild cohomology to the $n$-categorical setting. 

We remark that \cref{conj:infinityn} subsumes  \cref{thm:infinitytwothm}, and its proof is logically independent from it. However the proof of \cref{conj:infinityn} relies on the same key steps as \cref{thm:infinitytwothm}, which we believe are more easily grasped in the more familiar setting of ordinary presentable categories. Thus the proof of \cref{thm:infinitytwothm}, which is presented in  \cref{sec:inftytwo}, should be viewed as a practice run of our general argument, which will then be fully expounded in   \cref{sec:localsystemncat}.  
We start by fixing notations.

\begin{notation}
\label{notation:highercat}
In the rest of this paper we shall often work with higher categories.  Given an $n$-category $\scrC$ there are two basic operations we can perform. If $m < n$, we can consider the underlying $m$-category of $\scrC$, by discarding all non-invertible $k$-simplices such that $k > m$; viceversa, if $m>n$, we can promote $\scrC$ to a $m$-category such that all its $k$-simplices, for $k >n$, are invertible (for example, we will sometimes consider a space $X$ as an $n$-category). In order to avoid confusion  we shall adopt the following \emph{non standard} notations.

\begin{enumerate}
\item For any $n>1$, an $n$-category $\scrC$ admitting non-invertible $n$-simplices will be denoted as $n \bm{ \scrC}$ in order to highlight its ``categorical height''. 
    \item For $1<m \leqslant n$, the $m$-category  obtained by $n \bm{\scrC}$ discarding non-invertible $k$-simplices for all $m<k\leqslant n$ will be denoted as $m \bm{\scrC}$. For instance, we shall denote the underlying $n$-category of the $(n+1)$-category $(n+1)\CatVtwoinfty[n]$ as $n\CatVinftytwo[n]$. When $m=1$, we shall drop the $1$ and simply write $\scrC$: in particular, $\CatVinfty[n]$ is the $1$-category of $n$-categories.
    \item     If $m\leqslant n$, then any $m$-category seen as an $n$-category will still be denoted in the same way, e.g. $m\bm{\scrC}$. For example, any space $X$ seen as a trivial $n$-category will still be denoted as $X$ (instead of $\iota_n\cdots\iota_1X$, which is a convention  sometimes adopted in the literature). 
\end{enumerate}
For a precise technical formulation of the above constructions, we refer the reader to \cref{remark:adjunction}.
\end{notation}
\subsection{\texorpdfstring{$2$}{2}-categorical equivalences}
\label{sec:inftytwo}
We start by explaining a  technical construction due to Gepner--Haugseng, which allows us to change enrichment along lax monoidal functors. Next, using  this, we will show that  the categories appearing in the statements of \cref{monodromyeq} and  \cref{cor:maincor}  admit natural $2$-categorical enhancements. This will be explained in \cref{constr:enrichment1}, \cref{constr:enrichment2} and \cref{constr:enrichment3} below. The main result is \cref{thm:infinitytwothm}, that shows that the equivalences proved in  \cref{monodromyeq} and  \cref{cor:maincor} can be promoted to equivalences of 2-categories.

\begin{proposition}[{Change of base enrichment, \cite[Corollary $5.7.6$]{gepnerhaugseng}}]
	\label{prop:basechange}
Let $\mathscr{V}$ and $\mathscr{W}$ be two monoidal categories, and let $F\colon\mathscr{V}\to\mathscr{W}$ be a lax monoidal functor. Then there is a canonical functor $$\operatorname{Lin}_{\mathscr{V}}\CatVinfty\longrightarrow\operatorname{Lin}_{\mathscr{W}}\CatVinfty$$from the category of $\mathscr{V}$-enriched categories to the category of $\mathscr{W}$-enriched categories.
\end{proposition}

Recall that $\LinAPrLU$ is naturally symmetric monoidal. Indeed, 
 limits and colimits inside $\LinAPrLU$ are computed as in $\PrLU$ (see the proof of \cref{lemma:limitsandcolimitsovergroupoids}). Then, by \cite[Theorem $4.5.2.1$]{ha},  the category $\smash{\LinAPrLU}$ carries a  natural  symmetric monoidal structure given by the relative tensor product relative over $\scrA$; moreover,  such monoidal structure commutes with colimits separately in each variable

\begin{corollary}
	\label{cor:infinittwo}
Let $\scrA$ be a presentably symmetric monoidal category. Any category $\scrC$ enriched over $\LinAPrLU$ is enriched over $\CatVinfty$ hence is a $2$-category.
\end{corollary}
\begin{proof}
Note first that the natural inclusion functor $\PrLU\subseteq\CatVinfty$ is lax monoidal, because it is a composition of the strongly monoidal inclusion $\PrLU\subseteq\CatVinftycocom$ (\cite[Proposition $4.8.1.15$]{ha}) with the lax monoidal inclusion $\CatVinftycocom\subseteq\CatVinfty$ (\cite[Corollary $4.1.8.4$]{ha}). Moreover, given any presentably symmetric monoidal category $\scrA$, the natural cocontinuous and symmetric monoidal functor $\scrS\to\scrA$ yields a strongly monoidal functor$$\PrLU\simeq\operatorname{Lin}_{\scrS}\PrLU\longrightarrow\LinAPrLU$$which is left adjoint to the natural forgetful functor $\LinAPrLU\to\PrLU$. Therefore, such forgetful functor is lax monoidal (\cite[Theorem B]{laxmonoidal}), and we have a chain of forgetful lax monoidal functors$$\LinAPrLU\longrightarrow\PrLU\longhookrightarrow\CatVinfty.$$
The    statement then  follows from \cref{prop:basechange}. 
\end{proof}

\begin{remark}
	\label{remark:PrLModenrichment}
Let single out an important special case of \cref{cor:infinittwo}. 
As $\smash{\LinAPrLU}$ is symmetric monoidal  it is naturally enriched over itself, by \cite[Proposition $4.2.1.33.(2)$]{ha}. Hence, in virtue of \cref{cor:infinittwo}, it is enriched over $\CatVinfty$.This  provides the $2$-categorical enhancement of $\PrLU$ and $\LinAPrLU$, respectively. Following the conventions introduced in \cref{notation:highercat}, we denote these enhancements as $2\PrLUtwo$ and $2\LinAPrLUtwo$.
\end{remark}

\begin{theorem}
\label{thm:infinitytwothm}
Let $X$ be a  connected space, and let $\scrA$ be a presentably symmetric monoidal category. Then, there is an equivalences of $2$-categories
	$$2\LCXPrLAtwo\simeq 2\LModtwo_{\Omega_*X}{\lp2\LinAPrLUtwo\rp}$$
Additionally, if $X$ is simply connected, there are equivalences of $2$-categories
	$$2\LCXPrLAtwo\simeq 2\LModtwo_{\Omega_*X}{\lp2\LinAPrLUtwo\rp} \simeq2\LModtwo_{\LMod_{\Omega^2_*X}(\scrA)}{\lp2\LinAPrLUtwo\rp}.$$
\end{theorem}
 Let us briefly explain our  strategy to prove \cref{thm:infinitytwothm}. The three categories appearing in the statement of \cref{cor:maincor} are naturally cocompletely tensored over $\LinAPrLU$, and the tensor action is closed. This yields an enrichment over $\LinAPrLU$ and hence, by \cref{cor:infinittwo},  a structure of $2$-categories. We will prove that the equivalences of \cref{cor:maincor} intertwine the tensor action of $\LinAPrLU$: this imples that the three categories have compatible $\LinAPrLU$-enrichments and readily implies \cref{thm:infinitytwothm}.

\begin{construction}[{\bf The $2$-category of local systems}]
\label{constr:enrichment1}
Let $\scrC$ be a  cocompletely closed symmetric monoidal category and let $\scrD$ be a small category. The category of functors ${\Fun}{\lp\scrD,\hsp\scrC\rp}$ inherits a point-wise cocompletely symmetric monoidal structure by  \cite[Remark $2.1.3.4$]{ha}. With this  monoidal structure,  the functor$$\mathrm{const}\colon\scrC\longrightarrow{\Fun}{\lp \scrD,\hsp\scrC\rp}$$becomes   limit- and colimit-preserving,  and strongly monoidal. This turns ${\Fun}{\lp \scrD,\hsp\scrC\rp}$ into an $\scrC$-$\Einf$-algebra  by  \cite[Corollary $3.4.1.7$]{ha}; in particular, ${\Fun}{\lp\scrD,\hsp\scrC\rp}$ is left tensored over $\scrC$ inside $\CatVinftycocom$.

 Note  that the symmetric monoidal structure of ${\Fun}{\lp\scrD,\hsp\scrC\rp}$ is closed: for any functor $F\colon\scrD\to\scrC$ the action$$-\otimes F\colon{\Fun}{\lp\scrD,\hsp\scrC\rp}\longrightarrow{\Fun}{\lp\scrD,\hsp\scrC\rp}$$admits a right adjoint 
 $$\Mapin_{{\Fun}{\lp\scrD,\hsp\scrC\rp}}{\lp F,\hsp-\rp}\colon{\Fun}{\lp\scrD,\hsp\scrC\rp}\longrightarrow{\Fun}{\lp\scrD,\hsp\scrC\rp}
 $$
 which can be informally described as the assignment sending a functor $G$ to the functor 
$$D\mapsto \Mapin_{\scrC}{\lp FD,\hsp GD\rp}.$$In particular, \cite[Theorem $1.1$]{enrichment} guarantees that ${\Fun}{\lp\scrD,\hsp\scrC\rp}$ is enriched over itself, hence over $\scrC$ thanks to \cref{prop:basechange}.

In  our case $\scrC$ is $\PrLU$ and $\scrD$ is  a  space $X$ seen as a groupoid.  This yields the desired $\PrLU$-enrichment of $\smash{\LCXPrL}$. Following our conventions we denote   the resulting $2$-category $2\smash{\mathbf{LocSysCat}{\lp X \rp}}$.   Thanks to \cref{remark:PrLModenrichment} this discussion extends also to the case of  $\scrA$-enriched presentable categories,  where  $\scrA$ is a presentably symmetric monoidal category.  This yields the $2$-category $2\smash{\mathbf{LocSysCat}{\lp X;\scrA\rp}}.$ We will call this  $2$-category the \textit{$2$-category of $\scrA$-linear categorical local systems over $X$}.
\end{construction}
\begin{remark}
	\label{remark:agreement}
 It is possible to define alternatively the $2$-category of ($\scrA$-linear) categorical systems over $X$ as the $2$-category of $2$-functors between $X$ (seen as a trivial $2$-category via the strongly monoidal inclusion $\scrS\subseteq\CatVinfty$) and $2\LinAPrLUtwo$.  In fact, this approach might seem preferable, as it parallels more closely the definition of ordinary local systems. Indeed, if  $\scrC$ is a presentable category, local systems on $X$ with coefficients in $\scrC$ are defined precisely as functors from $X$ to $\scrC$.

 However, it is easy to see that the two definitions agree. A  straightforward computation  shows that the $2$-category $2\LocSysCattwo(X;\scrA)$  is equivalent to the evaluation at $\scrA$ of the right adjoint to the Cartesian product of $2$-categories$$-\times X\colon2\CatVinftytwo\longrightarrow2\CatVinftytwo,$$which is how the ``correct" $\CatVinfty$-enriched category of $\CatVinfty$-enriched functors is defined in \cite{gepnerhaugseng}. In particular, $2\smash{\mathbf{LocSysCat}{\lp X;\scrA\rp}}$ coincides with the internal mapping object in $2\CatVinftytwo$ between $X$ and $2\LinAPrLUtwo$.
\end{remark}
\begin{construction}[{\bf The $2$-category of $\Omega_*X$-module categories}]
	\label{constr:enrichment2}
Let $\scrA$ be a presentably symmetric monoidal category  and let $X$ be a pointed connected space. Consider the category  $\smash{\LMod_{\Omega_*X}{\lp\LinAPrLU\rp}}$. By \cref{lemma:rectificationofmodules} we can  replace $\Omega_*X$ with $\Omega_*X\otimes\scrA$, i.e., with the image of $\Omega_*X$ under the colimit preserving and symmetric monoidal functor \begin{align}
\label{functor:tensoring}
\scrS\longrightarrow\LinAPrLU.
\end{align}The category $\Omega_*X\otimes\scrA$ is presentably $\Ebb_1$-monoidal (because $\Omega_*X$ is a topological monoid) but it is also a \textit{cocommutative bialgebra} in $\LinAPrLU$. Indeed, the symmetric monoidal structure on $\scrS$ is Cartesian (\cite[Section $2.4.1$]{ha}), hence every object is naturally a cocommutative comonoid (this is a dual statement to \cite[Proposition $2.4.3.9$]{ha}). In particular, applying a strongly monoidal functor preserves both the algebra and coalgebra structures. Since we are dealing with a category of left modules over a bialgebra which is cocommutative (i.e., an $\Einf$-coalgebra) we can  apply  the following result due to Beardsley.

\begin{proposition}[{\cite[Corollary $3.19$]{bialgebrasmonoidalstructure}}]
	\label{prop:bialgebramonoidal}
	Let $\scrC$ be a symmetric monoidal category. Let $H$ be a $(n,k)$-bialgebra in $\scrC$ -- i.e., an $\Ebb_n$-algebra object in the category of $\Ebb_k$-coalgebras in $\scrC$. Then the category of left $H$-modules $\LMod_H(\scrC)$ admits an $\Ebb_k$-monoidal structure, and the forgetful functor $\oblv_H\colon\LMod_H(\scrC)\to\scrC$ is $\Ebb_k$-monoidal.
\end{proposition}
Note that this implies that $\smash{\LMod_{\Omega_*X}{\lp\LinAPrLU\rp}}$  carries a natural symmetric monoidal structure which is compatible with colimits.  Indeed, the action of $\scrS$ on $\LinAPrLU$ is compatible with colimits; therefore, \cite[Corollaries $4.2.3.3$ and $4.2.3.5$]{ha} imply that both limits and colimits of $\Omega_*X$-modules in $\LinAPrLU$ can be computed after forgetting the  $\Omega_*X$-action.

Note also that this symmetric monoidal structure is closed.  That is, if $\scrC$ and $\scrD$ are two presentably $\scrA$-linear categories   endowed with a $\Omega_*X$-action, then the  category of $\scrA$-linear and colimit-preserving functors ${\FuninL_{\scrA}}{\lp\scrC,\hsp\scrD\rp}$ carries a $\Omega_*X$-action informally described by the rule
\begin{align*}
g\cdot F\colon \scrC&\longrightarrow\scrD\\ C&\mapsto g\cdot_{\scrD}F\lp g^{-1}\cdot_{\scrC}C\rp.
\end{align*}

In particular, we have a trivial $\Omega_*X$-action functor \begin{align}
\label{funct:triv1}
\triv_{\Omega_*X\otimes\scrA}\colon\LinAPrLU\longrightarrow\operatorname{Lin}_{\Omega_*X\otimes\scrA}\PrLU,
\end{align}induced by pulling back along the functor of symmetric comonoidal categories $$\Omega_*X\otimes\scrA\longrightarrow\left\{*\right\}\otimes\scrA\simeq\scrA$$
The latter, in turn, is induced by the natural map $\Omega_*X\to\left\{*\right\}$ after applying the functor \eqref{functor:tensoring}. The functor $\triv_{\Omega_*X\otimes\scrA}$ is a right adjoint which commutes with both limits and colimits. Moreover it is strongly monoidal, in virtue of  the description of the symmetric monoidal structure on $\smash{\LMod_{\Omega_*X}{\lp\LinAPrLU\rp}}$ provided by \cref{prop:bialgebramonoidal}. This turns $\smash{\LMod_{\Omega_*X}{\lp\LinAPrLU\rp}}$ into a $\LinAPrLU$-$\Einf$-algebra with  a closed $\LinAPrLU$-action; hence it promotes it to an $2$-category. Following our conventions, we denote this $2$-category as 
 $2\smash{\LModtwo_{\Omega_*X}{\lp2\LinAPrLUtwo\rp}}$.

\end{construction}

\begin{construction}[{\bf The $2$-category of $\LMod_{\Omega^2_*X}$-modules}]
\label{constr:enrichment3}
Let $\scrA$ be a presentably symmetric monoidal category and let  $X$ be a pointed simply connected space. Arguing as in \cref{constr:enrichment2}, we  see  that the double based loop space $\Omega^2_*X$ is both an $\Ebb_2$-monoid and an $\Einf$-comonoid in spaces, in a compatible way. Using the terminology of \cite{bialgebrasmonoidalstructure}, we can say that $\Omega^2_*X$ is a \textit{$(2,\infinity)$-bimonoid}.

This implies that the presentable category $\LMod_{\Omega_*^2X}{\lp\scrA\rp}$ is naturally a $(1,\infinity)$-bialgebra object in $\LinAPrLU$. Indeed, it is the image of the pointed and connected $(2,\infinity)$-bimonoid $\Omega_*X$ under the composition of two strongly monoidal functors:  namely the functor \eqref{functor:modulesoverloopspace}, which is proved to be strongly monoidal in \cref{prop:koszulismonoidal}, and the functor\begin{align}
\label{functor:lmod}
	\LMod_{(-)}(\scrA)\colon\Alg_{\Ebb_1}{\lp\scrA\rp}\longrightarrow{\lp\LinAPrLU\rp}_{\scrA/}
\end{align}which is strongly monoidal for any presentably symmetric monoidal category $\scrA$ (\cite[Theorem $4.8.5.16$]{ha}). Using again \cref{prop:bialgebramonoidal}, we see that   $\smash{\operatorname{Lin}_{\LMod_{\Omega^2_*X}(\scrA)}\PrLU}$ comes equipped with a cocontinuous closed symmetric monoidal structure such that the forgetful functor$$\oblv_{\LMod_{\Omega^2_*X}(\scrA)}\colon \operatorname{Lin}_{\LMod_{\Omega^2_*X}(\scrA)}\PrLU\longrightarrow\LinAPrLU$$is strongly monoidal.  Let us make a comment on this monoidal structure: the internal mapping  category between two objects $\scrC$ and $\scrD$   in 
$\smash{\operatorname{Lin}_{\LMod_{\Omega^2_*X}(\scrA)}\PrLU}$  is \textit{not} the category of $\LMod_{\Omega^2_*X}(\scrA)$-linear colimit-preserving functors between $\scrC$ and $\scrD$. Rather, just like in \cref{constr:enrichment2}, it is  the category ${\FuninL_{\scrA}}{\lp\scrC,\hsp\scrD\rp}$ of $\scrA$-linear and colimit-preserving functors from $\scrC$ to $\scrD$, equipped with its natural $\LMod_{\Omega^2_*X}(\scrA)$-linear structure. Objectwise, this action can be described as follows: for a functor $F\colon\scrC\to\scrD$ and an $\Omega^2_*X$-module object $M$ in $\scrA$, the functor $M\otimes F$ is defined as $F(M\otimes_{\scrC}-)$

Now note that the forgetful functor$$\oblv_{\Omega^2_*X} \colon\LMod_{\Omega^2_*X}(\scrA)\longrightarrow\scrA$$is lax $\Ebb_1$-monoidal, since it is the right adjoint of the free $\Omega_*^2X$-module functor$$-\otimes\Omega^2_*X\colon\scrA\longrightarrow\LMod_{\Omega^2_*X}{\lp\scrA\rp}$$which is easily seen to be $\Ebb_1$-monoidal. Indeed, the functor $-\otimes\Omega_*^2X$ is  the image of the map of $\Ebb_2$-algebras $\boldone_{\scrA}\to \Omega^2_*X\otimes \boldone_{\scrA}$ under the strongly monoidal functor \eqref{functor:lmod}. In particular, $\oblv_{\Omega^2_*X}$ induces a pullback functor between categories of left modules\begin{align}
\label{functor:triv2}
\triv_{\LMod_{\Omega^2_*X}(\scrA)}\colon \LinAPrLU\longrightarrow\operatorname{Lin}_{\LMod_{\Omega^2_*X}(\scrA)}\PrLU,
\end{align}which equips an $\scrA$-linear presentable category $\scrC$  with the trivial $\LMod_{\Omega^2_*X}(\scrA)$-action.  Arguing as in \cref{constr:enrichment2},   we can see that this functor preserves all limits and colimits and is strongly monoidal. As a consequence $\smash{\operatorname{Lin}_{\LMod_{\Omega^2_*X}(\scrA)}\PrLU}$ is tensored (and hence enriched) over $\LinAPrLU$. This yields the desired $2$-categorical enhancement, which we denote $2\smash{\mathbf{Lin}_{\LMod_{\Omega^2_*X}(\scrA)}\PrLUtwo}$.
\end{construction}

\begin{proof}[Proof of \cref{thm:infinitytwothm}]
 The first part of \cref{thm:infinitytwothm} follows from the fact that the equivalence of categories proved in \cref{cor:koszulpresentable} intertwines the symmetric monoidal structures of $\smash{\LCXPrLA}$ and $\smash{\LMod_{\Omega_*X}{\lp\LinAPrLU\rp}}$  (see Costructions \ref{constr:enrichment1} and \ref{constr:enrichment2}) and hence their $\LinAPrLU$-enrichment. For this, it is sufficient to note that the equivalence of \cref{cor:koszulpresentable} is compatible with the coaugmentation functors from $\LinAPrLU$, i.e., that it takes constant functors to trivial $\Omega_*X$-modules, which is clear.

 Let us now move to  the second half of   \cref{thm:infinitytwothm}.  Arguing as above, we see that  it is enough to show that  the equivalence of \cref{cor:maincor} is compatible with the coaugmentation functors from $\LinAPrLU$. 
Recall that the equivalence of these categories of modules arises from the equivalence of presentably $\Ebb_1$-monoidal categories$$\Omega_*X\otimes\scrA\simeq\LMod_{\Omega^2_*X}(\scrA)$$proved in \cref{sec:categorifyingmodules}. In particular, such equivalence does not change the underlying objects and therefore turns trivial $\Omega_*X$-actions into trivial $\LMod_{\Omega^2_*X}(\scrA)$-actions. So we can conclude that the diagram of functors$$\begin{tikzpicture}[scale=0.75]
\node (a1) at (-4,-2){$\LMod_{\Omega_*X}{\lp\LinAPrLU\rp}$};\node (a2) at (4,-2){$\operatorname{Lin}_{\LMod_{\Omega^2_*X}(\scrA)}\PrLU$};
\node (b) at (0,0){$\LinAPrLU$};
\draw[->,font=\scriptsize](b) to[bend right] node[left]{$\ref{funct:triv1}$} (a1);
\draw[->,font=\scriptsize](b) to[bend left] node[right]{$\ref{functor:triv2}$} (a2);
\draw[->,font=\scriptsize](a1) to node[below]{$\ref{cor:maincor}$}(a2);
\draw[->,font=\scriptsize](a1) to node[above]{$\simeq$}(a2);
\end{tikzpicture}$$
commutes, and this concludes the proof.
\end{proof}
 
\subsection{Local systems of \texorpdfstring{$n$}{n}-categories and higher modules}
\label{sec:localsystemncat}
In this Section, we will switch gears and study local systems of $n$-categories on a space $X$. Under suitable connectedness assumptions on $X$, we will prove an analogue of \cref{thm:infinitytwothm} which  provides an equivalence of $(n+1)$-categories relating local systems of $n$-categories and higher monodromy data (see \cref{conj:infinityn} below).

 We start by fixing  notations and by recalling some  relevant constructions   from  \cite{gepnerhaugseng} and \cite{stefanich2020presentable}. We warn the reader to bear in mind the non-standard notations for higher categories  we introduced in \cref{notation:highercat}.

\begin{notation}
Following \cite[Remark $5.7.13$]{gepnerhaugseng}, we define the \textit{$(n+1)$-category of $n$-categories} inductively as follows. Recall  that we denote by $\CatVinfty$ the very large category of (possibly large) categories. The category $\CatVinfty$ is symmetric monoidal with the cartesian product, so we can consider the $2$-category of $\CatVinfty$-enriched categories. We set 
$$2\widehat{\mathbf{Cat}}_{({\scriptstyle\infty},1)}\coloneqq {\operatorname{Lin}_{\widehat{\operatorname{Cat}}_{({\scriptscriptstyle\infty},1)}}{\CatVinfty}}.$$
 By  \cite[Corollary $5.7.12$]{gepnerhaugseng} this is again a symmetric monoidal category. Suppose  by the inductive hypothesis that we have defined the $n$-category $n\widehat{\mathbf{Cat}}_{({\scriptstyle\infty},n-1)}$ of $(n-1)$-categories. Then we define the \textit{$(n+1)$-category of $n$-categories} as $$(n+1)\widehat{\mathbf{Cat}}_{({\scriptstyle\infty},n)}\coloneqq\operatorname{Lin}_{n\widehat{\mathbf{Cat}}_{({\scriptscriptstyle\infty},n-1)}}\CatVinfty,$$i.e., as the $(n+1)$-category of categories enriched over $n$-categories. This agrees with \cite[Definition $6.1.5$]{gepnerhaugseng}.
\end{notation}
\begin{remark}
	\label{remark:adjunction}
	In \cref{notation:highercat} we introduced operations which, given an $m$-category $\scrC$ and an integer $n$, allow  us to turn  $\scrC$ into an $n$-category. As we shall clarify here, these operations are in fact functorial. 	\cref{prop:basechange} guarantees that the limit-preserving (and hence, strongly monoidal) inclusion $\scrS\subseteq\CatVinfty$ produces a functor$$\operatorname{Lin}_{\scrS}\CatVinfty=\CatVinfty  \longrightarrow\operatorname{Lin}_{\widehat{\operatorname{Cat}}_{({\scriptscriptstyle\infty},1)}}\CatVinfty\eqqcolon 2\widehat{\mathbf{Cat}}_{(\scriptstyle\infty,1)}$$which is moreover \textit{lax monoidal} for the Cartesian monoidal structure on both source and target (\cite[Corollary $5.7.11$]{gepnerhaugseng}). Hence, we can apply again \cref{prop:basechange} to obtain another lax monoidal functor$$2\CatVtwoinfty[1]\coloneqq\operatorname{Lin}_{\widehat{\operatorname{Cat}}_{({\scriptscriptstyle\infty},1)}}\CatVinfty \stackrel{\iota_1}\longrightarrow\operatorname{Lin}_{2\widehat{\mathbf{Cat}}_{({\scriptscriptstyle\infty},1)}}\CatVinfty\eqqcolon3\CatVtwoinfty[2].$$
Iterating this argument, we obtain a chain of lax monoidal inclusions$$\mathscr{S} \subseteq
\CatVinfty
\subseteq
2\widehat{\mathbf{Cat}}_{(\scriptstyle\infty,1)}\stackrel{\iota_1}\subseteq
\ldots\subseteq n\widehat{\mathbf{Cat}}_{({\scriptstyle\infty},n-1)} \stackrel{\iota_{n-1}}\subseteq (n+1)\widehat{\mathbf{Cat}}_{({\scriptstyle\infty},n)}\subseteq\ldots$$which allows us to consider any $\Ebb_k$-topological monoid as an $\Ebb_k$-monoidal $n$-category for every $n\geqslant1$. This is just the formalization of the natural idea (discussed in \cref{notation:highercat} above) that, if $n\geqslant m$, an $m$-category can be viewed as an $n$-category such that every $k$-simplex is an equivalence for $k>m$. In particular, it makes sense to consider \textit{modules} over a topological monoid $G$ in the $(n+1)$-category of $n$-categories.

Conversely, consider the functor $$(-)^{\simeq}\colon\CatVinfty\longrightarrow\scrS$$which takes a category to its maximal subgroupoid. This is the right adjoint to the strongly monoidal and colimit preserving inclusion $\scrS\subseteq\CatVinfty$; hence it inherits a lax monoidal structure.  
Another inductive argument using \cite[Proposition $5.7.17$]{gepnerhaugseng} yields for any positive integer $n$ the adjunction\begin{align}
\label{adjunction:infinityn}
\adjunction{\iota_{n-1}}{ n\widehat{\mathbf{Cat}}_{({\scriptstyle\infty},n-1)}}{(n+1)\widehat{\mathbf{Cat}}_{({\scriptstyle\infty},n)}}{\tau_{\leqslant n-1}},
\end{align}
where applying $\tau_{\leqslant n-1}$
  amounts to considering an $n$-category as an $(n-1)$-category by forgetting the non-invertible $n$-simplices. As explained in \cref{notation:highercat}, we will mostly drop the symbol $\iota_{n-1}$ from our notations.
\end{remark}

\begin{remark}
\label{remark:infinitynfunctors}
For any two $n$-categories $n\bm{\scrC}$ and $n\bm{\scrD}$, using \cite[§$6.1.3$]{hinichyoneda} we can construct a category of $n$-functors (i.e., of $n\widehat{\mathbf{Cat}}_{({\scriptstyle\infty},n-1)}$-enriched functors) that we denote as $\smash{{\Fun}_n{\lp n \bm{\scrC},\hsp n \bm{\scrD} \rp}}$. 
Note that if $n \bm{\scrC}\simeq \iota_n\cdots\iota_k (k\bm{\scrC})$ is just an $k$-category $k\bm{\scrC}$ for some $k\leqslant n$, then the adjunction \eqref{adjunction:infinityn} implies that we have a chain of equivalences of categories
$$
{\Fun}_n{\lp k\bm{\scrC},\hsp n\bm{\scrD}\rp}
\simeq 
{\Fun}_{n-1}{\lp k\bm{\scrC},\hsp \tau_{\leqslant n-1} n\bm{\scrD} \rp}
\simeq
\cdots
\simeq {\Fun}_k{\lp k\bm{\scrC},\hsp \tau_{\leqslant k} n\bm{\scrD} \rp}.$$
If $k=1$, i.e., $n\bm{\scrC} \simeq \iota_k\cdots\iota_1(\scrC)$ is an ordinary category $\scrC$, using the $n\CatVtwoinfty[n-1]$-enrichment of $\tau_{\leqslant 1}\scrD$ one can produce an $n$-enhancement of the category $${\Fun}_{n}{\lp \scrC,\hsp n \bm{\scrD}  \rp}\simeq{\Fun}{\lp \scrC,\hsp\tau_{\leqslant1}n \bm{\scrD} \rp}$$which recovers the internal mapping object $n\Funtwo(\scrC,n\bm{\scrD})$ for the Cartesian symmetric monoidal structure on $n\CatVinftytwo[n-1]$. 
\end{remark}

The theory of presentable $n$-categories was only recently introduced by Stefanich. As it will play a key role in the following, we shall revisit its basic definitions and constructions.  For more details, the reader can consult \cite[Section $5$]{stefanich2020presentable}. 
We warn the reader that the definition of presentable $n$-categories due to Stefanich that we use in this section is just one of the available definitions for presentable $n$-categories. In \cite{mazelgee2021universal}, the authors propose yet another definition of presentable $n$-categories, which is incompatible with the one by Stefanich. For example, the underlying $n$-category of a presentable $(n+1)$-category in the sense of Mazel-Gee and Stern is again presentable; however, the theory of presentable $n$-categories developed by Stefanich is built in   such a way that the $(n+1)$-category $(n+1)\PrLUtwo[n]$ of presentable $n$-categories is a presentable $(n+1)$-category. In particular, $2\PrLUtwo$ is a presentable $2$-category in the sense of Stefanich, but cannot be presentable in the sense of Mazel-Gee and Stern since $\PrLU$ is known to be not presentable.

 Consider the functor\begin{align}
 \label{functor:allmodules}
\LMod_{(-)}{\lp\CatVinftycocom\rp}\colon\Alg{\lp\CatVinftycocom\rp}\longrightarrow\CatWinftycocom
\end{align}
which takes a cocomplete monoidal category $\scrA$ and sends it to the (very large) cocomplete category $\smash{\LMod_{\scrA}{\lp\CatVinftycocom\rp}}$ of cocomplete $\scrA$-modules. It is a symmetric monoidal functor, so if $\scrA$ is a cocomplete $\Ebb_k$-monoidal category for some $k\geqslant2$ then the category $\smash{\LMod_{\scrA}{\lp\CatVinftycocom\rp}}$ inherits a cocomplete $\Ebb_{k-1}$-monoidal structure given by the relative Lurie tensor product over $\scrA$. In particular, if $k=+\infinity$ (i.e., if we start from the category of cocomplete \textit{symmetric} monoidal categories), we obtain a functor\begin{align}
    \label{functor:modules}    \operatorname{Mod}_{(-)}\lp\CatVinftycocom[1]\rp\colon\CAlg{\lp\CatVinftycocom\rp}\longrightarrow\CAlg{\lp\CatWinftycocom\rp}.
\end{align}

For a cocomplete symmetric monoidal category $\scrA$, we would like to define an $n$-category of $\scrA$-modules  by iterating  $n$  times the functor \eqref{functor:modules}. However, in order to have a consistent theory,  we would need  a chain of nested universes.
We can fix this issue as follows. Let $\kappa_0$ be the first large cardinal with respect to our initial choice of universes $\scrU$ and $\scrV$. This means that $\kappa_0$-small spaces and categories are what we call small spaces and categories.

\begin{defn}
	\label{parag:npresentable1cat}
	 Let\begin{align}
\label{functor:compactmodules}
\operatorname{LMod}_{(-)}^{\operatorname{pr}}\coloneqq\LMod_{(-)}{\lp\CatVinftycocom\rp}^{\kappa_0}\colon \Alg{\lp\CatVinftycocom\rp}\longrightarrow\CatVinftycocom
\end{align}
denote the functor that sends a cocomplete monoidal category $\scrA$ to the category of $\kappa_0$-compact objects inside $\smash{\LMod_{\scrA} {\lp\CatVinftycocom\rp}}$. We say that an object of $\smash{\LMod_{\scrA}{\lp\CatVinftycocom\rp}^{\kappa_0}}$ is a \textit{presentable left $\scrA$-module}.
\end{defn}

The functor \eqref{functor:compactmodules} admits a lax monoidal structure (\cite[Remark $5.1.11$]{stefanich2020presentable}), hence it sends cocompletely $\Ebb_k$-monoidal categories to cocompletely $\Ebb_{k-1}$-monoidal categories. 
For $n\leqslant k$ we set 
\begin{align}
    \label{functor:nfoldcompactmodules}
    \operatorname{LMod}^{\operatorname{pr},n}_{(-)}\colon\Alg_{\Ebb_k}{\lp\CatVinftycocom[1]\rp}\longrightarrow\Alg_{\Ebb_{k-n}}{\lp\CatVinftycocom\rp}
\end{align}to be the $n$-fold iteration of the functor \eqref{functor:compactmodules}.

\begin{defn}
\label{def:nfoldmodules}
Let $k\in\NN_{\geqslant1}\cup\left\{\infinity\right\}$ and $n\leqslant k$ be integers, and let $\scrA$ be a presentably $\Ebb_k$-monoidal category.  
The \textit{category of presentable $\scrA$-linear $n$-categories} $$\LinAPrLU[n]$$ is the cocompletely symmetric monoidal category obtained by applying functor \eqref{functor:nfoldcompactmodules} to $\scrA$. In the case when $\scrA=\scrS$ is the category of spaces, we shall simply write $\PrLU[n]$ and call it the \textit{category of presentable $n$-categories}.  This choice of notation is justified by \cref{remark:PrL}.
\end{defn}

\begin{remark}
	\label{remark:PrL}
Note that when $n=1$ the category $\LinAPrLU[n]$ of \cref{def:nfoldmodules} agrees with the usual category of $\scrA$-linear presentable categories. This is a consequence of \cite[Proposition $5.1.4$ and Corollary $5.1.5$]{stefanich2020presentable}: $\kappa_0$-compact left $\scrA$-modules in $\CatVinftycocom$ are \textit{precisely} $\scrA$-linear presentable categories. 
\end{remark}

\begin{remark}
Our \cref{def:nfoldmodules} differs slightly  from Stefanich's conventions. Stefanich denotes the category $\LinAPrLU[n]$ by $\LMod_{\scrA}^{\operatorname{pr},n}$ and refers to it as the \textit{category of presentable categorical $n$-fold $\scrA$-modules} (see \cite[Definition $5.2.2$]{stefanich2020presentable}). We opted for the conventions in \cref{def:nfoldmodules} as this highlights the fact that objects in $\LinAPrLU[n]$ should be viewed as ($\scrA$-linear) presentable 
$n$-categories; also, by \cref{remark:PrL},  the notation $\LinAPrLU[n]$ has the advantage of being coherent with the notation we use in the ordinary case $n=1$. 
\end{remark}

\begin{construction}
    	\label{parag:infinitynpres}
Let $k\in\NN_{\geqslant1}\cup\left\{\infinity\right\}$ and $n\leqslant k$ be integers, and let $\scrA$ be a presentably $\Ebb_k$-monoidal category. We define the category $$\smash{\operatorname{Lin}_{\scrA}\CatVinftycocom[n]}\coloneqq \Mod_{\LinAPrLUtiny[n-1]}{\lp\CatVinftycocom[n]\rp}$$of cocompletely $\scrA$-linear $n$-categories as the image of $\scrA$ under the composition$$\Alg_{\Ebb_k}{\lp\CatVinftycocom\rp}\xrightarrow{\ref{functor:nfoldcompactmodules}}\Alg_{\Ebb_{k-n+1}}{\lp\CatVinftycocom\rp}\xrightarrow{\ref{functor:modules}}\Alg_{\Ebb_{k-n}}{\lp\CatWinftycocom\rp}$$where the first functor is the $(n-1)$-fold iteration of the functor (\ref{functor:compactmodules}) and the second is simply the functor (\ref{functor:modules}).

Note that, by definition, there is an inclusion
$$
\LinAPrLU[n] \longhookrightarrow\smash{\operatorname{Lin}_{\scrA}\CatVinftycocom[n]}.
$$
In general, objects in 
$\operatorname{Lin}_{\scrA}\CatVinftycocom[n-1]$
might not be   presentable, but they are 
always  cocompletely tensored and enriched over $\LinAPrLU[n-1]$. Thus, we can view them  as  $n$-categories as follows. Consider the composition of lax monoidal functors 
\begin{align*}
\operatorname{Lin}_{\scrA}\CatVinftycocom[n] \xrightarrow{
(a)}\operatorname{Lin}_{n\operatorname{Lin}_{\scrA}\PrLUtwotiny[n-1]}\PrLV& \stackrel{(b)}\longhookrightarrow\operatorname{Lin}_{m\operatorname{Lin}_{\scrA}\PrLUtwotiny[n-1]}\CatWinfty \stackrel{(c)}\longrightarrow\operatorname{Lin}_{n\widehat{\mathbf{Cat}}_{({\scriptscriptstyle\infty},n-1)}}\CatWinfty.
\end{align*}
where the functor $(a)$ is the Ind-completion functor $\operatorname{Ind}_{\kappa_0}$, the functor $(b)$ is induced by the forgetful functor 
$$
 \PrLV \longrightarrow \CatWinfty
$$
and the functor $(c)$ is the change of enrichment along the lax monoidal forgetful functor 
$$
\LinAPrLU[n] \longrightarrow (n+1)\CatVinftytwo[n].
$$

In this way, one obtains from an arbitrary $\scrA$-linear cocomplete $n$-category $n\bm{\scrC}$ an associated very large $n$-category. Taking its sub-$n$-category of $\kappa_0$-compact objects, we get a well defined lax monoidal functor
\[
\operatorname{Lin}_{\scrA}\CatVinftycocom[n] \longrightarrow\operatorname{Lin}_{n\widehat{\mathbf{Cat}}_{({\scriptscriptstyle\infty},n-1)}}\CatVinfty\eqqcolon(n+1)\CatVinftytwo[n].
\]
 In particular, for any $n\geqslant1$ we have an induced lax monoidal change-of-enrichment functor
\begin{align}
\label{onlols}
\begin{split}
\operatorname{Lin}_{(n+1)\operatorname{Lin}_{\scrA}\CatVinftycocomtiny[n]}\PrLV&\longrightarrow\operatorname{Lin}_{\operatorname{Lin}_{\scrA}\CatVinftycocomtiny[n]}\CatVinfty[1]\\&\longrightarrow\operatorname{Lin}_{(n+1)n\widehat{\mathbf{Cat}}_{({\scriptscriptstyle\infty},n)}}\CatVinfty[1]\eqqcolon(n+2)\CatVinftytwo[n+1]
\end{split}
\end{align}
\end{construction}
\begin{defn}
We denote by $(n+1)\mathbf{Lin}_{\scrA}\smash{\widehat{\mathbf{Cat}}}^{\operatorname{L}}_{({\scriptstyle\infty},n)}$ the $(n+1)$-category which is the image of the unit object in 
$\operatorname{Lin}_{\operatorname{Lin}_{\scrA}\CatVinftycocomtiny[n]}\PrLV$ under the functor \eqref{onlols}.
\end{defn}
The category $(n+1)\mathbf{Lin}_{\scrA}\smash{\widehat{\mathbf{Cat}}}^{\operatorname{L}}_{({\scriptstyle\infty},n)}$  is a symmetric monoidal $(n+1)$-category whose underlying symmetric monoidal category is equivalent to $\operatorname{Lin}_{\scrA}\CatVinftycocom[n]$ (\cite[Remark $5.3.5$]{stefanich2020presentable}).
\begin{warning}
When $n\geqslant1$, if $n\bm{\scrA}$ is a presentable monoidal $n$-category it is not obvious that an $n\bm{\scrA}$-module $n\bm{\scrC}$ is also enriched over $n\bm{\scrA}$, essentially because it it not known whether the monoidal structure on $(n+1)\PrLUtwo[n]$ is closed (see the proof of \cref{conj:infinityn} and \cref{conj:notenriched}). Therefore, in the following, we shall write $\smash{(n+1)\LModtwo_{n\bm{\scrA}}{\lp(n+1)\PrLUtwo[n]\rp}}$ for the $(n+1)$-category of presentable $n$-categories which are left tensored over $n\bm{\scrA}$, instead of $(n+1)\mathbf{Lin}_{n\bm{\scrA}}{\lp\PrLUtwo\rp}$. 
\end{warning}

\begin{defn}
	\label{def:npresentable}
The \textit{$(n+1)$-category of $\scrA$-linear presentable $n$-categories} is the full sub-$(n+1)$-category$$(n+1)\mathbf{Lin}_{\scrA}\mathbf{Pr}^{\mathrm{L}}_{({\scriptstyle\infty}, n)}\subseteq(n+1)\smash{\widehat{\mathbf{Cat}}}^{\operatorname{L}}_{({\scriptstyle\infty},n)}$$spanned by presentable $n$-categories (in the sense of \cref{def:nfoldmodules}). In particular, the underlying symmetric monoidal category of $(n+1)\mathbf{Lin}_{\scrA}\mathbf{Pr}^{\mathrm{L}}_{({\scriptstyle\infty}, n)}$ is equivalent to $\LinAPrLU[n]$ (\cite[Remark $5.3.7$]{stefanich2020presentable}).
\end{defn}
 
\begin{remark}
	\label{remark:comparisioninftytwo}
Unraveling all constructions, and using \cite[Remark $5.3.5$ and $5.3.7$]{stefanich2020presentable}, we  see that the $2$-categorical  enhancement  $2\mathbf{Lin}_{\scrA}\mathbf{Pr}^{\mathrm{L}}_{({\scriptstyle\infty}, 1)}$ relies on considering $\LinAPrLU$ as enriched over itself via its closed symmetric monoidal structure. In particular, thanks to \cref{remark:PrL}, the category $2\LinAPrLUtwo[1]$ as defined in \cref{def:npresentable} is equivalent to the $2$-categorical enhancement of $\LinAPrLU[1]$ we described earlier in \cref{remark:PrLModenrichment}.
\end{remark}
Having defined the $(n+1)$-category of $\scrA$-linear presentable $n$-categories, the definition of the $(n+1)$-category of local systems of $\scrA$-linear presentable $n$-categories on a space $X$ is straightforward.
\begin{defn}
The \textit{$(n+1)$-category of $\scrA$-linear presentable $n$-categories on $X$} is defined as 
 $$(n+1)\LocSysCattwo^n(X;\scrA)\coloneqq(n+1)\mathbf{Fun}_{n+1}{\lp X,\hsp (n+1)\mathbf{Lin_{\scrA}Pr}^{\operatorname{L}}_{({\scriptstyle\infty},n)}\rp}.$$
 \end{defn} 
\begin{remark}
Note that, as in \cref{remark:agreement}, this $(n+1)$-category is equivalent to the internal mapping object in $(n+2)\CatVtwoinfty[n+1]$ between $X$ (seen as an $(n+1)$-category) and $(n+1)\LinAPrLUtwo[n]$.
\end{remark}
We introduce the last bit of notations that we need in order to further categorify \cref{cor:maincor}.

\begin{construction}[{\cite[Section $5.2$]{stefanich2020presentable}}]
	\label{constr:highermodules}
	Let $k\in\NN_{\geqslant1}\cup\left\{\infinity\right\}$ and $n\leqslant k-1$ be integers, and let $\scrA$ be a presentably $\Ebb_k$-monoidal category. Let $A$ be an 
	$\Ebb_n$-algebra object in $\scrA$. We can give an inductive definition of the presentable  $(n+1)$-category $(n+1)\LModtwo^n_A$ of $n$-categorical $A$-modules.
\begin{enumerate}
\item For $n=0$, we simply define the category $\LMod^0_{A}(\scrA)$ to be $\smash{\LMod_A{\lp\scrA\rp}}$ and$$\LMod^0_{(-)}(\scrA)\coloneqq\LMod_{(-)}{\lp\scrA\rp}\colon\Alg_{\Ebb_{k}}(\scrA)\longrightarrow\Alg_{\Ebb_{k-1}}{\lp \LinAPrLU[1]\rp}$$to be the functor induced at the level of $\Ebb_{k}$-algebras by the strongly monoidal functor \eqref{functor:lmod}. 
\item For any $1\leqslant n\leqslant k$, we define the $(n+1)$-category $(n+1)\LModtwo_A^n(\scrA)$ to be the image of $A$ under the functor$$\Alg_{\Ebb_{k}}(\scrA)\longrightarrow\Alg_{\Ebb_{k-n-1}}{\lp \LinAPrLU[n+1]\rp}$$defined inductively as follows. It is the composition of the functor$$n\LModtwo^{n-1}_{(-)}(\scrA)\colon \Alg_{\Ebb_{k}}(\scrA)\longrightarrow\Alg_{\Ebb_{k-n}}{\lp\LinAPrLU[n]\rp}$$with the functor$$(n+1)\LModtwo_{(-)}^{\kappa_0}\colon\Alg_{\Ebb_{k-n}}{\lp\LinAPrLU[n]\rp}\longrightarrow\Alg_{\Ebb_{k-n-1}}{\lp \LinAPrLU[n+1]\rp}.$$The latter functor is induced by the functor \eqref{functor:compactmodules},  since for every cocompletely   $\Ebb_k$-monoidal category $\scrA$ the  assignment  $\scrA\mapsto\LMod_{\scrA}^{\operatorname{pr}}$ is functorial and strongly monoidal (\cite[Remark $5.1.13$]{stefanich2020presentable}).  This yields a strongly monoidal functor$$(n+1)\LModtwo^n_{(-)}(\scrA)\colon\Alg{\lp\LinAPrLU[n]\rp}\longrightarrow
\LinAPrLU[n+1].$$
In particular, for any $\Ebb_n$-algebra $A$ the $(n+1)$-category $(n+1)\LModtwo^n_A(\scrA)$ is a presentably $\Ebb_{k-n-1}$-monoidal $\scrA$-linear  $(n+1)$-category. If $\scrA$ is a symmetric monoidal category and $A$ is a \textit{commutative} algebra in $\scrA$, we shall  write simply $(n+1)\Modtwo^n_{A}(\scrA)$. Note that the latter is    a presentably \textit{symmetric} monoidal $\scrA$-linear $(n+1)$-category. 
\end{enumerate}
\end{construction}
\begin{defn}
\label{def:iteratedmodules}
We call $(n+1)\LModtwo^n_A$ the \textit{$(n+1)$-category of $n$-fold $A$-modules}. 
When $\scrA=\Sp$ is the category of spectra and $A$ is a commutative ring spectrum, we  set 
$$(n+1)\Modtwo^n_A\coloneqq (n+1)\Modtwo^n_A(\Sp).$$
and call it the \textit{$(n+1)$-category of $n$-fold $A$-modules}.
\end{defn}
\begin{remark}
\label{remark:functorsnenrichment}
Using \cite[Remark $5.2.4$]{stefanich2020presentable}, we can see that for $n\geqslant1$ the $(n+1)$-category $(n+1)\LModtwo^n_A(\scrA)$  is equivalent to   the $(n+1)$-category $(n+1)\LModtwo_{n\LModtwo^{n-1}_A(\scrA)}\PrLUtwo[n].$
So we can think of  $(n+1)\LModtwo^n_A(\scrA)$ as the $(n+1)$-category of  
 $A$-linear presentable $n$-categories. In particular, when $\scrA=\Sp$ is the category of spectra and $A$ is an $\Ebb_n$-ring spectrum, we will denote $(n+1)\LModtwo^n_A(\scrA)$ as $(n+1)\mathbf{Lin}_{A}\PrLUtwo[n].$
\end{remark}
The previous discussions provides all the ingredients to state the $n$-categorical generalization of \cref{thm:infinitytwothm}.
\begin{theorem}
	\label{conj:infinityn}
Let $n\geqslant 1$ be an integer, let $X$ be a pointed $n$-connected space (i.e., $\pi_k(X)\cong 0$ for every $k\leqslant n$), and let $\scrA$ be a presentably symmetric monoidal category. Then there exist equivalences of $(n+1)$-categories\begin{align*}
(n+1)\LocSysCattwo^n(X;\scrA)&\simeq(n+1)\mathbf{LMod}_{\Omega_*X}{\lp(n+1){\mathbf{Lin}_{\scrA}\mathbf{Pr}^{\mathrm{L}}}_{({\scriptstyle\infty}, n)}\rp}\\&\simeq
(n+1)\mathbf{LMod}_{n\LModtwo^{n-1}_{\Omega^{n+1}_*X}(\scrA)}{\lp(n+1){\mathbf{Lin}_{\scrA}\mathbf{Pr}^{\mathrm{L}}}_{({\scriptstyle\infty}, n)}\rp}.
\end{align*}
\end{theorem}

\begin{remark}
For simplicity, in the statement of \cref{conj:infinityn}, we assume $\scrA$ to be a presentably symmetric monoidal category. This is also the case that arises more naturally in our intended applications. We remark however that it is sufficient to assume that  $\scrA$ is a presentably \textit{$\Ebb_n$-monoidal category}.  The proof of \cref{conj:infinityn} we shall give below applies equally well to this more general setting.
\end{remark}

The proof strategy is essentially the same we followed in \cref{thm:infinitytwothm}, except for one additional technical subtlety (see the proof of \cref{conj:infinityn} below). Namely we shall deduce our $n$-categorical statement from the fact that the underlying $1$-categories are equivalent   in a way that is compatible with the enrichment.   

Since $\LinAPrLU[n]$ is cocomplete (\cite[Section $5.1$]{stefanich2020presentable}), \cref{monodromyeq} allows us to deduce  immediately  the following result.  
 \begin{lemma}
 \label{lemma:auxiliary1}
Let $X$ be a pointed connected space, and let $\scrA$ be a presentably symmetric monoidal category. Let $\LocSysCat^n(X;\scrA)$  be  the underlying category of  $(n+1)\LocSysCattwo^n(X;\scrA)$. Then there exists an equivalence of categories$$\LocSysCat^n{\lp X;\scrA\rp}\simeq\LMod_{\Omega_*X}{\lp\LinAPrLU[n]\rp}.$$
 \end{lemma}
As $\LinAPrLU[n]$ is  cocomplete   and  symmetric monoidal we have  a canonical strongly monoidal and colimit-preserving functor\begin{align}
    \label{functor:prln}
    \scrS\longrightarrow\LinAPrLU[n].
\end{align}
Recall that this functor sends a space $X$ to the colimit over the constant diagram with shape $X$ with values in the monoidal unit of $\LinAPrLU[n]$  (that is, $n\LinAPrLUtwo[n-1]$), i.e.,
$$X\mapsto\colim_X n\LinAPrLUtwo[n-1].$$
\begin{lemma}
    \label{lemma:limitsandcolimits2}
 For all $n\geqslant 1$, the  functor \eqref{functor:prln} is equivalent to the functor $$n\LocSysCattwo^{n-1}(-;\scrA)\coloneqq n{\Funtwo}_n{\lp -,\hsp n\LinAPrLUtwo[n-1]\rp}.$$ 
\end{lemma}
\begin{proof}
This is a  higher categorical generalization of \cref{prop:FunandTensor}. Assume first for ease of exposition that we are in the absolute case, where $\scrA=\scrS$ is the category of spaces. 
The core ingredient of the proof of \cref{prop:FunandTensor} was the so-called \textit{passage to the adjoints property} of $\PrLU$: a colimit over a diagram of presentable categories can be computed as the limit over the opposite diagram obtained after passing to the right adjoints. This implies the ambidexterity of limits and colimits over groupoids inside $\PrLU$ (\cref{lemma:limitsandcolimitsovergroupoids}). We establish the $n$-categorical analogue of this fact in the following Lemma.
\begin{lemma}
\label{lemma:limitsandcolimitsovergroupoids2}
Let $\scrA$ be a presentably symmetric monoidal category, let $X$ be a space, and let $F\colon X\to(n+1)\LinAPrLUtwo[n]$ be a diagram of presentable $\scrA$-linear $n$-categories of shape $X$. Then, there is a natural equivalence $\lim F\simeq\colim F$ in $(n+1)\LinAPrLUtwo[n]$.
\end{lemma}
\begin{proof}
Contrarily to the case when $n=1$, it is still unknown whether $(n+1)\PrLUtwo[n]$ admits all small limits. However, Stefanich proves that  it admits all limits of \textit{left adjointable diagrams} -- i.e., diagrams $K\to\PrLU[n]$ arising from an opposite diagram $K^{\op}\to\PrLU[n]$ by taking left adjoints. In this case, moreover, the limit over $K$ agrees with the colimit over $K^{\op}$ (\cite[Theorem $5.5.14$]{stefanich2020presentable}).

It is clear that every diagram over a small space  $X$ is left adjointable, and that it is canonically equivalent to its opposite diagram via the involution $X\simeq X^{\op}$. So, we have$$\colim_X n\PrLUtwo[n-1]\simeq\lim_X n\PrLUtwo[n-1]\simeq n\Funtwo_n{\lp X,\hsp n\PrLUtwo[n-1]\rp}.$$
\end{proof}
Thus, \cref{lemma:limitsandcolimitsovergroupoids2} implies the statement of \cref{lemma:limitsandcolimits2} in the absolute case. In turn, this implies the case for an arbitrary presentably symmetric monoidal category $\scrA$ of coefficients simply by base change.
\end{proof}
\begin{lemma}
 \label{lemma:auxiliary2}
Let $n\geqslant1$ be an integer, let $X$ be a pointed $n$-connected space and let $\scrA$ be a presentably symmetric monoidal category. Let $\Omega^{n+1}_*X$ denote the $(n+1)$-fold loop space of $X$, and let $\Omega^{n+1}_*X\otimes\boldone_{\scrA}$ denote the $\Ebb_{n+1}$-algebra in $\scrA$ obtained  by applying to $\Omega^{n+1}_*X$  the unique strongly monoidal and colimit preserving functor $\scrS\to\scrA$. Then, there is an equivalence of categories$$\LMod_{\Omega_*X}{\lp\LinAPrLU[n]\rp}\simeq\operatorname{LMod}_{n\LModtwo^{n-1}_{\Omega^{n+1}_*X\otimes\boldone_{\scrA}}}\PrLU[n].$$
\end{lemma}
\begin{proof}
 Combining \cref{lemma:auxiliary1} and \cref{lemma:limitsandcolimits2}, we obtain  equivalences of categories
\begin{align*}
\LocSysCat^n(X;\scrA)\simeq\LMod_{\Omega_*X}{\lp\LinAPrLU[n]\rp}&\simeq\operatorname{Lin}_{n\LocSysCattwo^{n-1}(\Omega_*X;\scrA)}\PrLU[n].
\end{align*} 
Note  that in the right hand side of the above chain of equivalences we can replace $n\LocSysCattwo^{n-1}(\Omega_*X;\scrA)$ with $\smash{n\LModtwo_{\Omega^2_*X}{\lp n\LinAPrLUtwo[n-1]\rp}}$. Indeed, arguing as in \cref{prop:koszulismonoidal}, we can deduce  that for any integer $n\geqslant1$ there is an equivalence 
\begin{equation*}
(n+1)\LModtwo_{\Omega_*(-)}{\lp (n+1)\LinAPrLUtwo[n]\rp}\simeq(n+1)\LocSysCattwo^n(-;\scrA)
\end{equation*}
of strongly monoidal functors from $\scrS_*^{\geqslant1}$ to $\operatorname{Lin}_{(n+1)\mathbf{Lin}_{\scrA}\mathbf{Pr}^{\operatorname{L}}_{({\scriptscriptstyle\infty},n)}}\CatVinftycocom$. 
 This implies that  $\LocSysCat^{n-1}(\Omega_*X;\scrA)$ and $\smash{\LMod_{\Omega^2_*X}{\lp\LinAPrLU[n-1]\rp}}$ are equivalent \textit{as $\Ebb_1$-monoidal categories}. Since $\Omega_*^kX$ is always $(n-k)$-connected for every $0\leqslant k \leqslant n$, we can iterate this argument and obtain an equivalence of categories between $\LocSysCat^n{\lp X;\scrA\rp}$ and the category of $\scrA$-linear presentable $n$-categories which are presentably left tensored over the presentable $n$-category of iterated left modules over $\Omega_*^nX$. Unraveling all definitions, we see that this $n$-category agrees precisely with the object $\LMod^{n-1}_{\Omega^{n}_*X}$ constructed in \cref{constr:highermodules}, hence we can conclude as desired.
\end{proof}

\begin{proof}[Proof of \cref{conj:infinityn}]
We would like to conclude as in the proof of \cref{thm:infinitytwothm} by combining the following two facts: first, the fact that the underlying 1-categories are equivalent, as proved in   \cref{lemma:auxiliary1,lemma:auxiliary2}; second, the observation that these equivalences are compatible with the enrichment over $n\PrLUtwo[n-1]$. 
There is however one subtlety that arises when $n \geqslant 3$ and that requires some extra care.  Namely, as mentioned in \cite[Remark 1.1.3]{stefanich2020presentable}, it is  not  known  whether the mapping objects in a presentable $n$-category $\scrC$ are \textit{presentable} $(n-1)$-categories. In particular, it is not  known whether the symmetric monoidal structure of $n\PrLUtwo[n-1]$ is closed.

We can fix the issue as follows. The category $\LinAPrLU[n]$ is a symmetric monoidal subcategory of the category $\mathrm{Lin}_{\scrA}\CatVinftycocom[n]$, which on the other hand is cocompletely \textit{closed} symmetric monoidal. So, up to enlarging appropriately our universe, \cref{monodromyeq} and \cref{lemma:rectificationofmodules} imply analogous equivalences
\begin{align*}
\Fun_n{\lp X,\hsp\mathrm{Lin}_{\scrA}\CatVinftycocom[n]\rp}&\simeq \LMod_{\Omega_*X}{\lp\mathrm{Lin}_{\scrA}\CatVinftycocom[n]\rp}\\&\simeq \LMod_{\Omega_*X\otimes n\LinAPrLUtwotiny[n-1]}{\lp\CatVinftycocom[n]\rp}
\end{align*}
in the more general setting of \textit{cocomplete} $\scrA$-linear $n$-categorical local systems over $X$. Now, all these objects inherit a \textit{closed} symmetric monoidal structure providing them an enrichment over themselves, hence over $\mathrm{Lin}_{\scrA}\CatVinftycocom[n]$ via a symmetric monoidal functor out of $\mathrm{Lin}_{\scrA}\CatVinftycocom[n]$: this is done exactly as in the $2$-categorical case presented in Constructions \ref{constr:enrichment1}, \ref{constr:enrichment2} and \ref{constr:enrichment3}. Then, it is straightforward to see that the above equivalences commute with the coaugmentation functors coming from $\mathrm{Lin}_{\scrA}\CatVinftycocom[n]$. In particular, they preserve the $\mathrm{Lin}_{\scrA}\CatVinftycocom[n]$-enrichment and hence can be promoted to $(n+1)$-categorical equivalences.

Since these equivalences obviously preserve objects whose underlying cocomplete $\LinAPrLU[n-1]$-module category is $\kappa_0$-compact, the $(n+1)$-categorical equivalences nicely restrict to the sub-$(n+1)$-categories appearing in the statement of \cref{conj:infinityn}. This concludes the proof that the $1$-categorical equivalence provided by \cref{lemma:auxiliary1,lemma:auxiliary2} can be enhanced to an $(n+1)$-categorical equivalence.

\end{proof}

\begin{parag}
\label{remark:nconnectedness}
Before proceeding, we would like to comment on the role played by the connectedness assumptions  in the statement of      \cref{conj:infinityn}.  In fact, it is possible to formulate a general dictionary relating 
$n$-categorical local systems and monodromy data for any space $X$. We will not do so explicitly  in this article, as the statement is  more cumbersome than  \cref{conj:infinityn}, and it is not relevant for our intended applications to Koszul duality.

However,  we shall try to clarify what is  involved in such an extension. To this effect we will prove, first,    \cref{prop:productmonoidalcategories} below. Equipped with \cref{prop:productmonoidalcategories} it is possible to state and prove a generalization of 
\cref{conj:infinityn} which applies to all spaces $X$  independently on connectedness assumptions.  
However,  formulating the correct statement is somewhat awkward. We    give an idea of these subtleties in the simple case of $X=S^1$ in \cref{lucwts}. 
\begin{proposition}
\label{prop:productmonoidalcategories}
Let $n\geqslant 1$ be a positive integer and let $n\bm{\scrA}\coloneqq\prod_{\alpha\in A}n\bm{\scrA}_{\alpha}$ be a small product of presentably symmetric monoidal $n$-categories. Then we have an equivalence of categories
\[
\mathrm{Mod}_{n\bm{\scrA}}\PrLU[n]\simeq\prod_{\alpha\in A}\mathrm{Mod}_{n\bm{\scrA}_{\alpha}}\PrLU[n].
\]
\end{proposition}
\begin{remark}
Note that in the de-categorified setting, the analogue of \cref{prop:productmonoidalcategories} is false as soon as the set of indices $A$ is not finite. This boils down to the well-known difference between the spectrum of an infinite product of commutative rings, and the infinite disjoint union of spectra of commutative rings. 
\end{remark}
\begin{proof}[Proof of \cref{prop:productmonoidalcategories}]
There is a natural functor
\[
\mathrm{Mod}_{n\bm{\scrA}}\PrLU[n]\longrightarrow\prod_{\alpha}\mathrm{Lin}_{n\bm{\scrA}_{\alpha}}\PrLU[n]
\]
induced by base change along the obvious projections $n\bm{\scrA}\to n\bm{\scrA}_{\alpha}.$ Moreover, on both sides we have forgetful functors\[
\oblv_{n\scrA}\colon\mathrm{Mod}_{n\bm{\scrA}}\PrLU[n]\longrightarrow\PrLU[n]
\]
and
\[
\prod\circ\langle\oblv_{n\bm{\scrA}_{\alpha}}\rangle_{\alpha}\colon\prod_{\alpha}\mathrm{Mod}_{n\bm{\scrA}_{\alpha}}\PrLU[n]\longrightarrow\prod_{\alpha}\PrLU[n]\longrightarrow\PrLU[n].
\]
It is not difficult to see that both forgetful functors are monadic over $\PrLU[n]$. In the first case, this is obvious; in the second case, we use the following facts together with Barr--Beck--Lurie monadicity.
\begin{itemize}
\item The functor $\prod\circ\langle\oblv_{n\bm{\scrA}_{\alpha}}\rangle_{\alpha}$ admits obviously a left adjoint, given by the assignment $n\bm{\scrC}\mapsto \left\{n\bm{\scrC}\otimes n\bm{\scrA}_{\alpha}\right\}_{\alpha\in A}.$
\item It follows from \cref{lemma:limitsandcolimitsovergroupoids2} that products are the same as coproducts inside $\PrLU[n]$, so they commute straightforwardly with \textit{all colimits}.
\item Since $\PrLU[n]$ is pointed, the operation of taking products of term-wise conservative functors is again conservative. (The proof of this fact is provided in \cite[Lemma $2.1.7$]{Pascaleff_Pavia_Sibilla_Koszul}.)
\end{itemize}
Now note that the diagram obtained by passing to the right adjoints commute straightforwardly: indeed, for every index $\alpha$ the natural $n$-functor
\[
n\bm{\scrC}\otimes n\bm{\scrA}\otimes_{n\bm{\scrA}}n\bm{\scrA}_{\alpha}\longrightarrow n\bm{\scrC}\otimes n\bm{\scrA}_{\alpha}
\]
is clearly an equivalence. So we conclude by  \cite[Corollary $4.7.3.16$]{ha}.
\end{proof}

As we mentioned, using \cref{prop:productmonoidalcategories}  it is possible 
to remove all connectedness assumptions from the statement of \cref{conj:infinityn}. This requires  fixing base points on each connected component of $X$ and of $\Omega^k_*X$ for all $k$'s ranging from $1$ to $n$; and then transferring the Day convolution monoidal structure from $\LocSysCat^{n-k+1}(\Omega^{k-1}_*X;\scrA)$ to $\prod_{\alpha}\LMod_{\Omega_*^kX_{\alpha}}(\LinAPrLU[n-k+1])$ under the equivalence of categories given by \cref{prop:productmonoidalcategories}. This is the most delicate point: as \cref{lucwts} makes apparent the statement, though not more difficult, becomes more involved.
 \begin{exmp}
\label{lucwts}
 Consider the case $n=2$ and $X=S^1$. Then $\Omega_*S^1\simeq \ZZ$ and we would like to obtain equivalences
\begin{align*}
3\LocSysCattwo^2(X;\scrA)&\simeq3\Modtwo_{\ZZ}\lp3\LinAPrLUtwo[2]\rp\simeq\prod_{n\in\ZZ}3\LinAPrLUtwo[2].
\end{align*}
Note that we have a chain of equivalences
\begin{align*}
    3\Modtwo_{\ZZ}{\lp3\LinAPrLUtwo[2]\rp}&\overset{\ref{prop:FunandTensor}}{\simeq}3\Modtwo_{2\LocSysCattwo(\ZZ,\scrA)}{\lp3\LinAPrLUtwo[2]\rp}\\&\simeq3\Modtwo_{\prod_n2\LocSysCattwo(\left\{n\right\};\scrA)}{\lp3\LinAPrLUtwo[2]\rp}\\&\simeq3\Modtwo_{\prod_n2\LinAPrLUtwo}{\lp3\LinAPrLUtwo[2]\rp}\\&\simeq\prod_{n\in\ZZ}3\mathbf{Mod}_{2\LinAPrLUtwo}\lp 3\LinAPrLUtwo[2]\rp\simeq\prod_{n\in\ZZ}3\LinAPrLUtwo[2].
\end{align*}
However, note that in the second equivalence we replaced $2\LocSysCattwo(\ZZ;\scrA)$, equipped with the Day convolution monoidal structure (which corresponds to the K\"{u}nneth-like monoidal structure on the category of graded objects of $2\LinAPrLUtwo$), with the product $\prod_n2\LinAPrLUtwo$, which is naturally endowed with a \textit{point-wise} monoidal structure.
\end{exmp}

In general, if $G$ is a topological monoid, the equivalence of categories$$\LocSysCat^n(G;\scrA)\simeq\prod_{\substack{G_{\alpha} \text{ connected}\\\text{component}}}\LocSysCat^n(G_{\alpha};\scrA)\simeq\prod_{\substack{G_{\alpha} \text{ connected}\\\text{component}}}\LMod_{\Omega_*G_{\alpha}}\lp\LinAPrLU[n]\rp$$ is not monoidal: on the right hand side we have the monoidal structure induced component-wise by the relative tensor product over the $\Ebb_2$-monoid $n\LocSysCattwo^{n-1}\lp\Omega_*G_{\alpha};\scrA\rp$, while on the left hand side we have the Day convolution monoidal structure -- which is the one we have to consider in order to obtain the correct $(n+1)$-category of presentable $n$-categories with an action of $G$. So one has to  impose the latter monoidal structure on the product of $\LocSysCat^n(G_{\alpha};\scrA)$ in order to obtain the desired generalization of \cref{conj:infinityn}. Further, for higher $n$, this issue gets compounded, as it arises  at each iterated application of the (component-wise) based loop space: unless, of course, the space $X$ is $n$-connected. This is the reason we assumed $n$-connectedness in \cref{conj:infinityn}: it allows us to bypass these issues, and   yields a much cleaner statement.  
\end{parag}

\subsection{Topological actions on \texorpdfstring{$n$}{n}-categories and higher Hochschild cohomology}
\label{sec:topactionncat}
In this Section we will generalize \cref{lemma:hochschildand2modules}, and hence Teleman's \cref{thm:teleman}, to the $n$-categorical setting. The possibility of proving such a statement was  already suggested by Teleman in \cite[Remark $2.8$]{teleman}. We stress however that our main result in this section (\cref{lemma:hochschildand2modules}) is conditional on the validity of an expected property of presentable $n$-categories which is  as yet conjectural.  Namely, as we discussed in the proof of \cref{conj:infinityn},  it is not known whether the symmetric monoidal structure on $n\PrLUtwo[n-1]$ is closed when $n \geqslant 3$. Our proof strategy depends in a crucial way on the assumption that this claim holds.  For clarity we formulate it explicitly as the following conjecture.
\begin{conj}
\label{conj:notenriched}
Let $\scrA$ be presentably symmetric monoidal category, and let $n\geqslant2$ be an integer. If  $n \bm{\scrC} $ and $n \bm{\scrD}$  are two presentably $\scrA$-linear $n$-categories, then the $\scrA$-linear $n$-category 
$$n\Funtwo^{\mathrm{L}}_{(n+1)\LinAPrLUtwotiny[n]}(n\bm{\scrC} ,\hsp n\bm{\scrD} )$$ of $\scrA$-linear colimit-preserving $n$-functors between $n \bm{\scrC} $ and $n\bm{\scrD} $ is presentable. In particular, it serves as a mapping object in $n\LinAPrLUtwo$.
\end{conj}
Note that when $n=1$ \cref{conj:notenriched} holds, as the symmetric monoidal structure on  
$\LinAPrLU$ is closed and therefore $\LinAPrLU$ is enriched over itself. It is natural to expect that this holds for all $n$, however this has not been established yet. As we discussed, the category $\CatVinftycocom$ is closed symmetric monoidal and thus $\Mod_{n\PrLUtwotiny[n-1]}{\lp\CatVinftycocom\rp}$ 
 admits morphism objects for every positive integer $n$. However, in general, these are  only categories tensored over $(n+1)\PrLUtwo[n]$. \cref{conj:notenriched} claims that the morphism object between two presentable $n$-categories is not only $(n+1)\PrLUtwo[n]$-tensored but is in fact  presentable. 

Let $\scrA$ be a presentably symmetric monoidal category, and let $n$ be a positive integer. Following  \cite{gepnerhaugseng}  
we can define  inductively the category of $\scrA$-enriched $n$-categories as
\[
\mathrm{Lin}_{\scrA}\CatVinfty[n]\coloneqq\mathrm{Lin}_{\mathrm{Lin}_{\scrA}\CatVinftytiny[n-1]}\CatVinfty.
\]
In virtue of \cite[Theorem $6.3.2$ and Corollary $6.3.11$]{gepnerhaugseng}, we have a lax monoidal functor
\[
\Omega^n\colon\lp\mathrm{Lin}_{\scrA}\CatVinfty[n]\rp_*\longrightarrow\Alg(\scrA)
\]
from the category of pointed $\scrA$-enriched $n$-categories to the category of algebras in $\scrA$, given by taking the algebra of endomorphisms of the object determined by the pointing.

\begin{defn}
\label{def:Enhochschild}
Let $\scrA$ be a presentably symmetric monoidal category and  let $n\geqslant1$ be an integer. Let $n\bm{\scrC}$ be an $n\mathbf{Lin}_{\scrA}\CatVinfty[n-1]$-enriched category (in particular, it is a $n$-category). Let
\[
n\smash{\underline{\mathbf{End}}}_{n\mathbf{Lin}_{\scrA}\CatVinftytiny[n-1]}{\lp n\bm{\scrC},\hsp n\bm{\scrC}\rp}
\]
be the $n\mathbf{Lin}_{\scrA}\CatVinfty[n-1]$-enriched category of $n\mathbf{Lin}_{\scrA}\CatVinfty[n-1]$-linear endofunctors of $n\bm{\scrC}$, seen as naturally pointed at the identity.
The \textit{$\Ebb_n$-Hochschild cohomology of $n\bm{\scrC}$} is the algebra object in $\scrA$ defined as 
$$\mathrm{HH}^{\bullet}_{\Ebb_n}(n\bm{\scrC})\coloneqq \Omega^n\lp n\smash{\underline{\mathbf{End}}}_{n\mathbf{Lin}_{\scrA}\CatVinftytiny[n-1]}{\lp n\bm{\scrC},\hsp n\bm{\scrC}\rp}\rp.$$
\end{defn}
\begin{remark}
Note that, in the case $n=1$, \cref{def:Enhochschild} agrees with \cref{def:hochschild}. Indeed, the $\Ebb_1$-Hochschild cohomology of a presentably $\scrA$-linear category $\scrC$ is just the object of $\scrA$ classifying endomorphisms of the identity functor of $\scrC$ -- i.e., it is just the ordinary Hochschild cohomology of $\scrC$. 
\end{remark}

Following  \cite[Remark $6.3.16$]{gepnerhaugseng}, it is possible to describe $\mathrm{HH}^{\bullet}_{\Ebb_n}(n\bm{\scrC})$  in fairly  concrete terms.  
The $n$-category of $n$-endofunctors of $n\bm{\scrC}$ is an $\Ebb_1$-algebra, so it is naturally pointed at the unit (i.e., the identity endofunctor $\mathrm{id}_{n\bm{\scrC}}$ of $n\bm{\scrC}$).  There is an $\scrA$-linear $(n-1)$-category of natural transformations between $\mathrm{id}_{n\bm{\scrC}}$ and itself, which is again naturally pointed at the identity. Again, the higher natural transformations between such natural transformations form an $\scrA$-linear $(n-2)$-category, which is itself pointed. Iterating this procedure, after $n$ steps we  obtain an algebra object of $\scrA$ which parametrizes the  (possibly non-invertible) natural transformations of $n$-simplices between $\mathrm{id}_{n\bm{\scrC}}$ and itself.

\begin{remark}
\label{remark:integraltransforms}
If $A$ is an $\Ebb_{n}$-algebra object inside a presentably symmetric monoidal category $\scrA$, the $\Ebb_n$-Hochschild cohomology of $A$ is defined in \cite[Definition $5.8$]{integraltransforms} as the mapping object$$\HH^{\bullet}_{\Ebb_n}(A)\coloneqq\Mapin_{\Mod^{\Ebb_n}_A(\scrA)}(A,A).$$We can also consider the $\scrA$-linear $n$-category $\mathbf{B}^nA$ defined in \cite[Corollary $6.3.11$]{gepnerhaugseng}, which is the $\scrA$-enriched $n$-category having one single object, one single morphism for all $1\leqslant k < n$, and an $\Ebb_n$-algebra of $n$-morphisms equivalent to $A$ itself. It is straightforward to see that
\[
\mathrm{HH}^{\bullet}_{\Ebb_n}(A)\simeq \mathrm{HH}^{\bullet}_{\Ebb_n}(\mathbf{B}^nA).
\]
\end{remark}
We are now ready to state our $n$-categorical generalization of   \cref{lemma:hochschildand2modules} and \cref{thm:teleman}.
\begin{proposition}
	\label{cor:en}
	Let $\scrA$ be a presentably symmetric monoidal category with monoidal unit $\boldone_{\scrA}$, and let $n\geqslant 1$ be an integer. Let $n\bm{\scrC}$ be an $\scrA$-linear presentable $n$-category, and let $G$ be an $(n-1)$-connected topological group. Let $G\text{-}\operatorname{ModStr}(n\bm{\scrC})$ denote the space of all possible left $G$-module structures on $n\bm{\scrC}$. If \cref{conj:notenriched} holds, there is an equivalence of spaces
	$$
	\Map_{\Alg_{\Ebb_n}(\scrA)}{\lp \Omega^n_*G\otimes\boldone_{\scrA} ,\hsp \mathrm{HH}^{\bullet}_{\Ebb_n}(\bm{\scrC}) \rp}
	 \simeq  
	  G\text{-}\operatorname{ModStr}(n\bm{\scrC}).
	$$
\end{proposition}

\begin{remark}
 For $n=1$, this is just \cref{thm:teleman}. Noting that $(-1)$-connected spaces are just non-empty spaces, using the notational trick of interpreting objects of $\scrA$ as ``$\scrA$-linear $0$-categories" we can make the statement of \cref{cor:en} meaningful also for $n=0$: indeed, a $G$-action on an object $M$ of a presentably symmetric monoidal category $\scrA$ is just a $\boldone_{\scrA}[G]\coloneqq (G\otimes\boldone_{\scrA}$)-module structure on $M$, which is the same as an $\Ebb_1$-algebra morphism$$\boldone_{\scrA}[G]\longrightarrow\Endin_{\scrA}(M).$$  
\end{remark}
\begin{proof}[Proof of \cref{cor:en}]
The proof of \cref{thm:teleman} works verbatim in the categorified  setting once we know that, as in the $1$-categorical case, the $\Ebb_n$-Hochschild cohomology in the presentable setting can be expressed as a right adjoint operation to taking the $n$-category of left modules. More precisely, let$$n\LModtwo_{(-)}\colon\Alg{\lp\LinAPrLU[n-1]\rp}\longrightarrow {\lp\LMod_{n\LinAPrLUtwotiny[n-1]}{\lp\CatVinftycocom\rp}\rp}_{n\LinAPrLUtwotiny[n-1]/}$$be the functor taking a presentably monoidal $\scrA$-linear $(n-1)$-category $(n-1)\scrM$ and sending it to the $n$-category of left $(n-1)\scrM$-modules $n\LModtwo_{(n-1)\scrM}{\lp n\LinAPrLUtwo[n-1]\rp}$ pointed at the unit $(n-1)\scrM$. Taking the $\kappa_0$-compact objects, this functor lands in $$\lp\LinAPrLU[n]\rp_{n\LinAPrLUtwotiny[n-1]/}\subseteq{\lp\LMod_{n\LinAPrLUtwotiny[n-1]}{\lp\CatVinftycocom\rp}\rp}_{n\LinAPrLUtwotiny[n-1]/}.$$
 Next, we need to show that the functor 
\[
\Alg\lp\LinAPrLU[n-1]\rp\longrightarrow\lp\LinAPrLU[n]\rp_{n\LinAPrLUtwotiny[n-1]/}
\]
admits a right adjoint $\Phi_n$, explicitly described by taking the endomorphisms of the object determined by the pointing from $n\LinAPrLUtwo[n-1]$. For this, we can apply the  same strategy used in the proof of \cite[Theorem $4.8.5.11$]{ha}, since the only caveat for the existence of such a right adjoint is the existence of internal morphism objects in $\LinAPrLU[n]$. This is precisely the content of \cref{conj:notenriched}.

Using \cref{conj:infinityn}, we can express the action of $G$ (which, being $(n-1)$-connected, is the based loop space of its $n$-connected classifying space $\mathbf{B}G$) on a presentably $\scrA$-enriched $n$-category $n\bm{\scrC}$ as an action of the $n$-category of iterated modules $n\LModtwo^{n-1}_{\Omega^n_*G}(\scrA)$ on $n\bm{\scrC}$. This is the same as a $\Ebb_1$-monoidal functor
\[
n\LModtwo^{n-1}_{\Omega^n_*G}(\scrA)\longrightarrow n\smash{\underline{\mathbf{End}}}_{(n+1)\LinAPrLUtwotiny[n]}{\lp n\bm{\scrC},\hsp n\bm{\scrC}\rp},
\]
which by adjunction is the same as a $\Ebb_2$-monoidal functor
\[
(n-1)\LModtwo^{n-1}_{\Omega^n_*G}(\scrA)\longrightarrow\Phi_n\lp n\smash{\underline{\mathbf{End}}}_{(n+1)\LinAPrLUtwotiny[n]}{\lp n\bm{\scrC},\hsp n\bm{\scrC}\rp}\rp.
\]
Iterating this construction, we obtain an $\Ebb_{n+1}$-algebra morphism
$$\Omega^n_*G\otimes\boldone_{\scrA}\longrightarrow\Phi_1\Phi_2\cdots\Phi_n\lp n\smash{\underline{\mathbf{End}}}_{(n+1)\LinAPrLUtwotiny[n]}{\lp n\bm{\scrC},\hsp n\bm{\scrC}\rp}\rp.$$In virtue of our description of the right adjoints $\Phi_n$, this is immediately seen to  match our \cref{def:Enhochschild} of the $\Ebb_n$-Hochschild cohomology of $n\bm{\scrC}$.
\end{proof}
\begin{remark}
Another way to interpret \cref{cor:en} is the following. In virtue of \cite[Corollary $6.3.11$]{gepnerhaugseng}, if $\scrA$ is a presentably symmetric monoidal category then the delooping functor
\[
\Omega^n\colon\lp\mathrm{Lin}_{\scrA}\CatVinfty\rp_*\longrightarrow\Alg(\scrA)
\]
is the lax monoidal right adjoint to the looping functor
\[
\mathbf{B}^n\colon\Alg(\scrA)\longrightarrow\lp\mathrm{Lin}_{\scrA}\CatVinfty\rp_*
\]
described in \cref{remark:integraltransforms}. In particular, if $A$ is an $\Ebb_{n+1}$-algebra in $\scrA$, it is immediate to conclude that a monoidal $n$-functor of $n$-categories
\[
\mathbf{B}^nA\longrightarrow n\smash{\underline{\mathbf{End}}}_{(n+1)\LinAPrLUtwotiny[n]}{\lp n\bm{\scrC},\hsp n\bm{\scrC}\rp}
\]
which equips an $\scrA$-enriched $n$-category $n\bm{\scrC}$ of a $\mathbf{B}^nA$-module structure is the same as an $\Ebb_{n+1}$-algebra morphism 
\[
A\longrightarrow\mathrm{HH}^{\bullet}_{\mathbb{E}_n}(n\bm{\scrC}).
\]
Thus, \cref{cor:en} suggests an $n$-categorical Morita equivalence: in $(n+1)\LinAPrLUtwo[n]$, modules for the $\scrA$-enriched $n$-category $\mathbf{B}^nA$ are the same as modules for the presentable $\scrA$-linear $n$-category of $(n-1)$-fold $A$-modules $n\LModtwo^{n-1}_A(\scrA)$ of \cref{def:iteratedmodules}.
\end{remark}
We conclude with an immediate consequence of \cref{cor:en}, which may be relevant in the context of higher Brauer groups and invertible objects in higher categories of $\Bbbk$-linear presentable $n$-categories. 
\begin{corollary}
\label{highertele}
Let $n\geqslant2$ be an integer, let $G$ be an $n$-connected topological group, and let $\Bbbk$ be a field. 
Assume that \cref{conj:notenriched} holds, and that the groupoid of invertible objects inside the symmetric monoidal $(\infinity,n+1)$-category $(n+1)\LinkPrLUtwo[n]$ is connected. Then we have an isomorphism of abstract groups
\[
\pi_0{\lp\lp(n+1)\LocSysCattwo^n(X;\Bbbk)^{\mathrm{inv}}\rp^{\simeq}\rp}\cong \Hom_{\mathrm{Grp}}{\lp\pi_{n+1}(X),\hsp\Bbbk^{\times}\rp}
\]
between the group of equivalence classes of invertible local systems of $\Bbbk$-linear presentable $n$-categories on $X$, and the group of multiplicative characters of $\pi_{n+1}(X)$.
\end{corollary}
\begin{proof}
Since $X$ is connected, we can follow the strategy of proof of \cref{prop:invertibleinlocsyscat} to deduce that invertible objects for the monoidal structure on $(n+1)\LocSysCattwo^n(X;\Bbbk)$ are equivalently described as invertible objects in $(n+1)\LinkPrLUtwo[n]$ together with an action of $\Omega_*X$. Since the space of invertible objects in $(n+1)\LinkPrLUtwo[n]$ is assumed to be connected, in order to characterize invertible objects inside $(n+1)\LocSysCattwo^n(X;\Bbbk)$ it is sufficient to characterize all possible actions on an object of $(n+1)\LinkPrLUtwo[n]$. But now the assumption that \cref{conj:notenriched} holds allows us to assume that \cref{cor:en} holds as well. Hence, the proof is now completely analogous to the one of \cref{prop:invertibleobjectsoflocsyscat}.
\end{proof}
\begin{remark}
When $n=1$, the group of connected components of invertible objects in $2\LinkPrLUtwo$ is the Brauer group of $\Bbbk$, and it is known to be trivial whenever $\Bbbk$ is algebraically closed. However, invertible objects in $(n+1)\LinkPrLUtwo[n]$ have not been studied yet. Still, we expect that if $\Bbbk$ is algebraically closed then the groupoid of equivalence classes of invertible objects in $(n+1)\LinkPrLUtwo[n]$ is connected for all $n>1$ as well. 
\end{remark}

%% file: symplectic-geometry.tex
 \section{Local systems of categories and symplectic geometry}
 \numberwithin{equation}{section}
\setcounter{subsection}{1}
 \label{symplectic-geometry}
Like Teleman,  our interest in the questions studied in this article stems in large part from the fact that symplectic geometry furnishes examples of local systems of categories, by applying the theory of Fukaya categories to Hamiltonian fibrations. As this is one of the main motivations underlying our work we find it worthwhile to explain this story in some detail. Let $(S,s_{0})$ be a connected based space. A \emph{Hamiltonian fibration} over $S$ is a smooth fibration of manifolds $\pi\colon X \to S$ where each fiber $X_{s} = \pi^{-1}(s)$ is equipped with a symplectic form $\omega_{s}$, and the fibration is equipped with a reduction of the structure group to $\mathrm{Ham}(X_{s_{0}},\omega_{s_{0}})$. By this definition, to a Hamiltonian fibration there is an associated classifying map
\[S \longrightarrow \mathbf{B}\mathrm{Ham}(X_{s_{0}},\omega_{s_{0}}).\]

A Hamiltonian fibration has an underlying symplectic fibration classified by a map $S \to \mathbf{B}\mathrm{Symp}(X_{s_{0}},\omega_{s_{0}}).$ When $S$ is simply-connected, the reduction of the structure group from $\mathrm{Symp}$ to $\mathrm{Ham}$ is equivalent to a choice of closed two-form $\tau \in \Omega^{2}_{\mathrm{cl}}(X)$ such that $\tau|X_{s} = \omega_{s}$ for every $s \in S$ \cite[Theorem 6.36]{mcduff-salamon}. Such a two-form $\tau$ defines a Ehresmann connection on $\pi\colon X \to S$ by taking the $\tau$-orthogonals to the fibers, and the looped classifying map
\[\Omega_*S \longrightarrow \mathrm{Ham}(X_{s_{0}},\omega_{s_{0}})\]
admits an interpretation as the holonomy of such a connection (at least if $\pi$ is proper or if the connection has appropriately tame behavior at infinity).

When $(X_{s_{0}},\omega_{s_{0}})$ is a monotone symplectic manifold, Savelyev \cite{savelyev} has constructed a map of $\infinity$-categories
\[\mathbf{B}\mathrm{Ham}(X_{s_{0}},\omega_{s_{0}}) \longrightarrow\CatVinfty\]
(the source is an $\infinity$-groupoid) whose value at the base point is $\mathrm{Fuk}(X_{s_{0}},\omega_{s_{0}})$.
From another perspective, Oh and Tanaka \cite{oh-tanaka} have constructed a topological action of $\mathrm{Ham}(X_{s},\omega_{s})$ on $\mathrm{Fuk}(X_{s},\omega_{s})$ when $(X_{s},\omega_{s})$ is a Liouville sector. From the perspective of the present paper, these results essentially state that in these situations we obtain a local system of categories.

\begin{proposition}
\label{hamfib}
  Let $\pi\colon X \to S$ be a Hamiltonian fibration, such that the fibers $(X_{s},\omega_{s})$ are either compact monotone or are Liouville sectors. Then there is an associated local system of $(\infinity,1)$-categories over $S$ whose fiber over $s \in S$ is the Fukaya category of $(X_{s},\omega_{s})$.
\end{proposition}

\begin{proof}
  In the terminology of the present paper, a local system of $(\infinity,1)$-categories over $S$ is determined by a functor $S \to \CatVinfty$, where the source is regarded as an $\infinity$-groupoid.

  In the case where the fibers are compact monotone, Savelyev's construction \cite{savelyev} produces a map \[\mathbf{B}\mathrm{Ham}(X_{s_{0}},\omega_{s_{0}}) \longrightarrow\CatVinfty.\] Composing Savelyev's map with the classifying map $S \to \mathbf{B}\mathrm{Ham}(X_{s_{0}},\omega_{s_{0}})$ yields the desired functor.

  In the case where the fibers are Liouville sectors, the construction of Oh-Tanaka \cite{oh-tanaka} produces a topological action of $\mathrm{Ham}(X_{s},\omega_{s})$ on $\mathrm{Fuk}(X_{s},\omega_{s})$. By \cref{cor:koszulpresentable}, this is equivalent to giving a local system of categories over $\mathbf{B}\mathrm{Ham}(X_{s_{0}},\omega_{s_{0}})$, which after pulling back under the classifying map $S \to \mathbf{B}\mathrm{Ham}(X_{s_{0}},\omega_{s_{0}})$ yields the desired functor.
\end{proof}

A related case is where $(X,\omega)$ carries a Hamiltonian action of a Lie group $G$. The Proposition then applies with $S = \mathbf{B}G$. By \cref{thm:teleman}, we obtain a map $$\mathrm{C}_{\bullet}(\Omega_{*}G;k) \longrightarrow \mathrm{HH}^{\bullet}(\mathrm{Fuk}(X_{s_{0}},\omega_{s_{0}})) \cong \mathrm{QH}(X_{s_{0}},\omega_{s_{0}}).$$ This result also applies when $S = \mathbf{B}\mathrm{Ham}(X_{s_{0}},\omega_{s_{0}})$. By passing to degree zero homology on the source, we obtain a map
\begin{equation*}
  \pi_{1}(\mathrm{Ham}(X_{s_{0}},\omega_{s_{0}})) \to \mathrm{QH}(X_{s_{0}},\omega_{s_{0}}),
\end{equation*}
thus recovering the celebrated Seidel homomorphism.